\pdfoutput=1
\documentclass[11pt,a4paper]{article}
\usepackage{geometry} 
\usepackage{amsmath,amsfonts,amsthm,amssymb} 
\usepackage{graphicx, float, subfig, overpic} 
\usepackage{enumitem} 
\usepackage{stmaryrd} 
\usepackage{faktor} 
\usepackage{mathrsfs,mathtools} 
\usepackage{titlesec} 
\usepackage[none]{hyphenat} 
\usepackage[font=small,labelfont=bf,tableposition=top]{caption} 
\usepackage{authblk} 

\usepackage[dvipsnames]{xcolor}  
\usepackage{hyperref} 
\hypersetup{
	colorlinks=true,
	linkcolor=blue,
	citecolor=green,
}

\usepackage[style=trad-alpha, doi=false, isbn=false, url=false, eprint=false]{biblatex}
\graphicspath{{Figures/}}
\numberwithin{equation}{section}
\newcommand{\st}{\h:\h}

\newcommand{\h}{\hspace{1mm}}
\newcommand{\hh}{\hspace{5mm}}

\newcommand{\efr}[2]{^{\frac{#1}{#2}}}
\newcommand{\mefr}[2]{^{-\frac{#1}{#2}}}

\theoremstyle{plain}
\newtheorem{lemma}{Lemma}[section]
\newtheorem{corollary}[lemma]{Corollary}
\newtheorem{proposition}[lemma]{Proposition}
\newtheorem{theorem}[lemma]{Theorem}
\newtheorem{remark}[lemma]{Remark}

\newtheorem{definition}[lemma]{Definition}

\def\be{\begin{eqnarray}}
\def\ee{\end{eqnarray}}
\def\beal{\begin{aligned}}
\def\enal{\end{aligned}}

\newcommand{\norm}[1]{\left\lVert#1\right\rVert}

\newcommand{\vabs}[1]{\left| #1 \right|}
\newcommand{\paren}[1]{\left(#1\right)}
\newcommand{\claus}[1]{\left\{#1\right\}}
\newcommand{\boxClaus}[1]{\left[#1\right]}

\newcommand{\ol}[1]{\overline{#1}}
\newcommand{\conj}[1]{\overline{#1}}
\newcommand{\tl}{\tilde}
\newcommand{\wt}{\widetilde}

\renewcommand{\Re}{\mathrm{Re\, }}
\renewcommand{\Im}{\mathrm{Im\,}}
\renewcommand{\arg}{\mathrm{arg\,}}
\renewcommand{\inf}{\mathrm{inf\,}}

\newcommand{\reals}{\mathbb{R}}

\newcommand{\complexs}{\mathbb{C}}
\newcommand{\torus}{\mathbb{T}}

\newcommand{\integers}{\mathbb{Z}}

\newcommand{\al}{\alpha}

\newcommand{\de}{\delta}

\newcommand{\e}{\varepsilon}

\newcommand{\la}{\lambda}
\newcommand{\La}{\Lambda}
\newcommand{\g}{\gamma}

\newcommand{\s}{\sigma}

\newcommand{\tht}{\theta}

\newcommand{\qd}{\dot{q}}

\newcommand{\AAA}{\mathcal{A}}
\newcommand{\BB}{\mathcal{B}}
\newcommand{\CC}{\mathcal{C}}

\newcommand{\FF}{\mathcal{F}}
\newcommand{\GG}{\mathcal{G}}
\newcommand{\HH}{\mathcal{H}}
\newcommand{\II}{\mathcal{I}}
\newcommand{\JJ}{\mathcal{J}}
\newcommand{\KK}{\mathcal{K}}
\newcommand{\LL}{\mathcal{L}}

\newcommand{\OO}{\mathcal{O}}
\newcommand{\PP}{\mathcal{P}}

\newcommand{\RRR}{\mathcal{R}}

\newcommand{\XX}{\mathcal{X}}

\definecolor{myGreen}{RGB}{0, 200, 0}
\definecolor{myOrange}{RGB}{255, 100, 0}
\definecolor{myYellow}{RGB}{255, 200, 0}
\definecolor{myBlue}{RGB}{0, 200, 255}
\definecolor{myPurple}{RGB}{200, 0, 200}

\newcommand{\Res}{\mathrm{Res\,}}

\newcommand{\unstable}{{\mathrm{u}}}
\newcommand{\stable}{{\mathrm{s}}}

\newcommand{\HInicial}{h}

\newcommand{\pol}{\mathrm{pol}}

\newcommand{\Poi}{\mathrm{Poi}}
\newcommand{\sca}{\mathrm{sc}}
\newcommand{\equi}{\mathrm{eq}}
\newcommand{\osc}{\mathrm{osc}}
\newcommand{\pend}{\mathrm{p}}

\newcommand{\xh}{x}
\newcommand{\xp}{x}
\newcommand{\yh}{y}
\newcommand{\yp}{y}
\newcommand{\Ltres}{\mathfrak{L}}
\newcommand{\LtresLa}{\mathfrak{L}_{\Lambda}}
\newcommand{\Ltresx}{\mathfrak{L}_x}
\newcommand{\Ltresy}{\mathfrak{L}_y}

\newcommand{\fh}{\hat{f}}

\newcommand{\proj}{\mathtt{p}}
\newcommand{\ini}{0}
\newcommand{\fin}{\mathrm{end}}
\newcommand{\argf}{\mathtt{g}}

\newcommand{\inn}{\mathrm{in}}
\newcommand{\Inn}{\mathrm{in}}
\newcommand{\Inner}{{\mathrm{in}}}
\newcommand{\rhoInn}{\kappa}
\newcommand{\CInn}{\Theta}
\newcommand{\wCInn}{\widetilde{\Theta}}

\newcommand{\DuInn}{\mathcal{D}^{\mathrm{u}}_{\rhoInn}}
\newcommand{\DuInnR}[1]{\mathcal{D}^{\mathrm{u}}_{#1}}
\newcommand{\DsInn}{\mathcal{D}^{\mathrm{s}}_{\rhoInn}}

\newcommand{\DdInn}{\mathcal{D}^{\diamond}_{\rhoInn}}

\newcommand{\EInn}{\mathcal{E}_{\rhoInn}}

\newcommand{\XcalInn}{\mathcal{X}}

\newcommand{\YcalInn}{\mathcal{Y}}
\newcommand{\ZcalInn}{\mathcal{Z}}
\newcommand{\XcalInnTotal}{\mathcal{X}_{\times}}

\newcommand{\normInn}[1]{\left\lVert#1\right\rVert}
\newcommand{\normInnSmall}[1]{\lVert#1\rVert}

\newcommand{\normInnTotal}[1]{\left\lVert#1\right\rVert_{\times}}
\newcommand{\normInnTotalSmall}[1]{\lVert#1\rVert_{\times}}

\newcommand{\normInnDiff}[1]{\left\lVert#1\right\rVert}
\newcommand{\normInnDiffSmall}[1]{\lVert#1\lVert}
\newcommand{\normInnDiffExp}[1]{\left\llbracket#1\right\rrbracket}
\newcommand{\ZInn}{Z_0}
\newcommand{\ZInnB}{\widetilde{Z}_0}

\newcommand{\ZuInn}{Z^{\mathrm{u}}_0}

\newcommand{\ZsInn}{Z^{\mathrm{s}}_0}

\newcommand{\ZdInn}{Z^{\diamond}_0}

\newcommand{\WInn}{W_0}
\newcommand{\WInnB}{\widetilde{W}_0}

\newcommand{\WdInn}{W^{\diamond}_0}

\newcommand{\XInn}{X_0}
\newcommand{\XInnB}{\widetilde{X}_0}

\newcommand{\XdInn}{X^{\diamond}_0}

\newcommand{\YInn}{Y_0}
\newcommand{\YInnB}{\widetilde{Y}_0}

\newcommand{\YuInn}{Y^{\mathrm{u}}_0}

\newcommand{\YsInn}{Y^{\mathrm{s}}_0}

\newcommand{\YdInn}{Y^{\diamond}_0}

\newcommand{\DZInn}{\Delta Z_0}
\newcommand{\DWInn}{\Delta W_0}
\newcommand{\DXInn}{\Delta X_0}
\newcommand{\DYInn}{\Delta Y_0}

\newcommand{\DZo}{\Delta Z_{\mathrm{init}}}

\newcommand{\DYo}{\Delta Y_{\mathrm{init}}}

\newcommand{\out}{\mathrm{sep}}

\newcommand{\zuOut}{z^{\mathrm{u}}}

\newcommand{\zsOut}{z^{\mathrm{s}}}
\newcommand{\zusOut}{z^{\mathrm{u,s}}}
\newcommand{\zdOut}{z^{\diamond}}

\newcommand{\wdOut}{w^{\diamond}}

\newcommand{\xdOut}{x^{\diamond}}

\newcommand{\ydOut}{y^{\diamond}}

\newcommand{\dzOut}{\Delta z}

\newcommand{\cttInnDerA}{c_1}
\newcommand{\cttInnDerAA}{\gamma_1}
\newcommand{\cttInnDerB}{c_2}
\newcommand{\cttInnDerBB}{\gamma_2}
\newcommand{\cttInnDerC}{b_0}
\newcommand{\cttInnExist}{b_1}
\newcommand{\cttInnDiff}{b_2}
\newcommand{\cttInnExistA}{b_3}
\newcommand{\cttInnExistB}{b_4}
\newcommand{\cttInnExistC}{b_5}
\newcommand{\cttInnDiffA}{b_6}
\newcommand{\cttInnDiffB}{b_7}

\title{Breakdown of homoclinic orbits to $L_3$ in the RPC3BP (I). Complex singularities and the inner equation} 
\author[1,2]{Inmaculada Baldom\'a}
\author[1]{Mar Giralt
	\thanks{Corresponding author.\\
\emph{E-mail adresses:} \href{mailto:immaculada.baldoma@upc.edu}{immaculada.baldoma@upc.edu} (I. Baldom\'a), 
\href{mailto:mar.giralt@upc.edu}{mar.giralt@upc.edu} (M. Giralt),
\href{mailto:marcel.guardia@upc.edu}{marcel.guardia@upc.edu} (M. Guardia).}}
\author[1,2]{Marcel Guardia}
\affil[1]{Departament de Matem\`atiques \& IMTECH, Universitat Polit\`ecnica de Catalunya, Diagonal 647, 08028 Barcelona, Spain}
\affil[2]{Centre de Recerca Matem\`atiques, Campus de Bellaterra, Edifici C, 08193 Barcelona, Spain}
\date{July 21, 2021}

\begin{document}

\maketitle 

\begin{abstract}
 
The Restricted $3$-Body Problem models the motion of a body of negligible mass under the gravitational influence of two massive bodies, called the primaries. If the primaries perform circular motions and the massless body is coplanar with them, one has the Restricted Planar Circular $3$-Body Problem (RPC$3$BP). In synodic coordinates, it is a two degrees of freedom Hamiltonian system with five critical points, $L_1,..,L_5$, called the Lagrange points.

The Lagrange point $L_3$ is a saddle-center critical point which  is collinear with the primaries and is located beyond the largest of the two.
In this paper and its sequel \cite{articleOuter},  we provide an asymptotic formula for the distance between the one dimensional stable and unstable invariant manifolds of $L_3$ when the ratio between the masses of the primaries $\mu$ is small. It implies that $L_3$ cannot have one-round homoclinic orbits.

If the mass ratio $\mu$ is small,  the hyperbolic eigenvalues are weaker than the elliptic ones by factor of order $\sqrt{\mu}$. This implies that the distance between the invariant  manifolds is exponentially small with respect to $\mu$ and, therefore, the classical  Poincar\'e--Melnikov method cannot be applied.

In this first paper, we approximate the RPC$3$BP by an averaged integrable Hamiltonian system which possesses a saddle center with a homoclinic orbit and we  analyze the complex singularities of its time parameterization. We also derive and study the inner equation associated to the original perturbed problem. 
The difference between certain solutions of the inner equation gives the leading term of the distance between the stable and unstable manifolds of $L_3$.

In the sequel \cite{articleOuter} we complete the proof of the asymptotic formula for the distance between the invariant manifolds.
\end{abstract}

\tableofcontents


\section{Introduction} 
\label{section:introduction}
The understanding of the motions of the $3$-Body Problem has been of deep interest  in the last centuries. Since Poincar\'e, see \cite{Poincare1890}, one of the fundamental problems is to understand how the invariant manifolds of its different invariant objects (periodic orbits, invariant tori) structure its global dynamics.
%
%
Assume that one of the bodies (say the third) has mass zero. Then, one has the Restricted  $3$-Body Problem. In this model, the two first bodies, called the primaries, are not influenced by the massless one. As a result, their motions are governed by the classical Kepler laws. If one further assumes that the primaries perform circular motion and that the third body is coplanar with them, one has the Restricted Planar Circular 3-Body Problem (RPC$3$BP).
%
%
%
%

Let us name the two primaries $S$ (star) and $P$ (planet).
Normalizing their masses, we can assume that $m_S=1-\mu$ and $m_P=\mu$, with $\mu \in \left( 0, \frac{1}{2} \right]$. 
Since the primaries follow circular orbits, in rotating (usually also called synodic) coordinates, their positions can be fixed at $q_S=(\mu,0)$ and  $q_P=(\mu-1, 0)$.
Then, denoting by $(q,p) \in \reals^2 \times \reals^2$ the position and momenta of the third body and  taking appropriate units, the RPC$3$BP is a $2$-degrees of freedom Hamiltonian system with respect to
\begin{equation}\label{def:hamiltonianInitialNotSplit} 
	\begin{split}
		\HInicial(q,p;\mu) &= \frac{||p||^2}{2} 
		- q^t \left( \begin{matrix} 0 & 1 \\ -1 & 0 \end{matrix} \right) p 
		-\frac{(1-\mu)}{||q-(\mu,0)||} 
		- \frac{\mu}{||q-(\mu-1,0)||}.
	\end{split}
\end{equation}

For $\mu>0$, it is a well known fact that $\HInicial$ has five equilibrium points: $L_1$, $L_2$, $L_3$, $L_4$ and $L_5$, called Lagrange points\footnote{
	For $\mu=0$, the system has a circle of critical points $(q,p)$ with $\norm{q}=1$ and $p=(p_1,p_2)=(-q_2,q_1)$.
}
(see Figure~\ref{fig:L3perturbed}(a)). 
On an inertial (non-rotating) system of coordinates, the Lagrange points correspond to periodic orbits with the same period as the two primaries, i.e they lie on a 1:1 mean motion resonance. 
%
%
The three collinear points with the primaries, $L_1$, $L_2$ and $L_3$, are of center-saddle type and, for small $\mu$, the triangular ones, $L_4$ and $L_5$, are of center-center type (see for instance \cite{Szebehely}). 

%

Since the points $L_1$ and $L_2$ are rather close to the small primary, their invariant manifolds  have been widely studied 
 for their interest in astrodynamics applications, (see \cite{KLMR00, GLMS01v1, CGMM04}).
The dynamics around the points $L_4$ and $L_5$ have also been considerably studied since, due to its stability, it is common to find objects orbiting around these points (for instance the Trojan and Greek Asteroids associated to the pair Sun-Jupiter, see \cite{GDFGS89,CelGio90, RobGab06}).

On the contrary, the invariant manifolds of the Lagrange point $L_3$ have received somewhat less attention. Still, they  structure the dynamics in regions of the phase space of the RPC3BP. In particular, the horseshoe-shapped  orbits that explain the orbits of Saturn satellites Janus and Epimetheus lie ``close'' to the invariant manifolds of $L_3$ (see \cite{NPR20}). Moreover, the invariant manifolds of  $L_3$ (more precisely its center-stable and center-unstable manifolds) act as effective boundaries of the stability domains around $L_{4,5}$ (see in \cite{SSST13}). See the companion paper \cite{articleOuter} for more references about the dynamics of the RPC3BP in a neighborhood of $L_3$ and its invariant manifolds.
%

The purpose of this paper and its sequel \cite{articleOuter} is to study the invariant manifolds of $L_3$ and, particularly, show that they do not intersect  for $0< \mu\ll 1$ (at their first round).

\begin{figure}
	\subfloat[$\mu>0$]{%
		\begin{overpic}[scale=0.7]{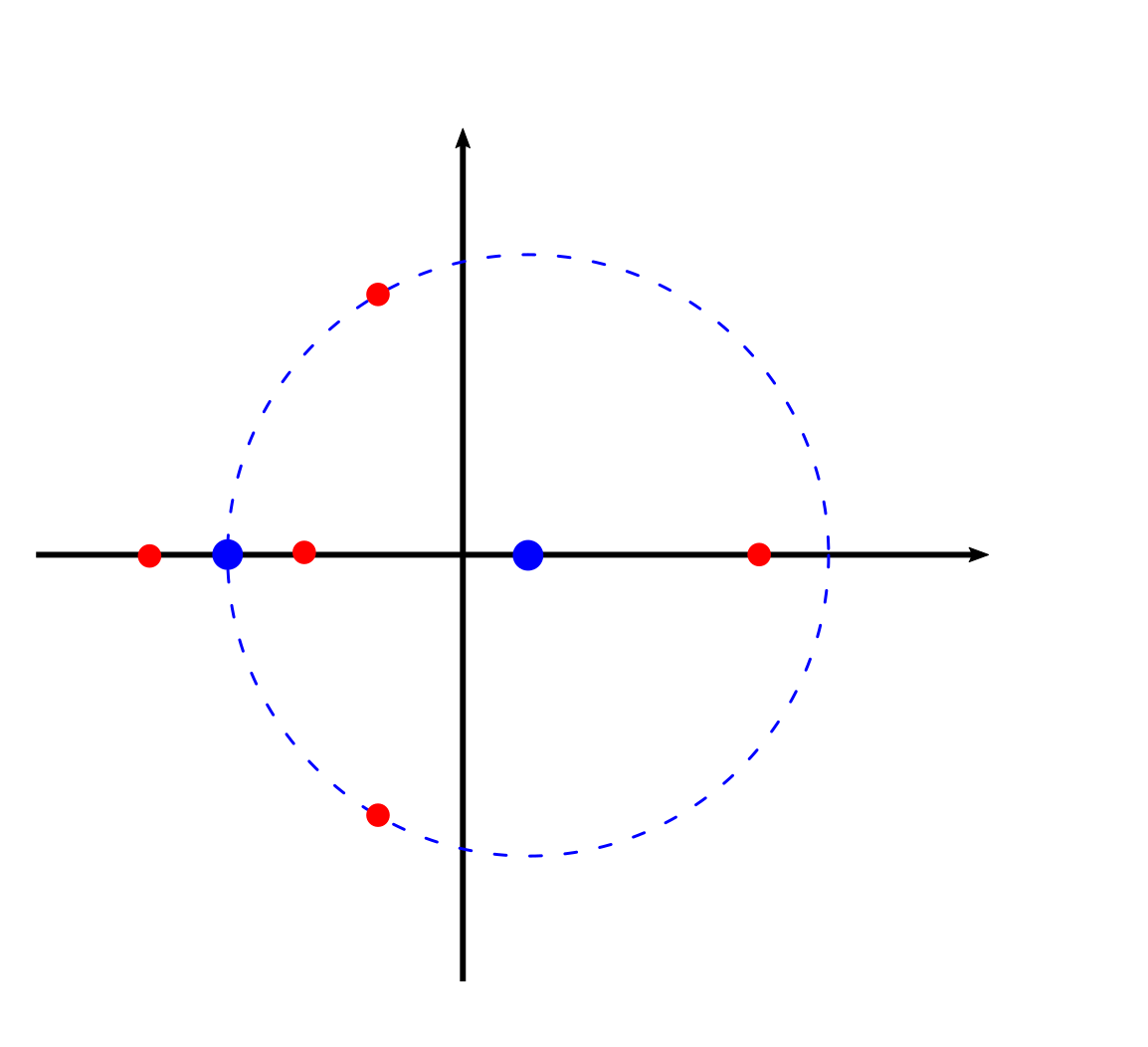}
			\put(43,38){{\color{blue} $q_S$ }}
			\put(14,38){{\color{blue} $q_P$ }}
			\put(23,45){{\color{red} $L_1$ }}
			\put(10,45){{\color{red} $L_2$ }}
			\put(63.5,45){{\color{red} $L_3$ }}
			\put(27,69){{\color{red} $L_5$ }}	
			\put(27,14){{\color{red} $L_4$ }}
			\put(89,42){$q_1$}
			\put(37.5,82.5){$q_2$}
		\end{overpic}
		\vspace{1cm}	
	}\hfill
	\subfloat[$\mu=0.003$]{%
		\begin{overpic}[width=7.5cm]{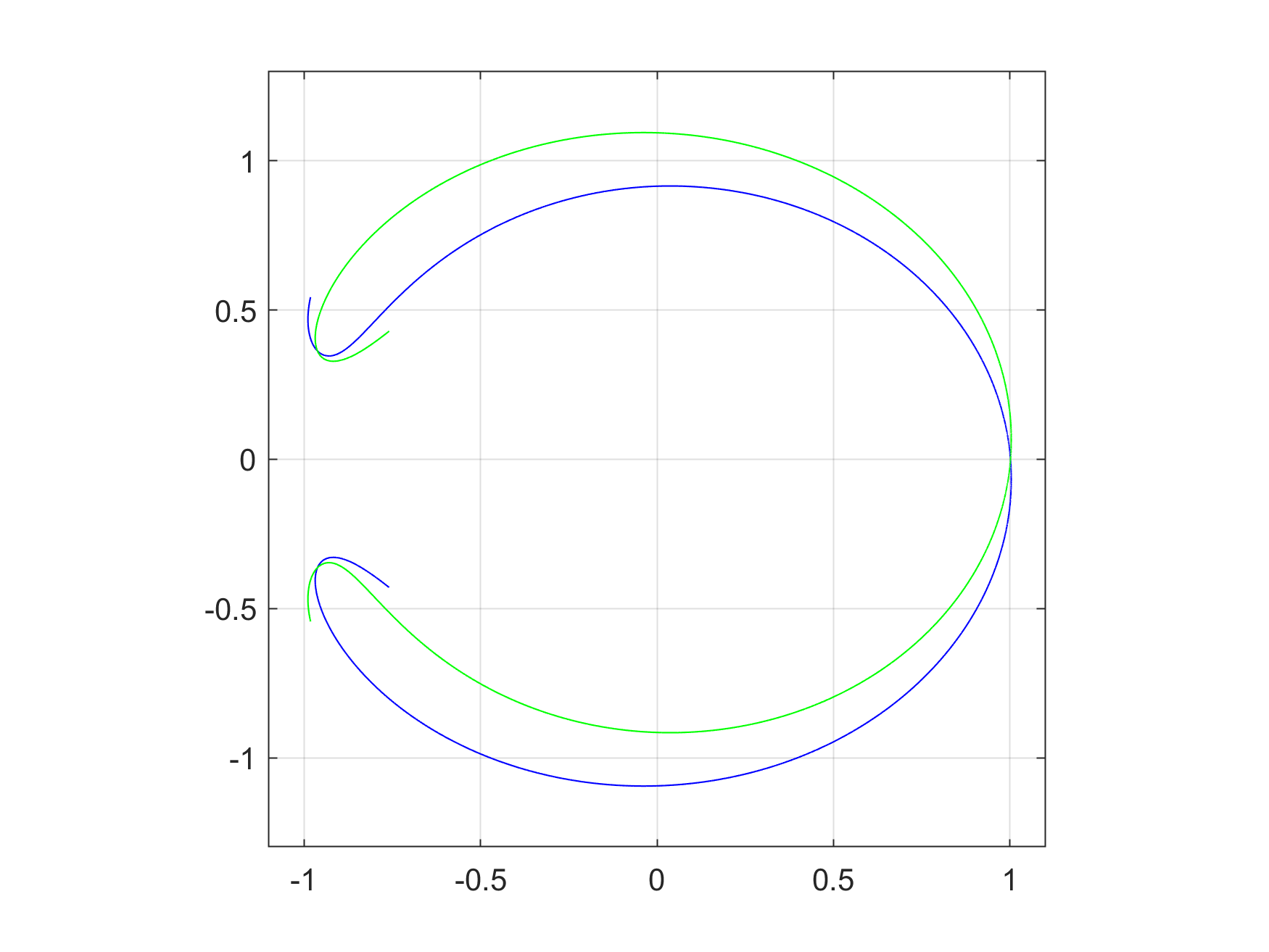}
			\put(50,1){$q_1$}
			\put(12,38){$q_2$}
			\put(71.5,37.5){{\color{red} $ L_3 \h \bullet$ }}
			\put(37,58){{\color{red} $\bullet$ }}
			\put(28,61){{\color{red} $L_5$ }}
			\put(37,17){{\color{red} $\bullet$ }}	
			\put(28,12){{\color{red} $L_4$ }}
			\put(50.5,37.5){{\color{blue} $\bullet$ }}
			\put(50.5,34){{\color{blue} $q_S$ }}
			\put(23,37.5){{\color{blue} $\bullet$ }}
			\put(23,34){{\color{blue} $q_P$ }}
		\end{overpic}
	}
	\caption{(a) Projection onto the $q$-plane of the equilibrium points for the RPC$3$BP on rotating coordinates. (b) Plot of the stable (green) and unstable (blue) manifolds of $L_3$ projected onto the $q$-plane for $\mu=0.003$. }
	\label{fig:L3perturbed}
\end{figure}
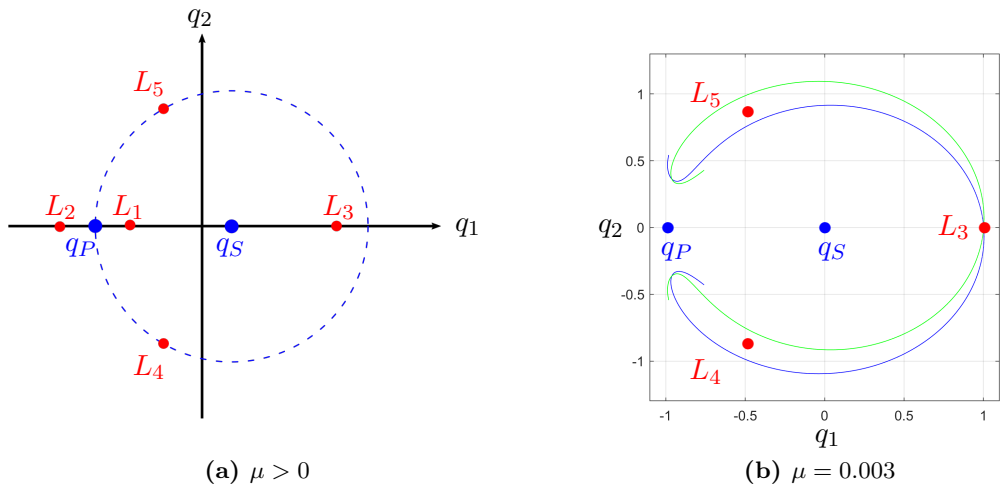

\subsection{The unstable and stable invariant manifolds of \texorpdfstring{$L_3$}{L3}}

The eigenvalues of the  the Lagrange point $L_3$
%
satisfy that
\begin{equation}\label{eq:specRot}
	\mathrm{Spec}  = \claus{\pm \sqrt{\mu} \, \rho(\mu), \pm i  \, \omega(\mu) },
	\quad
	\text{with} \quad \left\{
	\begin{array}{l}
		\rho(\mu)=\sqrt{\frac{21}{8}}  + \OO(\mu),\\[0.4em]
		\omega(\mu)=1 + \frac{7}{8}\mu + \OO(\mu^2),
	\end{array}\right.
\end{equation}
as $\mu \to 0$ (see \cite{Szebehely}). Notice that, due to the different size in the eigenvalues, the system posseses two time scales, which translates to rapidly rotating dynamics coupled with a slow hyperbolic behavior around the critical point $L_3$.

The one dimensional unstable and stable invariant manifolds have two branches each (see Figure~\ref{fig:L3perturbed}(b)).
One pair, which we denote  by $W^{\unstable,+}(\mu)$ and $W^{\stable,+}(\mu)$ circumvents $L_5$ whereas the other, denoted as  $W^{\unstable,-}(\mu)$ and $W^{\stable,-}(\mu)$,  circumvents $L_4$.
%
%
Since the Hamiltonian system associated to  $\HInicial$ in \eqref{def:hamiltonianInitialNotSplit} is reversible with respect to the involution
\begin{equation}
	\label{def:involutionCartesians}
	\Phi(q,p;t)=(q_1,-q_2,-p_1,p_2;-t),
\end{equation}
the $+$ branches are symmetric to the $-$ ones. Thus, one can restrict the study to the first ones.

As already mentioned, the aim of this paper (and its sequel \cite{articleOuter}) is to give an asymptotic formula for the distance between $W^{\unstable,+}(\mu)$ and $W^{\stable,+}(\mu)$, for $0<\mu\ll 1$ (in an appropriate transverse section). 
However, due to the rapidly rotating dynamics of the system (see \eqref{eq:specRot}), the stable and unstable manifolds of $L_3$ are exponentially close to each other with respect to $\sqrt{\mu}$. This implies that the classical Melnikov Theory \cite{Melnikov1963} cannot be applied and that obtaining the asymptotic formula is a rather involved problem.

The precise statement for the asymptotic formula for the distance is properly stated in the  sequel of this paper \cite{articleOuter}. Let us give here a more informal statement.
Consider the classical symplectic polar coordinates $(r,\theta,R, G)$, where $R$ is the radial momentum and $G$ the angular momentum, and  the section $\Sigma=\{\theta=\pi/2, r>1\}$. Then, if one denotes by $P^\unstable$ and $P^\stable$ the first intersections of the invariant manifolds $W^{\unstable,+}(\mu)$,$W^{\stable,+}(\mu)$ with $\Sigma$,  the distance between these points, is 
given by
\begin{equation}\label{eq:dist}
	\mathrm{dist}_{\Sigma}(P^\unstable,P^\stable) = \sqrt[3]{4} \, \mu^{\frac13}
	e^{-\frac{A}{\sqrt\mu}} \boxClaus{\vabs{\CInn}
			+ \OO\paren{\frac1{\vabs{\log\mu}}}},
\end{equation}
for $0<\mu\ll1$ and certain constants $A>0$ and $\CInn \in \complexs$.
%
%
The proof of this asymptotic formula spans both the present paper and the sequel \cite{articleOuter}. In this first paper, we perform the first steps towards the proof (see Section \ref{sec:schemeproof} below). In particular, we obtain and describe the constants $A$ and $\CInn$ appearing in \eqref{eq:dist}.  
%


A fundamental problem in dynamical systems is to prove that a model has chaotic dynamics (for instance a Smale Horseshoe). 
For many physically relevant problems, like those in Celestial Mechanics, this is usually a remarkably difficult problem.
%
%
%
%
Certainly, the fact that the invariant manifolds of $L_3$ do not coincide does not lead to chaotic dynamics. However, one should expect the existence of Lyapunov periodic orbits which are exponentially close (with respect to $\sqrt{\mu}$) to $L_3$ and whose stable and unstable invariant manifolds intersect transversally. If so, the Smale-Birkhoff Theorem 
would imply the existence of a hyperbolic set whose dynamics is conjugated to that of the Bernoulli shift (in particular, with positive topological entropy) exponentially close to the invariant manifolds of $L_3$. 
%
%

\subsection{Exponentially small splitting of separatrices}
\label{subsection:introSplittin}

Even though there is a standard theory to analyze the breakdown of homoclinic and heteroclinic connections, the so called Poincar\'e-Melnikov Theory (see \cite{Melnikov1963} and \cite{GuckenheimerHolmes} for a more
modern exposition),
it cannot be applied to obtain \eqref{eq:dist} due to its exponential smallness.
%
%
%
Indeed, the breakdown of homoclinic connections to $L_3$ fits into what is usually  referred to as exponentially small splitting of separatrices problems. 
This \emph{beyond all orders} phenomenon was first detected by Poincar\'e in~\cite{Poincare1890} when he  studied the non integrability of the $3$-Body Problem. 

The first obtention of an asymptotic formula for an exponentially small splitting of separatrices did not appear until the 1980's with the pioneering work  by Lazutkin for the \emph{standard map}~\cite{Laz84, Laz05}.
%
%
Even if Lazutkin's work was not complete (the complete proof was achieved by
Gelfreich in~\cite{Gel99}), the ideas he developed have been very influential and  have been the basis of many of the works in the field (and in particular of this paper and the sequel \cite{articleOuter}). Other methods to deal with exponentially small splitting of separatrices are Treschev's \emph{continuous averaging}  (see \cite{Tre97}) or ``direct'' series methods (see \cite{GGM99}).

Usually, the exponentially small splitting of separatrices problems are classified as \emph{regular} or \emph{singular}.

In the regular cases, even if  Melnikov theory cannot be straightforwardly applied, the Melnikov function  gives the leading term for the distance between the perturbed invariant manifolds. That is,   Melnikov theory provides the  first order for the distance but leads to too crude estimates for the higher order terms.
%
%
This phenomenon has been studied in rapidly forced periodic or quasi-periodic perturbations of $1$-degree of freedom Hamiltonian systems, see 
\cite{HMS88, DelSea92, DelSea97, Gel97a, Gel00, BF04, BF05, GGS20}, 
in close to the identity area preserving  maps, see \cite{DelRRR98}, and in Hamiltonian systems with two or more degrees of freedom which have hyperbolic tori with  fast quasiperiodic dynamics \cite{Sau01}.
In particular, \cite{GMS16} provides  the first  prove of exponentially small splitting of separatrices in a Celestial Mechanics problem (see also \cite{GMSS17,GPSV21}).

In the singular cases, the exponentially small   first order for the distance between the invariant manifolds is no longer given by the Melnikov function. Instead, one has to consider an auxiliary equation, usually called  \emph{inner equation}, which does not depend on the perturbative parameter and  provides the first order for the distance. Some results on inner equations can be found on \cite{Gel97b, GS01, OSS03, Bal06, BalSea08, BM12} and
the application of  the inner equation analysis to the original problem can be  found in  \cite{BFGS12, Lom00, GOS10, GaiGel10, MSS11, BCS13,BCS18a, BCS18b}.
%

Due to the extreme sensitivity of the exponentially small splitting of separatrices phenomenon on the features of each particular model, most of the available results apply  under quite restrictive hypothesis and,  therefore, cannot be applied to analyze the invariant manifolds of $L_3$.

\subsection{Strategy to obtain an asymptotic formula for the breakdown of the invariant manifolds of \texorpdfstring{$L_3$}{L3}}\label{sec:schemeproof}

For the limit problem $h$ in \eqref{def:hamiltonianInitialNotSplit} with $\mu=0$, the five Lagrange point ``disappear'' into the circle of (degenerate) critical points  $\norm{q}=1$ and $p=(p_1,p_2)=(-q_2,q_1)$. As a consequence, the one-dimensional invariant manifolds of $L_3$ disappear when $\mu=0$ too. 
For this reason, to analyze perturbatively these invariant manifolds, the first step is to perform a singular change of coordinates to obtain a ``new first order'' Hamiltonian which  has a center saddle equilibrium point (close to $L_3$) with stable and unstable manifolds that coincide along a separatrix.
%
To perform the change of coordinates  we use the Poincar\'e planar elements (see~\cite{MeyerHallOffin}) plus a singular (with respect to $\mu$) scaling.
%
%


In the following list, we present the main steps of our strategy to prove  formula \eqref{eq:dist}.
%
%
We split the list in two. First we explain the steps performed in this paper and later those carried out in its sequel \cite{articleOuter}.

In this paper, we complete the following steps: 
\begin{enumerate}[label*=\Alph*.]
	\item We perform a change of coordinates which captures the slow-fast dynamics of the system.
	The new Hamiltonian becomes a (fast) oscillator weakly coupled to a $1$-degree of freedom Hamiltonian with a saddle point and a separatrix associated to it.
	\item We analyze the analytical continuation of a time-parametrization of the separatrix.
	In particular, we obtain its maximal strip of analyticity (centered at the real line).
	We also describe the character and location of the singularities  at the boundaries of this region.
	%
	
	%
	%
	
	\item We derive the inner equation, which gives the first order of the original system close to the singularities of the separatrix described in Step B.
	This equation is independent of the perturbative parameter $\mu$.
	
	\item We study two special solutions of the inner equation which are approximations of the perturbed invariant manifolds near the singularities.
	Moreover, we provide an asymptotic formula for the difference between these two solutions of the inner equation.
	We follow the approach presented in~\cite{BalSea08}.
\end{enumerate}

%

In \cite{articleOuter} we complete the following steps:

\begin{enumerate}
\item[E.] We prove the existence of the analytic continuation of suitable parametrizations of $W^{\unstable,+}(\de)$ and $W^{\stable,+}(\de)$ in appropriate complex domains (and as graphs).
These domains contain a segment of
the real line and intersect a neighborhood sufficiently close to the singularities of the  separatrix.
\item[F.] By using complex matching techniques, we compare the solutions of the inner equation with the graph parametrizations of the perturbed invariant manifolds.
\item[G.] Finally, we prove that the dominant term of the difference between manifolds is given by the term obtained from the difference of the solutions of the inner equation.
\end{enumerate}

The structure of this paper goes as follows. 
In Section \ref{section:mainResults}, we present the main results for the Steps A to D and introduce some heuristics to contextualize them.
Sections \ref{section:proofA-singularities}-\ref{section:proofD-AnalysisInner} are devoted to the proof of the results in Section~\ref{section:mainResults}.

\paragraph{The constants in the asymptotic formula for the distance.}

The constant $A$ in \eqref{eq:dist} is given by the height of the maximal strip of analyticity of the unperturbed separatrix (see Step B).
Therefore, to obtain its value, one has to compute the imaginary part of the singularities of the separatrix which are closer to the real line.

On all the previous mentioned works on splitting of separatrices, either the separatrix of the unperturbed model has an analytic expression (see for example~\cite{DelSea92,GaiGel10,GOS10,BCS13,GMS16}) or otherwise certain properties of its analytic continuation are given as assumptions (see~\cite{DelSea97,Gel97b,BF04, BFGS12}).
%
%
In this case, we do not have an explicit expression for the time-parameterization of the separatrix and, to obtain its complex singularities, we need to rely on techniques of analytical continuation to analyze them (see Section \ref{subsection:singularities}). In particular, we describe the parametrization of the separatrix in terms of a multivalued function involving a complex integral and (see Theorem \ref{theorem:singularities} below) we obtain 
%
%
\begin{equation*}
	A= \int_0^{\frac{\sqrt{2}-1}{2}} \frac{2}{1-x}\sqrt\frac{x}{3(x+1)(1-4x-4x^2)}  dx\approx 0.177744.
\end{equation*}	
This value agrees with the numerical computations of the distance between the invariant manifolds given in \cite{FontPhD, SSST13}.

Since we are in a singular case, the constant $\CInn$ in \eqref{eq:dist} is not correctly given by the Melnikov function but by the analysis of the  inner equation of the system (see Step C above and also  Sections \ref{subsection:innerDerivation} and \ref{subsection:innerComputations}).
In particular, $\CInn$ corresponds to a Stokes constant 
and does not have a closed formula.
%
%
By a numerical computation,  we see that $\vabs{\CInn}\approx 1.63$ (see Remark \ref{remark:CInnNumerics}). 
%
%
We expect that, by means of a computer assisted proof, it would be possible to obtain rigorous estimates and verify that $\vabs{\CInn}\neq 0$, see \cite{BCGS21}.

\section{Main results} 
\label{section:mainResults}

We devote this section to state the main results concerning the Steps A, B, C and D explained in Section \ref{sec:schemeproof}. 
First, in Section~\ref{subsection:reformulation}, we present the changes of coordinates involved  to rewrite the Hamiltonian $h$ in~\eqref{def:hamiltonianInitialNotSplit} as a singular perturbation problem given by a  fast oscillator weakly coupled with a one degree of freedom Hamiltonian with a saddle point and a separatrix (Step A).
%
In Section~\ref{subsection:singularities}, we consider the  time-parametrization of the separatrix and analyze the properties of its analytical continuation (Step B).
In Section~\ref{subsection:innerDerivation} we give some heuristic ideas regarding the parametrization of the perturbed manifolds on certain complex domains (Step E) and deduce the singular change of variables which leads to the inner equation (Step C).
Finally, in Section~\ref{subsection:innerComputations}, we present the study of certain solutions of the inner equation and give an asymptotic formula for their difference (Step D).

\subsection{A singular perturbation formulation of the problem} 
\label{subsection:reformulation}

When studying a close to integrable Hamiltonian system at a resonance, it is usual to ``blow-up'' the ``resonant zone'' to capture the slow-fast time scales.
In this section we present the singular change of coordinates which transforms the Hamiltonian $h$ in \eqref{def:hamiltonianInitialNotSplit} into a pendulum-like Hamiltonian plus a fast oscillator with a small coupling, namely
\[
H(\la,\La,x,y) = 
-\frac{3}{2}\La^2 + V(\la) +
\frac{x y}{\sqrt{\mu}} + o(1),
\]
with respect to the symplectic form $d\la \wedge d\La + i dx \wedge dy$.
In these coordinates, the first order of the  Hamiltonian has a saddle in the $(\la,\La)$--plane and a center in the $(x,y)$--plane.
Notice that the system possesses two time scales ($\sim 1$ and $\sim1/{\sqrt{\mu}}$). Recall that this two time scales are also present in the eigenvalues of $L_3$ in~\eqref{eq:specRot}.

We consider Poincar\'e coordinates for the RPC$3$BP (see~\eqref{def:hamiltonianInitialNotSplit}) in order to write the system as a close to integrable Hamiltonian system and decouple (at first order) the saddle and the center behaviour.
To this end, we first consider the symplectic polar and Delaunay coordinates.
%

\paragraph{Polar Coordinates.}
%
%
Let us consider the change of coordinates: 
\[
\phi_{\pol}:(r,\tht,R,G) \mapsto (q,p),
\]
where $r$ is the radius, $\tht$ the argument of $q$, $R$ the linear momentum in the $r$ direction and $G$ is the angular momentum.
Then, the Hamiltonian~\eqref{def:hamiltonianInitialNotSplit}, becomes 
\begin{equation}\label{def:hamiltonianPolars}
	H^{\pol} = H_0^{\pol} + \mu H_1^{\pol},
\end{equation}
where
\begin{equation} \label{def:hamiltonianPolarsSplitting}
\begin{split}
H^{\pol}_0(r,R,G) &= \frac{1}{2} \paren{ R^2 + \frac{G^2}{r^2} } - \frac{1}{r} - G,  \\
\mu H^{\pol}_1(r,\theta;\mu) &= 
\frac{1}{r} -
\frac{ 1-\mu}{\sqrt{r^2-2\mu r\cos \tht + \mu^2}}
- \frac{\mu}{\sqrt{r^2+2(1-\mu) r\cos\tht +(1-\mu)^2}}.
\end{split}
\end{equation}
The critical point $L_3$ (see~\cite{Szebehely} for the details) satisfies that, as $\mu \to 0$,
\begin{equation}\label{def:pointL3}
(r,\tht,R,G) = (d_{\mu},0,0,d_{\mu}^2), 
\qquad \text{with} \quad
d_{\mu} = 1 + \frac{5}{12} \mu + \OO(\mu^3).
\end{equation}



\paragraph{Delaunay coordinates.}

The Delaunay elements,
denoted $(\ell,L,\hat{g},G)$,
are action--angle variables for the $2$-Body Problem (for negative energy) in non-rotating coordinates.
The variable $\ell$ is the mean anomaly, 
$\hat{g}$ is the argument of the pericenter,
$L$ is the square root of the semi major axis and $G$ is the angular momentum,
(see~\cite{MeyerHallOffin}).

Let us introduce some formulae to describe these elements from the non-rotating polar coordinates $(r,\hat{\tht},R,G)$, namely $\hat{\tht}={\tht}+t$. 
The action $L$ is defined by
\begin{equation*}
-\frac{1}{2L^2} =
\frac{1}{2} \paren{ R^2 + \frac{G^2}{r^2} } - \frac{1}{r},
\end{equation*}
%
and the (osculating) eccentricity of the body is expressed as 
\begin{equation}\label{def:eccentricityDelaunay}
e = \sqrt{1-\frac{G^2}{L^2}}
= \frac{\sqrt{(L-G)(L+G)}}{L}.
\end{equation} 
Let us recall now the ``anomalies'': the three angular parameters that define a position at the (osculating) ellipse. These are the mean anomaly $\ell$,  the eccentric anomaly $u$, and the true anomaly $f$, which satisfy
\begin{equation}\label{eq:rthetaDefinition}
r = L^2(1-e\cos u) \qquad
\text{ and } \qquad 
\hat{\tht} = f + \hat{g}.
\end{equation}
To use these elements in a rotating frame, we consider rotating Delaunay coordinates $(\ell,L,g,G)$, where the new
angle is defined as $g=\hat{g}-t$ (the argument of the pericenter with respect to the line defined by the primaries $S$ and $J$).
As a result, 
\begin{equation}\label{eq:thetaDefinition}
{\tht} = f + {g},
\end{equation}
and the unperturbed Hamiltonian $H_0^{\pol}$ becomes
\[
H_0^{\pol} = -\frac{1}{2L^2} - G.
\]
The eccentric and true anomalies are related by 
\begin{equation}\label{eq:fTrigDefinition}
\cos f=\frac{\cos u-e}{1-e\cos u}, 
\qquad
\sin f=\frac{\sqrt{1-e^2}\sin u}{1-e\cos u},
\end{equation}
and the mean anomaly $u$ is given by Kepler's equation
\begin{equation}\label{eq:uImplicitDefinition}
u-e \sin u = \ell.
\end{equation}
The critical point $L_3$ (see~\eqref{def:pointL3}) satisfies $\tht = \ell + g = 0$ and
\begin{equation}\label{eq:pointL3Delaunay}
%
L = \sqrt{\frac{d_{\mu}}{2-d^3_{\mu}}}
= 1 + \OO(\mu), \qquad
G = d^2_{\mu}
= 1  + \OO(\mu), \qquad
L- G = \OO(\mu^2).
\end{equation}
Note that the Delaunay coordinates are not well defined for circular orbits ($e=0$), since the pericenter, and as a consequence the angle $g$, are not well defined.

\paragraph{Poincar\'e coordinates.}

To ``blow-down'' the singularity of the Delaunay coordinates at circular motions,
we use the classical Poincar\'e coordinates, which can be expressed  by means of (rotating) Delaunay variables.
Let us define
\begin{equation}\label{def:changePoincare}
\phi_{\Poi}: (\la, L, \eta, \xi) \mapsto
(r, \tht, R,G) ,
\end{equation}
given by
\begin{equation}\label{def:changePoincare2}
\begin{aligned}
\la=\ell+g, \qquad  
\eta=\sqrt{L-G} e^{i g}, \qquad
\xi= \sqrt{L-G} e^{-i g}.
\end{aligned}
\end{equation}
%
These coordinates are symplectic and analytic. 
Moreover, even though they are defined through the Delaunay variables, they are also analytic when the eccentricity tends to zero (i.e at $L=G$), see~\cite{MeyerHallOffin, Fejoz2013}.
%

The Hamiltonian equation associated to~\eqref{def:hamiltonianPolars}, expressed in Poincar\'e coordinates, defines a Hamiltonian system with respect to the symplectic form $d\la \wedge dL + i \h d\eta \wedge d\xi$ and
\begin{equation*} 
H^{\Poi} = H_0^{\Poi}
+ \mu H_1^{\Poi},
\end{equation*}
where
\begin{equation}\label{def:hamiltonianPoincareSplitting}
\begin{split}
H_0^{\Poi}(L,\eta,\xi) = -\frac{1}{2L^2} - L +  \eta \xi 
\qquad \text{and} \qquad
H_1^{\Poi} = H_1^{\pol} \circ \phi_{\Poi}.
\end{split}
\end{equation}
%
%
%
%
In Poincar\'e coordinates, the critical point $L_3$, as given in~\eqref{eq:pointL3Delaunay}, satisfies
\[
\la=0, \qquad
(L,\eta,\xi) = (1,0,0) + 
\OO(\mu).
\]
%
%
The linearized part of the vector field associated to this point has, at first order, an uncoupled nilpotent and center blocks,
\begin{equation*}
\begin{pmatrix}
0 & -3 & 0 & 0 \\
0 & 0 & 0 & 0 \\
0 & 0 & i & 0 \\
0 & 0 & 0 & -i
\end{pmatrix} + \OO(\mu).
\end{equation*}
The center is found on the projection to coordinates $(\eta,\xi)$ and the degenerate behavior on the projection to $(\la,L)$. 
%
%

The perturbative term $\mu H_1^{\Poi}$ is not explicit. 
We overcome this problem by computing the first terms of the series of $\mu H_1^{\Poi}$ in powers of $(\eta,\xi)$, (see Lemma~\ref{lemma:seriesH1Poi}).
%

Notice that, on the original coordinates, Hamiltonian $h$ (see~\eqref{def:hamiltonianInitialNotSplit}) is analytic at points away from collision with the primaries.
However, the collisions are not as easily defined in Poincar\'e coordinates. 



\paragraph{A singular scaling.}
%
We consider the parameter
\begin{equation*}
	\de= \mu^{\frac{1}{4}}
\end{equation*}
and we define the  symplectic scaling 
\begin{equation}\label{def:changeScaling}
\phi_{\sca}:
(\la, \La, \xh,  \yh) 
 \mapsto 
(\la,L,\eta,\xi),
\qquad
L = 1 + \de^2 \La , \quad
\eta = \de \xh , \quad
\xi = \de \yh , \quad
\end{equation}
and the time reparameterization $t = \de^{-2} \tau$. The transformed equations are Hamiltonian with respect to the Hamiltonian
\[
\de^{-4} \paren{ H_0^{\Poi} \circ \phi_{\sca}}+\paren{H_1^{\Poi} \circ \phi_{\sca} }
\]
and the symplectic form $d\la\wedge d\Lambda+idx\wedge dy$.  The Hamiltonian (up to a constant) satisfies
\begin{align*}
\de^{-4} \paren{ H_0^{\Poi} \circ \phi_{\sca}} &=-\frac{3}{2}\La^2 
+ \frac{1}{\de^4} F_{\pend}(\de^2\La)
+ \frac{\xh\yh}{\de^2}, \\
\paren{H_1^{\Poi} \circ \phi_{\sca} }&=
V(\la) + \OO(\de),
\end{align*}
with 
\begin{equation}\label{def:Fpend}
\begin{split}
V(\la) &=  H_1^{\Poi}(\la,1,0,0;0),\\
F_{\pend}(z)& = \paren{-\frac{1}{2(1+z)^2}-(1+z)}+\frac{3}{2} + \frac{3}{2}z^2 = \OO(z^3).
\end{split}
\end{equation}
The function $V(\la)$, which we call the potential, has an explicit formula:
\begin{equation}\label{def:potentialV}
V(\la) =  H_1^{\pol}(1,\la;0)
= 1 - \cos \la - \frac{1}{\sqrt{2+2\cos \la}},
\end{equation}
where $H_1^{\pol}$ is defined in~\eqref{def:hamiltonianPolarsSplitting}.
Indeed, taking $\de=0$ on the change of coordinates~\eqref{def:changeScaling}, we have that $(\la,L,\eta,\xi)=(\la,1,0,0)$.
%
%
These coordinates, correspond with a circular orbit, $e=0$, and applying~\eqref{eq:rthetaDefinition} and \eqref{eq:thetaDefinition}, we obtain that $(r,\tht)=(1,\la)$.

We summarize the previous results in the following theorem.

\begin{theorem}\label{theorem:HamiltonianScaling}
The Hamiltonian system given by $h$ in~\eqref{def:hamiltonianInitialNotSplit} expressed in coordinates $(\la,\La,x,y)$
defines a Hamiltonian system with respect to the symplectic form $d\la \wedge d\La + i dx \wedge dy$ 
and the Hamiltonian
\begin{equation} \label{def:hamiltonianScaling}
\begin{split}
{H} = 
{H}_{\pend} + 
{H}_{\osc} + 
H_1,
\end{split}
\end{equation}
with 
\begin{align} 
&{H}_{\pend}(\la,\La) = -\frac{3}{2} \La^2  + V(\la), \qquad
{H}_{\osc}(x,y; \de) = \frac{x y}{\de^2}, 
\label{def:HpendHosc}\\
&H_1(\la,\La,\xh,\yh;\de) =
\paren{H_1^{\Poi} \circ \phi_{\sca}}  - V(\la) + \frac{1}{\de^4} F_{\pend}(\de^2\La),
\label{def:hamiltonianScalingH1}
\end{align}
and $H_1^{\Poi}$, $F_{\pend}$ and $V$ defined in
\eqref{def:hamiltonianPoincareSplitting},
\eqref{def:Fpend}
and
\eqref{def:potentialV}, respectively.
Moreover, the Hamiltonian $H$ is real-analytic\footnote{
	Real-analytic in the sense of $\conj{H(\la,\La,x,y;\de)}= H(\conj{\la},\conj{\La},y,x;\conj{\de}).$} 
away from collision with the primaries. 

Moreover, for $\de > 0$ small enough:
\begin{itemize}
\item The critical point $L_3$ (see~\eqref{def:pointL3}) expressed in  coordinates $(\la,\La,x,y)$ is given by
\begin{equation}\label{def:pointL3sca}
\Ltres(\de)=\paren{0,
	\de^2 \LtresLa(\de),
	\de^3 \Ltresx(\de),
	\de^3 \Ltresy(\de)},
\end{equation}
with $\vabs{\LtresLa(\de)}$, $\vabs{\Ltresx(\de)}$, $\vabs{\Ltresy(\de)} \leq C$, for some constant $C>0$ independent of $\de$.
\item The point $\Ltres(\de)$ is a saddle-center equilibrium point and its linear part is
\begin{equation*}
	\begin{pmatrix}
		0 & -3 & 0 & 0 \\
		-\frac{7}{8} & 0 & 0 & 0 \\
		0 & 0 & \frac{i}{\de^2} & 0 \\
		0 & 0 & 0 & -\frac{i}{\de^2}
	\end{pmatrix} + \OO(\de).
\end{equation*}
Therefore, it possesses one-dimensional stable and unstable manifolds and a two- dimensional center manifold.
\end{itemize}
\end{theorem}
%
The proof of Theorem~\ref{theorem:HamiltonianScaling} follows from the results obtained through Section~\ref{subsection:reformulation}.

Since the original Hamiltonian is symmetric with respect to the involution ${\Phi}$ in~\eqref{def:involutionCartesians}, the Hamiltonian $H$ is reversible with respect to the involution 
\begin{equation}\label{def:involutionScaling}
	\wt{\Phi}(\la,\La,x,y) = (-\la,\La,y,x).
\end{equation}
From now on, we consider as ``new'' unperturbed Hamiltonian
\begin{equation}\label{def:hamiltonianScalingH0}
{H}_0(\la,\La,\xh,\yh;\de) = {H}_{\pend}(\la,\La) + {H}_{\osc}(\xh,\yh; \de),
\end{equation}
which corresponds to an uncoupled pendulum-like  Hamiltonian ${H}_{\pend}$ and an oscillator ${H}_{\osc}$, and  we refer
to $H_1$ as the perturbation.

\subsection{The Hamiltonian \texorpdfstring{$H_{\pend}$}{Hp} and its separatrices} 
\label{subsection:singularities}

In this section we analyze the 1-degree of freedom Hamiltonian ${H}_{\pend}(\la,\La)$ introduced in~\eqref{def:HpendHosc},
\begin{equation*}
{H}_{\pend}(\la,\La) = -\frac{3}{2} \La^2  + V(\la), \qquad
V(\la) = 1 - \cos \la - \frac{1}{\sqrt{2+2\cos \la}},
\end{equation*}
and the associated Hamiltonian system
\begin{equation} \label{eq:sistemODEs}
\dot \la= -3\La, \qquad
\dot \La=-\sin \la \left(1-\frac{1}{(2+2\cos\la)\efr{3}{2}
}\right),
\end{equation}
This Hamiltonian system has a singularity at $\la=\pi$, which corresponds to the collision with the small primary $P$, and a saddle at $(\la,\La)=(0,0)$ with two homoclinic connections or separatrices, see  Figure~\ref{fig:separatrix}.
\begin{figure} 
	\centering
	\begin{overpic}[scale=0.75]{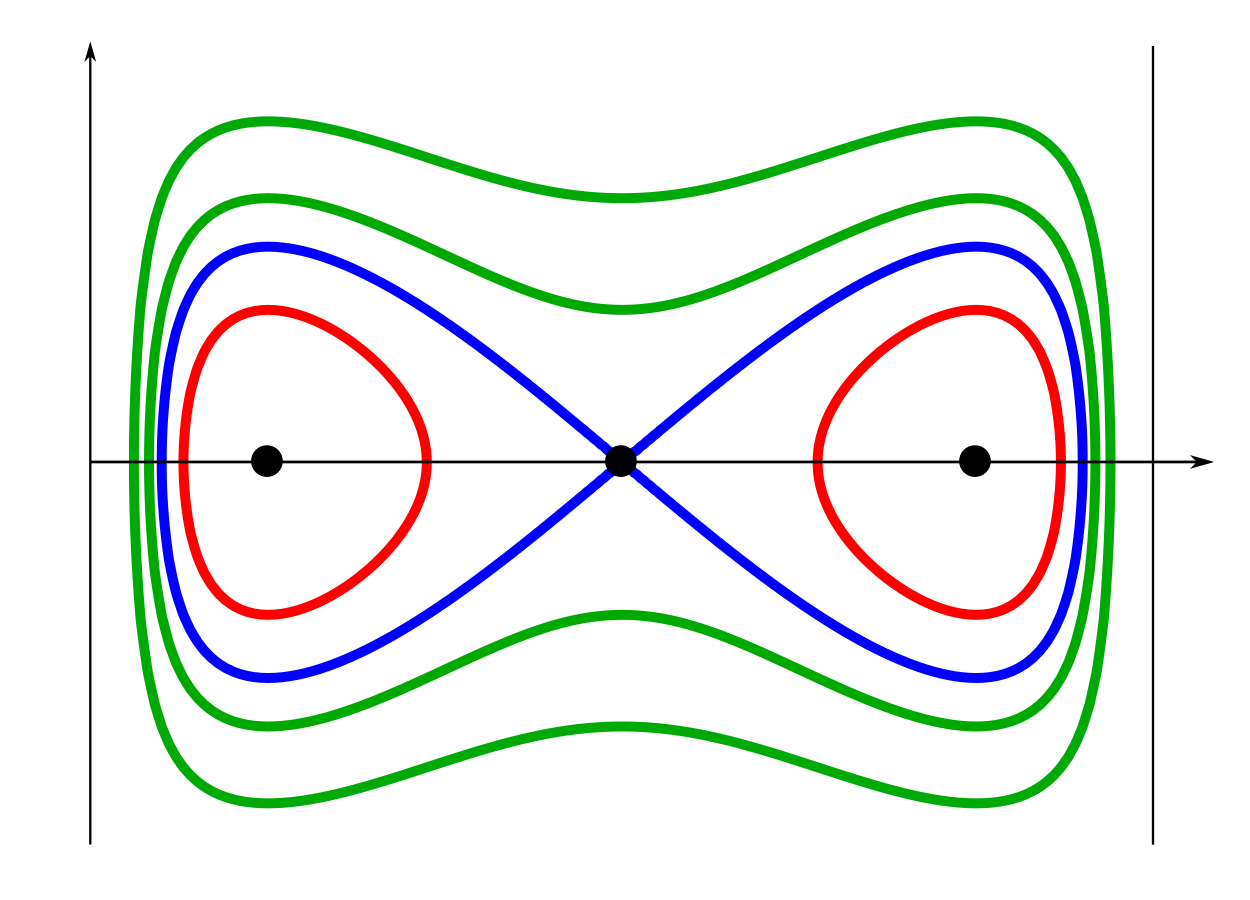}
		\put(93,31){{$\pi$}}
		\put(77,29){{$\frac{2}{3}\pi$}}
		\put(49,29){{$0$}}
		\put(17,29){{$-\frac{2}{3}\pi$}}
		\put(-2,31){{ $-\pi$}}
		\put(98,34){{$\la$}}
		\put(5,70){{$\La$}}
	\end{overpic}
	\bigskip
	\caption{Phase portrait of equation~\eqref{eq:sistemODEs}. On blue the two separatrices.}
	\label{fig:separatrix}
\end{figure}
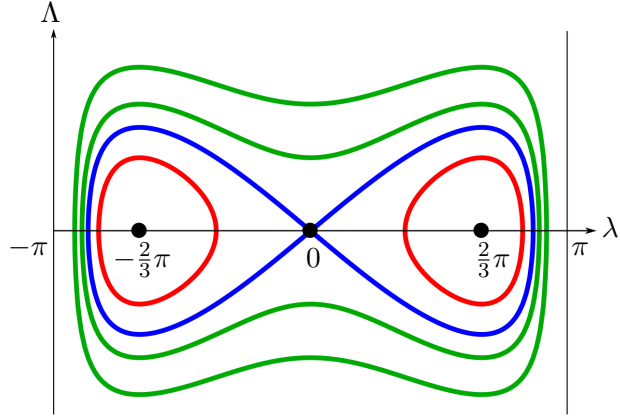
From now on, we only consider the separatrix on the right; by symmetry (see \eqref{def:involutionScaling}), the results obtained below are analogous for the separatrix on the left. 

We consider the real-analytic time parametrization of the separatrix,
\begin{equation}\label{eq:separatrixParametrization}
\begin{split}
	\s:\reals &\to \torus \times \reals \\
		t &\mapsto \s(t)=(\la_h(t),\La_h(t)),
\end{split}
\end{equation}
with initial condition 
\[
\s(0)=(\la_0,0) 
\quad \text{ where } \quad 
\la_0 \in \left(\frac{2}{3}\pi,\pi\right).
\]
\begin{theorem}\label{theorem:singularities}
	The real-analytic time parametrization $\s$ in~\eqref{eq:separatrixParametrization} satisfies:
	\begin{itemize}
	\item It extends analytically to the strip
	\begin{equation}\label{def:PiA}
	\Pi_A = \{ t \in \complexs \st  |\Im t|< A \},
	\end{equation}
	where 
	\begin{equation}\label{def:integralA}
	A= \int_0^{a_+} \frac{1}{1-x}\sqrt\frac{x}{3(x+1)(a_+-x)(x-a_-)}  dx
	\approx 0.177744,
	\end{equation}
	with 
	$ 
	a_{\pm}=-\frac{1}{2}\pm\frac{\sqrt{2}}{2}.
	$
	\item It has only two singularities in $\partial \Pi_A$ at $t=\pm iA$.
	\item There exists $\nu>0$ such that,  for $t \in \complexs$ with $\vabs{t -iA}<\nu$ and $\arg(t -iA) \in (-\frac{3\pi}{2},\frac{\pi}{2})$, $\s(t)=(\la_h(t),\La_h(t))$ can be expressed as
	\begin{equation}\label{eq:homoclinicaSingularities}
	\begin{split}
	&\la_h(t) =  \pi + 3\al_{+} (t- iA)^{\frac{2}{3}} + \OO(t- iA)^{\frac{4}{3}}, \\[0.5em]
	&\La_h(t) =  -\frac{ 2 \al_{+}}{3} \frac{1}{(t- iA)^{\frac{1}{3}}} + 
	\OO(t- iA)^{\frac{1}{3}}, 
	\end{split}
	\end{equation}
with $\al_{+} \in \complexs$ such that $\al_{+}^3=\frac{1}{2}$. 

An analogous result holds for $\vabs{t+iA}<\nu$, $\arg(t+iA) \in(-\frac{\pi}{2},\frac{3\pi}{2})$ and $\al_- = \conj{\al_+}$.
\end{itemize}
\end{theorem}

%
%

We can also describe the zeroes of $\La_h(t)$ in $\ol{\Pi_A}$. 


\begin{proposition}\label{proposition:domainSeparatrix}
	Consider the real-analytic time parametrization $\s(t)=(\la_h(t),\La_h(t))$ and the domain $\Pi_A$ defined in~\eqref{eq:separatrixParametrization}
	and \eqref{def:PiA} respectively.
	Then, $\La_h(t)$ has only one zero in $\ol{\Pi_A}$ at $t=0$.
\end{proposition}


We can expand the region of analyticity of the time parametrization $\s$.

\begin{corollary}
%
There exists $0<\beta<\frac{\pi}{2}$ such that the real-analytic time parametrization $\s(t)$ extends analytically to
\begin{equation}\label{def:dominiBow}
\begin{split}	
\Pi^{\mathrm{ext}}_{A, \beta} =& \claus{
t \in \complexs \st 
\vabs{\Im t} < \tan \beta \, \Re t + A}
\cup \\
&\claus{
t \in \complexs \st 
\vabs{\Im t} < -\tan \beta \, \Re t + A}.
\end{split}
\end{equation}
(See Figure~\ref{fig:dominiBow}). Moreover, 
\begin{enumerate}
\item $\s$ has only two singularities on $\partial \Pi^{\mathrm{ext}}_{A, \beta}$ at 
$t=\pm iA$.
\item $\La_h$ has only one zero in the closure of  $\Pi^{\mathrm{ext}}_{A, \beta}$ at 
$t=0$.
\end{enumerate}
\end{corollary}	


\begin{proof}
By~\cite{Fon95}, there exists $T>0$ such that $\s(t)$ is analytic in
$
\claus{\vabs{\Re t} > T}.
$
Moreover, applying Theorem~\ref{theorem:singularities},  $\s(t)$ has two branching points at $t=\pm iA$ and can be expressed as in~\eqref{eq:homoclinicaSingularities} in the domains
\[
\textstyle
D_1= \claus{\vabs{t- iA}<\nu, \,
\arg(t -iA) \in \paren{-\frac{3\pi}{2},\frac{\pi}{2}}}\bigcup \claus{\vabs{t+ iA}<\nu, \, \arg(t +iA) \in \paren{-\frac{\pi}{2},\frac{3\pi}{2}}},
\]
for some $\nu>0$. This implies that the only singularities in $D_1$ are at $t=\pm iA$.

Thus, we only need to check Item 1 in
\[
\left(\Pi^{\mathrm{ext}}_{A, \beta} \cap 
\claus{\vabs{\Re t} \leq T}\right)\setminus D_1.
\]
To this end, note that, by Theorem~\ref{theorem:singularities},  $\s(t)$ is analytic in the compact set
$
D_2 = \paren{\ol{\Pi_A} \cap
\claus{\vabs{\Re t} \leq T} }
\setminus D_1
$. Therefore, there exists a cover of $\partial D_2$ by open balls centered in $\partial D_2$ where $\s(t)$ is analytic.
Moreover, since $\partial D_2$ is compact, it has a finite subcover.
This implies that there exists $\eta>0$ such that we can extend the analyticity domain of $\s(t)$ to
$
\paren{{\Pi_{A+\eta}}\cap
\claus{\vabs{\Re t} \leq T} } \setminus D_1
$.
In particular, taking $\beta=\arctan ({\eta}/{T})$,
$\s(t)$ is analytic in 
$
({\Pi_{A,\beta}^{\mathrm{ext}} \cap
	\claus{\vabs{\Re t} \leq T} }) \setminus D_1.
$

The prove of Item 2 follows the same lines.
\end{proof}

\begin{figure} 
	\centering
	\begin{overpic}[scale=0.65]{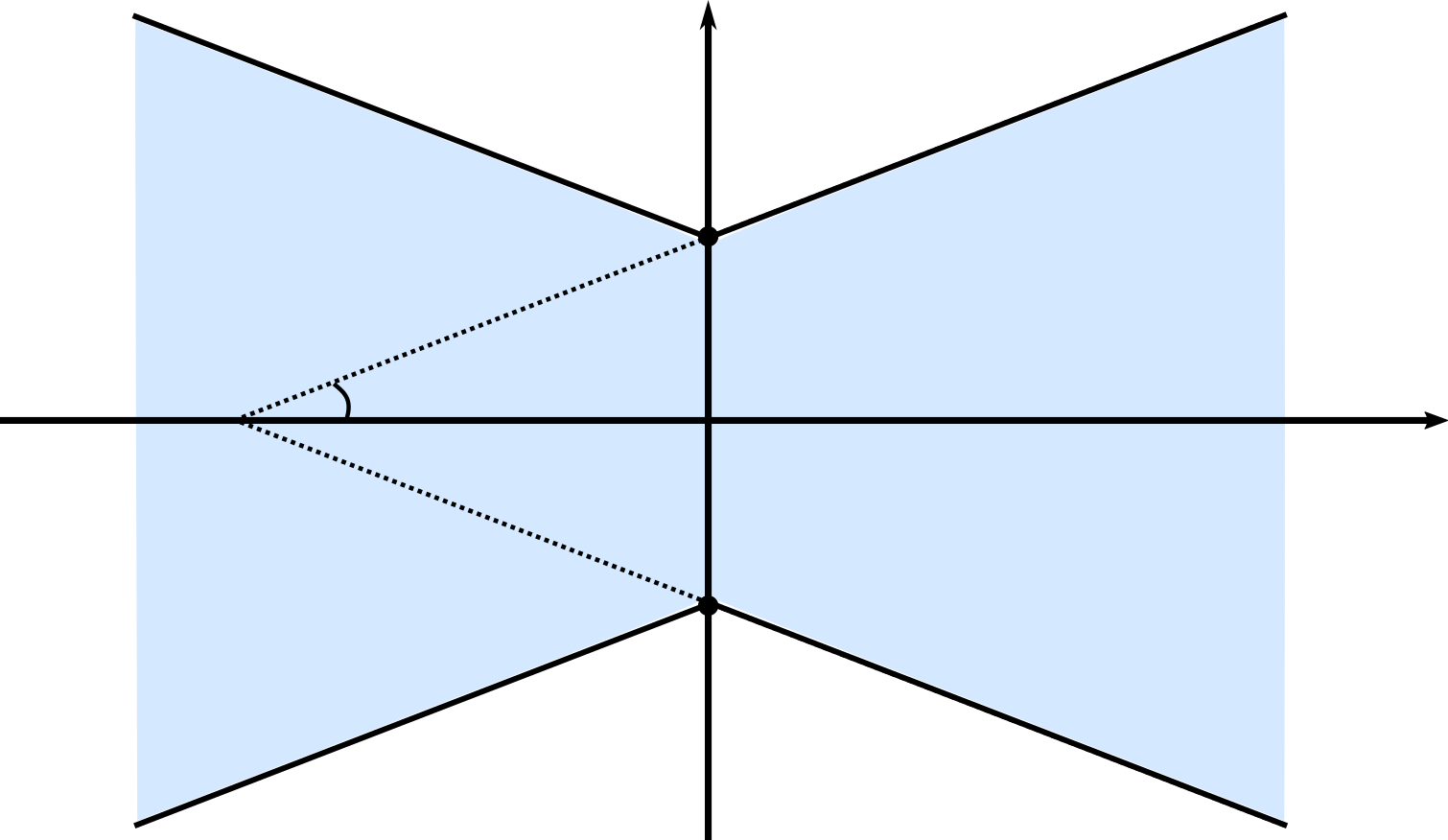}
		\put(101,27.5){{$\Re t$}}
		\put(45.5,59){{$\Im t$}}
		\put(28,30){{$\beta$}}
		\put(43,45){{$iA$}}
	\end{overpic}
	\bigskip
	\caption{Representation of the domain $\Pi_{A,\beta}^{\mathrm{ext}}$ in~\eqref{def:dominiBow}.}
	\label{fig:dominiBow}
\end{figure}

\subsection{Derivation of the inner equation} 	
\label{subsection:innerDerivation}

%
The inner equation  associated to the Hamiltonian $H$ in \eqref{def:hamiltonianScaling} describes the dominant behavior of suitable complex parametrizations of the invariant manifolds  close to (one of) the singularities $\pm iA$ of the unperturbed separatrix. 
Let us explain how this equation arises from the Hamiltonian $H$. 
%
%
%
%
%

First, we consider the translation of the equilibrium point $\Ltres(\de)$ to the origin, 
\begin{equation}\label{def:changeEqui}
	\phi_{\equi}:(\la,\La,x,y) \mapsto (\la,\La,x,y) + \Ltres(\de).
\end{equation}
Second, to measure the distance of the stable and unstable manifolds, we parameterize them as graphs.
In the unperturbed case, we know that the invariant manifolds coincide along the separatrix $(\la_h(t),\La_h(t),0,0)$.
Since we need to involve, in some sense, the time; we consider as a new independent variable $u$ such that $\la_h(u)=\la$. 
Notice that $\dot{u}=1$ for the unperturbed system.
To this end, we consider the symplectic change of coordinates 
\begin{equation}\label{def:changeOuter}
\phi_{\out}:(u,w,\xp,\yp) \to (\la,\La,\xh,\yh),
\qquad 
{\la}= \la_h(u), 
\quad
 {\La}= \La_h(u) - \frac{w}{3\La_h(u)},
\end{equation}
where $\s=(\la_h,\La_h)$ is the parametrization of the separatrix studied in Theorem~\ref{theorem:singularities}. 
%
Notice that, except for $u=0$ (see Proposition~\ref{proposition:domainSeparatrix}), the perturbed manifolds can be expressed as a graph and the change \eqref{def:changeOuter} is well defined. 

The Hamiltonian $H$, written in these coordinates and after the translation $\phi_{\equi}$ in~\eqref{def:changeEqui}, becomes
\begin{equation}\label{def:hamiltonianOuter}
H^{\out} =  H^{\out}_0 +  H^{\out}_1,
\end{equation}
with 
\begin{equation*}
\begin{split}
H^{\out}_0(w,x,y) &= 
w + \frac{xy}{\de^2}, \qquad
%
H^{\out}_1 =
{H} \circ \paren{\phi_{\equi} \circ \phi_{\out}} -H^{\out}_0.
\end{split}
\end{equation*}
Since we look for the perturbed manifolds as graphs with respect to $u$, we consider
 parametrizations
\[
\zdOut(u) = \left(\wdOut(u),\xdOut(u),\ydOut(u)\right)^T,
\quad \text{ for } \diamond=\unstable,\stable,
\]
such that the unstable  and stable invariant manifolds of $H$ associated to $\Ltres(\de)$ can be expressed as
\begin{equation*}
W^{\diamond}= \left\{ \paren{\la_h(u), \La_h(u)-\frac{\wdOut(u)}{3\La_h(u)}, \xdOut(u), \ydOut(u)} + \Ltres(\de)\right\}, \quad \text{for } \diamond=\unstable,\stable,
\end{equation*}
with $u$ belonging to appropriate domains.
The proof of existence of $\zuOut$ and $\zsOut$ defined in appropriate (complex) domains requires a significant amount of technicalities. 
We present this result in the companion paper, see~\cite{articleOuter}.

%
%

%
%

Due to the slow-fast character of the system, to capture the asymptotic first order of the difference $\dzOut = \zuOut-\zsOut$, we need to give the main terms of this difference close to the singularities, concretely, up to distance of order $\de^2$. 
%
%
%
%
To this end, we derive the inner equation, see~\cite{Bal06, BalSea08}, which  contains the first order of the Hamiltonian $H^{\out}$ (see~\eqref{def:hamiltonianOuter}) close to (one of) the singularities and is independent of the small parameter $\de$.
That is, we look, for instance, for a Hamiltonian which is a good approximations of $H^{\out}$ in a neighborhood of $u=iA$. 
Here, we focus on a domain near the singularity $u=iA$, but a similar analysys  can be done near $u=-iA$. 

Since we need to control the difference up to distance of order $\de^2$ of the singularity $u=iA$, we consider $U$ such that
\[
u-iA = \de^2 U.
\] 
Notice that we can take $\vabs{U}\gg 1$ independent of $\de$.
Close to the singularity $u=iA$, the homoclinic connection is not the dominant term of the perturbed invariant manifolds anymore.
Let us be more precise, take $\La=\La_h(u)-\frac{w}{3\La_h(u)}$, and recall that, by Theorem~\ref{theorem:singularities}, we have
\[
\La_h(u) \sim -\frac{2\al_+}{3} (u-iA)^{-\frac{1}{3}}, \quad
\text{for}  \quad \vabs{u-iA} \leq \nu,
\]
or equivalently,
\[
\La_h(iA + \de^2 U) \sim 
-\frac{2\al_+}{3 \de^{\frac{2}{3}} U^{\frac{1}{3}}} \sim
\OO\paren{\frac{1}{\de^{\frac{2}{3}} U^{\frac{1}{3}} }}.
\]
Then,
\begin{equation}\label{eq:approxInnerw}
\wdOut(iA + \de^2 U) \sim 
3 \La_h^2(iA + \de^2 U) \sim
\OO\paren{\frac{1}{\de^{\frac{4}{3}} U^{\frac{2}{3}}}}.
\end{equation}
In addition,  the unperturbed Hamiltonian must have all of its terms of the same order. Therefore,
\begin{equation}\label{eq:approxInnerxy}
\frac{\xdOut(iA + \de^2U) \, \ydOut(iA+\de^2 U)}{\de^2}
\sim \OO\paren{\frac{1}{\de^{\frac{4}{3}}
		U^{\frac{2}{3}}}}.
\end{equation}
By symmetry, $\xdOut(iA + \de^2U)$, $\ydOut(iA + \de^2U) \sim \OO(\de^{\frac{1}{3}} U\mefr{1}{3})$.

To avoid the dependence on the inner equation with respect to $\al_+$ (see Theorem~\ref{theorem:singularities}) and to keep the symplectic character, we perform the scaling
\begin{equation}\label{def:changeInner}
\phi_{\Inner}:(U,W,X,Y) \to (u,w,x,y),
\end{equation}
given by
\begin{equation*}
U=\frac{u-iA}{\de^2}, 
\qquad
W=\de^{\frac{4}{3}}\frac{w}{2 \al_+^2}, 
\qquad
X=\frac{x}{\de^{\frac{1}{3}} \sqrt{2} \al_+ }, \qquad
Y=\frac{y}{\de^{\frac{1}{3}}\sqrt{2} {\al_+}},
%
\end{equation*}
and the time scaling $\tau=\de^2 t$. 
%
The heuristics above lead us to assume that $(U,W,X,Y) = \OO(1)$ when $u-iA=\OO(\de^2)$.
In the following proposition, by applying the change of coordinates $\phi_{\Inner}$,  we obtain the inner equation of the Hamiltonian $H^{\out}$.

\begin{proposition}\label{proposition:innerDerivation}
	The Hamiltonian equations associated to~\eqref{def:hamiltonianOuter} expressed in inner coordinates (see~\eqref{def:changeInner}) are Hamiltonian with respect to
	\begin{equation*}
	H^{\Inner} = \HH + H_1^{\Inner},
	\end{equation*}
	where 
	\begin{equation}\label{def:hamiltonianInner}
	\HH(U,W,X,Y) =\left. H^{\Inner}(U,W,X,Y;\de)\right|_{\de=0} = W + XY + \KK(U,W,X,Y),
	\end{equation}
	with
	\begin{align}
	\KK(U,W,X,Y) &= 
	-\frac{3}{4}U^{\frac{2}{3}} W^2 
	- \frac{1}{3 U^{\frac{2}{3}}}
	\paren{\frac{1}{\sqrt{1+\JJ(U,W,X,Y)}} - 1 } \label{def:hamiltonianK}
	\end{align}
	and
	\begin{equation}\label{def:hFunction}
	\begin{split}
	\JJ(U,W,X,Y) =& \,  
	\frac{4 W^2}{9 U^{\frac{2}{3}} } 
	-\frac{16 W}{27 U^{\frac{4}{3}}}  
	+\frac{16}{81 U^{2}}
	%
	+\frac{4(X+Y)}{9 U}
	\paren{W -\frac{2}{3 U^{\frac{2}{3}}}} \\[0.5em]
	&- \frac{4i(X-Y)}{3 U^{\frac{2}{3}}}
	-\frac{X^2+Y^2}{3 U^{\frac{4}{3}}}
	+\frac{10 XY}{9 U^{\frac{4}{3}}}. 	
	\end{split}
	\end{equation}
	Moreover, if 
	$\cttInnDerA^{-1} \leq \vabs{U} \leq \cttInnDerA$ and 
	$\vabs{(W,X,Y)} \leq \cttInnDerB$ 
	for some $\cttInnDerA>1$ and $0<\cttInnDerB<1$,
	we have that there exist $\cttInnDerC, \cttInnDerAA, \cttInnDerBB>0$ independent of $\de, \cttInnDerA, \cttInnDerB$ such that
	\[
	\vabs{H_1^{\Inner}(U,W,X,Y;\de)} \leq 
	\cttInnDerC \cttInnDerA^{\cttInnDerAA} \cttInnDerB^{\cttInnDerBB}\de^{\frac{4}{3}}.
	\] 
\end{proposition}

\begin{remark}\label{remark:innerComputationConjugats}
The change of coordinates \eqref{def:changeInner} allows us to study an approximation of the invariant manifolds $\zusOut(u)$ near the singularity $u=iA$. 
To obtain an approximation near $u=-iA$, one can proceed analogously by
\begin{equation*}\label{def:changeInnerNegatiu}
U=\frac{u+iA}{\de^2}, 
\qquad
W=\de^{\frac{4}{3}}\frac{w}{2 \al_-^2}, 
\qquad
X=\frac{x}{\de^{\frac{1}{3}} \sqrt{2} \al_-}, \qquad
Y=\frac{y}{\de^{\frac{1}{3}}\sqrt{2} \al_-},
\end{equation*}
where $\al_- = \conj{\al_+}$ (see Theorem \ref{theorem:singularities}). 
\end{remark}

\subsection{The solutions of the inner equation and their difference}
\label{subsection:innerComputations}
We devote this section to study two special solutions of the inner equation given by the Hamiltonian $\HH$ in~\eqref{def:hamiltonianInner}. 
%
%
We introduce $Z=(W,X,Y)$ and the matrix
\begin{equation}\label{def:matrixAAA}
\AAA= \begin{pmatrix}
0 & 0 & 0 \\
0 & i & 0 \\
0 & 0 & -i
\end{pmatrix}.
\end{equation}
Then, the equation associated to the Hamiltonian $\HH$ can be written as
\begin{equation}\label{eq:systemEDOsInner}
\left\{ \begin{array}{l}
\dot{U} = 1 + g(U,Z),\\
\dot{Z} = \AAA Z + f(U,Z),
\end{array} \right.
\end{equation}
where 
$f = \paren{-\partial_U \KK, 
	i \partial_Y \KK, -i\partial_X \KK }^T$ 
and
$g = \partial_{W} \KK$.

We look for solutions of this equation parametrized as graphs with respect to $U$,
namely we look for functions
\begin{equation*}
\ZdInn(U) = \big(\WdInn(U),\XdInn(U),\YdInn(U)\big)^T,
\qquad
\text{for } \diamond=\unstable,\stable,
\end{equation*}
satisfying the invariance condition given by~\eqref{eq:systemEDOsInner}, that is
\begin{equation}\label{eq:invariantEquationInner}
\partial_U \ZdInn = 
\AAA \ZdInn + \RRR[\ZdInn],
\qquad
\text{for } \diamond=\unstable,\stable,
\end{equation}
where
\begin{equation}\label{def:operatorRRRInner}
\RRR[\varphi](U)= 
\frac{f(U,\varphi)- g(U,\varphi) \AAA \varphi }{1+g(U,\varphi)}.
\end{equation}

In order to ``select'' the solutions we are interested in, we point out that, since we need some uniformity with respect to $\de$
and $U=\de^{-2}(u-iA)$, then $\Re U \to \pm \infty$ as $\de \to 0$, depending on the sign of $\Re U$.
Then, according to~\eqref{eq:approxInnerw} and~\eqref{eq:approxInnerxy}, we deduce that $(W,X,Y) \to 0$ as $\Re U \to \pm \infty$.
For that reason, we look for $\ZdInn$ satisfying the asymptotic conditions
\begin{equation}\label{eq:asymptoticConditionsInner}
\begin{split}
\lim_{\Re U \to -\infty} \ZuInn(U) = 0, 
\qquad 
\lim_{\Re U \to +\infty} \ZsInn(U) = 0. 
\end{split}
\end{equation}
%
%
In fact, for a fixed $\beta_0 \in \paren{0,\frac{\pi}{2}}$, we look for functions $\ZuInn$ and $\ZsInn$ satisfying~\eqref{eq:invariantEquationInner},
\eqref{eq:asymptoticConditionsInner}  defined in the domains
\begin{equation}\label{def:domainInnner}
\begin{split}
&\DuInn = \claus{ U \in \complexs \text{ : }
|\Im U | \geq \tan \beta_0 \,\Re U + \rhoInn}, 
\qquad 
\DsInn = -\DuInn,
\end{split}
\end{equation}
respectively, for some $\rhoInn>0$ big enough (see Figure~\ref{fig:dominiInnerUnstable}).
\begin{figure}[t] 
	\centering
	\vspace{5mm}
	\begin{overpic}[scale=0.8]{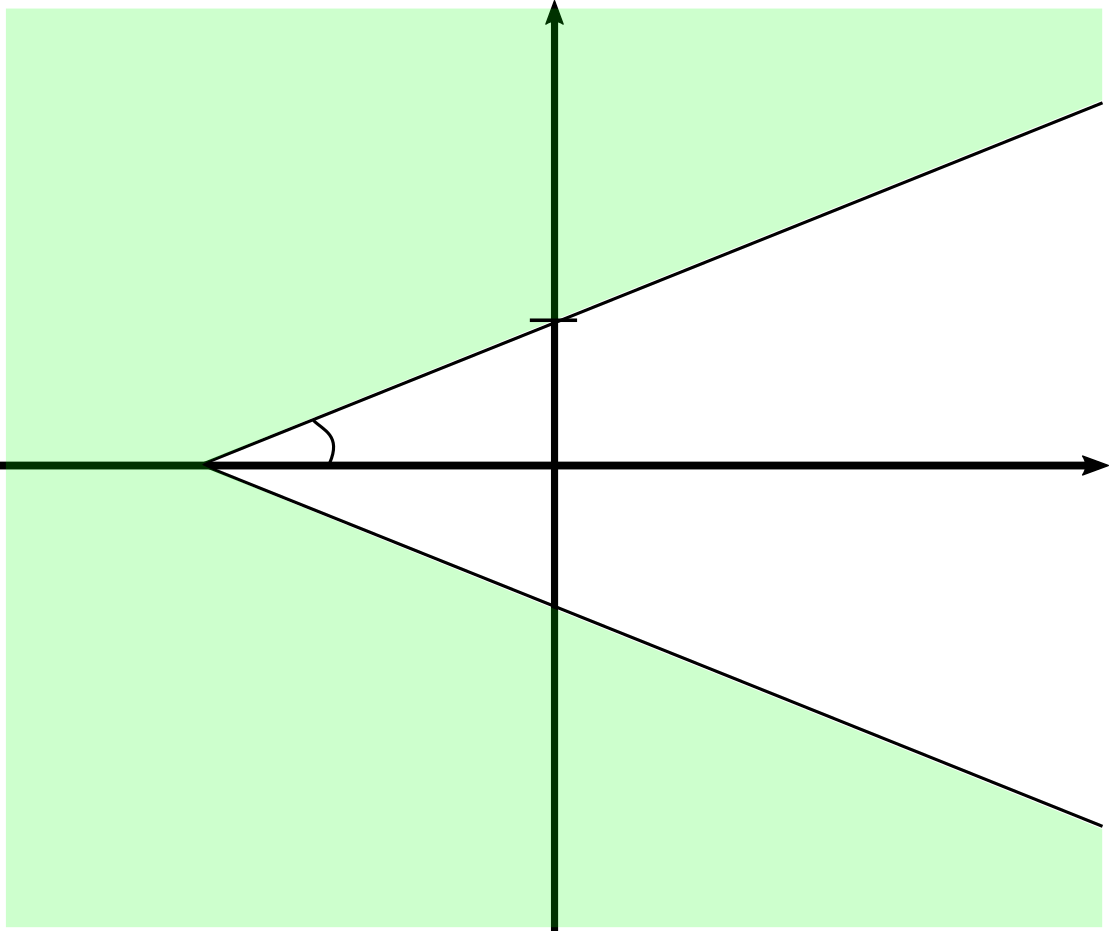}
		\put(10,60){$\DuInn$}
		\put(32.5,44){{$\beta_0$}}
		\put(45,55){{$\rhoInn$}}
		\put(102,41){$\Re U$}
		\put(47,87){$\Im U$}
	\end{overpic}
	\bigskip
	\caption{The inner domain, $\DuInn$, for the unstable case.}
	\label{fig:dominiInnerUnstable}
\end{figure} 
We analyze the  the difference $\DZInn=\ZuInn-\ZsInn$ in the overlapping domain
\begin{equation}\label{def:domainInnnerDiff}
\EInn = \DuInn \cap \DsInn \cap 
\claus{ U \in \complexs \st \Im U< 0}.
\end{equation}



\begin{theorem}\label{theorem:innerComputations}
	There exist $\rhoInn_0, \cttInnExist>0$
	such that for any $\rhoInn\geq\rhoInn_0$,	
	the equation~\eqref{eq:invariantEquationInner} has analytic solutions
	$
	\ZdInn(U) =(\WdInn(U),\XdInn(U),\YdInn(U))^T, 
	$ for $U \in \DdInn$, $\diamond=\unstable,\stable$, satisfying
	\begin{equation}\label{result:innerExistence}
	| U^{\frac{8}{3}} \WdInn(U)| \leq \cttInnExist, \qquad
	| U^{\frac{4}{3}} \XdInn(U) | \leq \cttInnExist, \qquad
	| U^{\frac{4}{3}} \YdInn(U) | \leq \cttInnExist.
	\end{equation}
	In addition, there exist $\CInn \in \complexs$ and $\cttInnDiff>0$ independent of $\rhoInn$, and a function $\chi=(\chi_1,\chi_2,\chi_3)^T$ 
	such that
	\begin{equation}\label{result:innerDifference}
	\DZInn(U) = \ZuInn(U)-\ZsInn(U) =
	\CInn e^{-iU} \Big(
	(0,0,1)^T + \chi(U) \Big),
	\end{equation}
	and, for $U \in \EInn$,
	\begin{align*}
	| U^{\frac{7}{3}} \chi_1(U)| \leq \cttInnDiff, \qquad
	| U^{2} \chi_2(U) | \leq \cttInnDiff, \qquad
	| U \chi_3(U) | \leq \cttInnDiff.
	\end{align*}
\end{theorem}


\begin{remark}\label{remark:CInnNumerics}
This theorem implies that ${\CInn} = \lim_{\Im U \to - \infty}{\DYInn(U) e^{iU}}$. Thus, we can obtain a numerical approximation of the constant $\CInn$. 
Indeed, for $\rho>\kappa_0$, we can define
\begin{equation}\label{def:CInnrho}
\CInn_{\rho} = \vabs{ \DYInn(-i\rho)} e^{\rho},
\end{equation} 
which, for $\rho$ big enough, satisfies $\CInn_{\rho} \approx \vabs{\CInn}$.

To compute $\DYInn(-i\rho)=\YuInn(-i\rho)-\YsInn(-i\rho)$,
we first look for good approximations of $\ZuInn(U)$ for $\Re U\ll -1$ and of $\ZsInn(U)$ for $\Re U \gg 1$, as power series in $U^{-\frac13}$.
One can easily check that $\ZuInn(U)$ as $\Re U \to -\infty$  and $\ZsInn(U)$ as $\Re U \to +\infty$ have the same asymptotics expansion:
\begin{align*}
\WdInn(U) &= 
\frac{4}{243 \, U^{\frac{8}{3}} } 
- \frac{172}{2187 \, U^{\frac{14}{3}}}
+ \OO \paren{U^{-\frac{20}{3}}}, \\
\XdInn(U) &= -\frac{2i}{9 \, U^{\frac{4}{3}} } 
+\frac{28}{81 \, U^{\frac{7}{3}} }
+ \frac{20i}{27 \, U^{\frac{10}{3}} }
- \frac{16424}{6561 \, U^{\frac{13}{3}} }
+ \OO \paren{U^{-\frac{16}{3}}}, \\
\YdInn(U) &= \frac{2i}{9 \, U^{\frac{4}{3}} } 
+\frac{28}{81 \, U^{\frac{7}{3}} }
- \frac{20i}{27 \, U^{\frac{10}{3}} }
- \frac{16424}{6561 \, U^{\frac{13}{3}} }
+ \OO \paren{U^{-\frac{16}{3}}}.
\end{align*}	
We use these expressions to set up the  initial conditions for the numerical integration  for computing $\DYInn(-i\rho)$.
We take as initial points  the value of the truncated power series at order $U^{-\frac{13}{3}}$ at $U= 1000 -i\rho$ (for $\diamond=\stable$) and $U= -1000 -i\rho$ (for $\diamond=\unstable$).
%
(See Table~\ref{table:CInnrho}). 
%
%
We perform the numerical integration for different values of  $\rho\leq 23$ and an  integration solver with tolerance $10^{-12}$.

Table~\ref{table:CInnrho} shows that the constant $\Theta$ is approximately 1.63 which indicates that it is not zero. We expect that this computation method can be implemented rigorously \cite{BCGS21}.

\begin{table}[H]
\begin{equation*}
\begin{array}{|c||c|c|c|}
	\hline
	\rho & \vabs{\DYInn(-i\rho)} & e^{\rho} & \CInn_{\rho}\\
	\hline 
	\hline
	13 & 3.7 \cdot 10^{-6} & 4.4 \cdot 10^{5} & 1.6373 \\
	\hline
	14 & 1.4 \cdot 10^{-6} & 1.2 \cdot 10^{6} & 1.6361 \\
	\hline
	15 & 5.0 \cdot 10^{-7} & 3.3 \cdot 10^{6} & 1.6351 \\
	\hline
	16 & 1.8 \cdot 10^{-7} & 8.9 \cdot 10^{6} & 1.6341 \\
	\hline
	17 & 3.7 \cdot 10^{-8} & 2.4 \cdot 10^{7} & 1.6333 \\
	\hline
	18 & 6.8 \cdot 10^{-8} & 6.6 \cdot 10^{7} & 1.6326 \\
	\hline
	19 & 9.1 \cdot 10^{-9} & 1.8 \cdot 10^{8} & 1.6320 \\
	\hline
	20 & 3.4 \cdot 10^{-9} & 4.9 \cdot 10^{8} & 1.6315 \\
	\hline
	21 & 1.2 \cdot 10^{-9} & 1.3 \cdot 10^{9} & 1.6312 \\
	\hline
	22 & 4.6 \cdot 10^{-10} & 3.6 \cdot 10^{9} & 1.6313 \\
	\hline
	23 & 1.7 \cdot 10^{-10} & 9.7 \cdot 10^{9} & 1.6323 \\
	\hline
\end{array}
\end{equation*}
\caption{Computation of $\CInn_{\rho}$, as defined in~\eqref{def:CInnrho}, for different values of $\rho\leq 23$.}
\label{table:CInnrho}
\end{table}

\end{remark}


\section{Analytic continuation of the separatrix}
\label{section:proofA-singularities}

In this section we prove Theorem~\ref{theorem:singularities} and Proposition~\ref{proposition:domainSeparatrix},  which deal with the study of the complex singularities and zeroes of the analytic extension of the time-parametrization $\s(t)=(\la_h(t),\La_h(t))$ of the homoclinic connection given in~\eqref{eq:separatrixParametrization}.

Let us recall that $\s(t)$ is a solution of the Hamiltonian system $H_{\pend}$ in~\eqref{def:HpendHosc} and it is found at the energy level ${H}_{\pend}=-\frac{1}{2}$. 
Therefore, 
\begin{equation}\label{eq:odeSeparatrix}
(\dot{\la})^2 = 
15 - 12\cos^2\paren{\frac{\la}{2}}
- \frac{3}{\cos\paren{\frac{\la}{2}}}.
\end{equation}
Equation~\eqref{eq:odeSeparatrix} can be solved as
$t=F(\la)$,
where $F$ is a function defined by means of an integral.
Prove Theorem~\ref{theorem:singularities} and Proposition~\ref{proposition:domainSeparatrix} boils down to studying the analytic continuation of $F^{-1}$.

We divide the proof of Theorem~\ref{theorem:singularities} into three main steps. 
First, in Section~\ref{subsection:changeq}, we perform the change of variables $q=\cos(\frac{\la}{2})$ and
rephrase Theorem~\ref{theorem:singularities} in terms of $q(t)$ (Theorem~\ref{theorem:singularitiesChangeq}).
Then, in Section~\ref{subsection:classificationSingularities}, we analyze all the possibles types of singularities that $q(t)$ may have  (Proposition~\ref{proposition:singularitiesSeparatrix}), which turn out to be poles or branching points.
In addition we prove that all the singularities have to be given by integrals along suitable complex paths.
Finally, in Section~\ref{subsection:computationSingularities},
taking into account all complex paths leading to  singularities, we prove that the singularities of $q(t)$ with smaller imaginary part (in the first Riemann sheet of $q(t)$) are $t=\pm iA$.

Finally, in Section~\ref{subsection:propZeroes}, we use the results obtained in the previous sections, to
analyze the zeroes of $\La_h(t)$ in the strip of analyticity $\Pi_A$ (see \eqref{def:PiA}), thus proving Proposition~\ref{proposition:domainSeparatrix}.

In order to simplify the notation, through the rest of the section we denote by $C$ any positive constant independent of $t$.

\subsection{Reformulation of Theorem~\ref{theorem:singularities}}
\label{subsection:changeq}

To prove Theorem~\ref{theorem:singularities}, it is more convenient to work with the variable $q=\cos\paren{\frac{\la}{2}}$ instead of $\la$.
Notice that this change of coordinates, when restricted to $\la \in (0,\pi)$, is a diffeomorphism. 
%

\begin{theorem} \label{theorem:singularitiesChangeq}
	Consider the real-analytic time parametrization $\s(t)=(\la_h(t),\La_h(t))$ introduced in~\eqref{eq:separatrixParametrization} and denote $a_\pm =-\frac{1}{2}\pm \frac{\sqrt{2}}{2}$.
	Then, $q(t)=\cos \paren{\frac{\la_h(t)}{2}}$ satisfies
	\begin{equation}\label{eq:basicsAboutq}
	q(t) \in [a_+,1) 
	\quad \text{for} \quad t \in  \reals, \qquad
	q(0) = a_+, 
	\end{equation}
	and the differential equation
	\begin{equation} \label{eq:odeSeparatrixQ}
	\qd^2 = \frac{3}{q} (q-1)^2 (q+1)(q-a_-)(q-a_+).
	\end{equation}
	
	Moreover, we have that:
	\begin{itemize}
	\item The function $q(t)$ extends analytically to the strip $\Pi_A$ defined in~\eqref{def:PiA}.
	\item The function $q(t)$ has only two singularities on $\partial \Pi_A$ at $t=\pm iA$. 
	\item There exists $\nu>0$ such that, for $t \in \complexs$ with $\vabs{t -iA}<\nu$ and $\arg(t-iA) \in (-\frac{3\pi}{2},\frac{\pi}{2})$,  we have
	\begin{align}\label{eq:expressionqSingularities}
		q(t) = - \frac{3\al_{+}}{2} (t -  iA)^{\frac{2}{3}} + 
		\OO(t - iA)^{\frac{4}{3}}, 
	\end{align}
	with $\al_{+} \in \complexs$ such that $\al_{+}^3=\frac{1}{2}$.

	An analogous result holds for
	$\vabs{t+iA}<\nu$ and $\arg(t + iA) \in (-\frac{\pi}{2},\frac{3\pi}{2})$
	with $\al_{-}=\conj{\al_+}$.
	\end{itemize} 
\end{theorem}


Theorem~\ref{theorem:singularities} is a corollary of Theorem~\ref{theorem:singularitiesChangeq}.

\begin{proof}[Proof of Theorem~\ref{theorem:singularities}]

To obtain Theorem~\ref{theorem:singularities} from Theorem~\ref{theorem:singularitiesChangeq} it is enough to prove that $\La_h(t)$ has no singularities in $\ol{\Pi_A} \setminus \claus{\pm iA}$ and that $(\la_h(t),\La_h(t))$ can be expressed as in~\eqref{eq:homoclinicaSingularities} close to $t=\pm iA$.
	
Since $\dot{\la}_h = -3\La_h$ and using the change of coordinates $q=\cos (\frac{\la}{2})$, we have that
\begin{equation}\label{proof:LahSquared}
\La_h^2(t) = \frac{4}{9}\frac{\dot{q}^2(t)}{\paren{1-q^2(t)}}.
\end{equation}	
We claim that, if $\La_h(t)$ has a singularity at $t=t^*$, then $\La_h^2(t)$ has a singularity at $t=t^*$ as well. 
Indeed, the only case when the previous affirmation could be false is if $t^*$ is a branching point of order $\frac{k}{2}$ with $k\geq 1$ an odd natural number. 
In this case,
\begin{align*}
\la_h(t) = \la_h(t^*) + C (t-t^*)^{\frac{k}{2}-1} \paren{1 + \OO(t-t^*)^{\beta}}, \quad
\text{ when } 0<\vabs{t-t^*} \ll 1,
\end{align*}
for some $\beta>0$.
Replacing this expression in~\eqref{eq:odeSeparatrix} and comparing orders we see that this case is not possible.

Thus, we proceed to prove that $\La_h(t)$ has no singularities in $\ol{\Pi_A} \setminus \claus{\pm iA}$. Let us assume it has. That is, there exists $t^* \in \ol{\Pi_A} \setminus \claus{\pm iA}$ such that $\La_h(t)$ is singular at $t=t^*$. Note that Theorem \ref{theorem:singularitiesChangeq} implies that $q(t)$ and $\dot q(t)$  are analytic in a neighborhood of $t^* \in \ol{\Pi_A} \setminus \claus{\pm iA}$.

\begin{enumerate}
	\item If $q^2(t^*)\neq 1$, $1/(1-q^2(t))$ is analytic for $0<\vabs{t-t^*} \ll 1$. Since $\dot q(t)$ is also analytic in this neighborhood, \eqref{proof:LahSquared} implies that $\La^2_h(t)$ has no singularity at $t=t^*$ and we reach a contradiction.
	%
	
	\item If $q^2(t^*)=1$, by \eqref{eq:odeSeparatrixQ} and~\eqref{proof:LahSquared}, we deduce that
	\begin{align}\label{eq:equationLahSimplified}
		\La_h^2 = \frac{4}{3 q} (1-q)(q-a_+)(q-a_-).
	\end{align}
	Since by Theorem~\ref{theorem:singularitiesChangeq} $q$ is analytic in $\ol{\Pi_A} \setminus \claus{\pm iA}$, then $\La^2_h$ must be as well. 
\end{enumerate}
Finally, we notice that, by  equations~\eqref{eq:expressionqSingularities} and~\eqref{proof:LahSquared}, we have
\begin{align*}
\La_h^2(t) = \frac{4}{9}\al^2_{\pm} (t\mp iA)^{-\frac{2}{3}} + \OO(1), \quad
\text{ when } 0<\vabs{t\mp iA} \ll 1.
\end{align*}
Therefore, $\La_h(t)$ has branching points of order $-\frac{1}{3}$ in $t=\pm iA$.
Moreover, integrating the expression for $\La_h(t)$ and 
applying that $q(t)=\cos (\frac{\la_h(t)}{2})$ (and \eqref{eq:expressionqSingularities}), it is immediate to see that $\la_h(t)$ has branching points of order $\frac{2}{3}$ at $t=\pm iA$ and can be expressed as in~\eqref{eq:homoclinicaSingularities} close to $t=\pm iA$.
\end{proof}

We devote Sections \ref{subsection:classificationSingularities} and \ref{subsection:computationSingularities} to prove Theorem~\ref{theorem:singularitiesChangeq}. 
The statements~\eqref{eq:basicsAboutq} and~\eqref{eq:odeSeparatrixQ} are straightforward by applying the change of coordinates $q=\cos(\frac{\la}{2})$ to equation~\eqref{eq:odeSeparatrix}. 

We divide the rest of the proof of Theorem~\ref{theorem:singularitiesChangeq} into two parts. 
In Section~\ref{subsection:classificationSingularities} we classify the singularities of $q(t)$ and introduce a way to compute them using integration in complex paths. 
Finally, in Section~\ref{subsection:computationSingularities} we prove that the singularities of $q(t)$ with smallest imaginary part are $t=\pm iA$ and are branching points of order $\frac{2}{3}$.

\subsection{Classification of the singularities of \texorpdfstring{$q(t)$}{q(t)}} \label{subsection:classificationSingularities}

Equation~\eqref{eq:odeSeparatrixQ} with initial condition $q(0)= a_+ =-\frac{1}{2}+\frac{\sqrt{2}}{2}$ is equivalent to 
\begin{align*}
t = \int_{a_+}^{q(t)} f(s) ds, 
\qquad \text{for }  t\in \reals, 
\end{align*}
where
\begin{equation*}
 f(q)=\frac{1}{q-1}\sqrt{\frac{q}{3(q+1)(q-a_+)(q-a_-)}},
\end{equation*}
is defined in $\reals \setminus \{ [a_-,-1] \cup (0,a_+] \cup \{1\} \} $
with $a_{\pm}=-\frac{1}{2}\pm\frac{\sqrt{2}}{2}$.
%
From~\cite{Fon95}, we know that there exist $\upsilon>0$ such that $q(t)$ can be extended to the open complex strip 
\begin{equation*}
\Pi_{\upsilon} = \{ t \in \complexs \st  |\Im t|< \upsilon \},
\end{equation*}
and  $q(t)$ has singularities in $\partial \Pi_{\upsilon}$. Namely,
\begin{align}\label{proof:integralReals}
t = \int_{a_+}^{q(t)} f(q) dq, 
\qquad \text{for }  t\in \Pi_{\upsilon}.
\end{align}

Since $f$ is a multi-valued function in the complex plane, in order to analyze the possible values of $\int_{a_+}^{q(t)} f(s) ds$, we consider its complete analytic continuation. That is,
\begin{align}\label{eq:fcompleteanalytic}
\begin{array}{rl}
\fh: \mathscr{R}_f &\to \complexs  \\
\Big(q;\arg \argf(q)\Big) &\mapsto
\displaystyle \frac{\sqrt{\argf(q)}}{q-1},
\end{array}
\quad \text{where} \quad 
\argf(q)=\frac{q}{3 (q+1)(q-a_+)(q-a_-)},
\end{align}
and $\mathscr{R}_f$ is the Riemann surface associated to $f$.
%
We define ${\proj:\mathscr{R}_f \to \complexs}$ as the projection to the complex plane.
We choose the first Riemann sheet to correspond to ${\arg \argf(q) \in (-\pi,\pi]}$. Accordingly the second Riemann sheet corresponds to $(\pi, 3\pi]$.

To integrate $\fh$ along a path $\g \subset \mathscr{R}_f$, we introduce the notation
\[
\int_{\g} \fh(q) dq=\int_{q_{\ini}}^{q_{\fin}} \fh d\g =\int_{s_{\ini}}^{s_{\fin}} \fh(\g(s)) \g'(s)ds ,
\]	
such that $\g:(s_{\ini},s_{\fin}) \to \mathscr{R}_f$ where $\lim_{s \to s_{\ini}} \proj \g(s)=q_{\ini}$ and $\lim_{s \to s_{\fin}} \proj \g(s)=q_{\fin}$.
Moreover, we assume that the paths $\g \subset \mathscr{R}_f$ are $\CC^0$ and $\CC^1\text{-piecewise}$.
Therefore, by~\eqref{proof:integralReals}, we have that
\begin{equation}\label{eq:lemmaIntegralEquationGeneral}
t = \int_{a_+}^{q(t)} \fh d \g,
\qquad \text{for }  t\in \Pi_{\upsilon}.
\end{equation}
Now, for $q\in\mathscr{R}_f$ and an integration path $\g \subset \mathscr{R}_f$,  we define the function $G$ as the right hand side of~\eqref{eq:lemmaIntegralEquationGeneral},
\[
G(q)= \int_{a_+}^{q} \hat{f} \h d\g.
\]
Notice that for a given $q\in\mathscr{R}_f$, $G(q)$ may depend on the integration path, and therefore,  $G$ may be multi-valued on $\mathscr{R}_f$.
However, by~\eqref{eq:lemmaIntegralEquationGeneral}, $G$ is single-valued when 
\[
t=G(q(t)), \qquad \text{for } t \in \Pi_{\upsilon}.
\]
We use $G$ to characterize and locate the singularities of $q(t)$.
Indeed, if function $G(q)$ is biholomorphic at $q=q^*$, then $q(t)$ is analytic at a neighborhood of all values of $t$ such that $q(t)=q^*$.
Therefore, $q(t)$ may have singularities at $t=t^*$ when the hypothesis of the Inverse Function Theorem are not satisfied for $G$.
That is, for $q(t^*)=q^*$ such that either
\begin{align}\label{eq:conditionsInverseTheo}
G'(q^*)= 0, \qquad
G \notin \CC^1 \text{ at }
q=q^*, 
\qquad \text{or} \quad
\vabs{q^*} \to \infty.
\end{align}
Namely, when there exist $q^*$ and $\g \subset \mathscr{R}_f$ satisfying~\eqref{eq:conditionsInverseTheo},
 such that 
\begin{equation}\label{eq:lemmaIntegralEquationSingularities}
t^* = G(q^*) = \int_{a_+}^{q^*} \hat{f} \h  d\g.
\end{equation}
Since $G$ is a multivalued function, the values of $t^*$ can, and in fact will, depend on the integration path on $\g \subset \mathscr{R}_f$.
From \eqref{eq:fcompleteanalytic} and \eqref{eq:conditionsInverseTheo}, one deduces that the singularities may take place only if $q(t^*)=q^*$ with
\begin{equation}\label{eq:candidatesSingularities}
q^*=0,1,-1, a_+, a_- 
\qquad \text{and} \qquad 
\vabs{q^*} \to \infty.
\end{equation}
The following proposition proves that we only need to consider $\vabs{q^*} \to \infty$ and $q^*=0$ .

\begin{proposition} \label{proposition:singularitiesSeparatrix}
 Let $q(t)$ be a solution of equation~\eqref{eq:odeSeparatrixQ} with initial condition $q(0)=a_+$. Then, the singularities $t^* \in \complexs$ of the analytic extension of $q(t)$ are characterized by  either
\begin{equation*} 
 t^* = \int_{a_+}^{0} \hat{f} \h d\g, 
 \hh  \text{ or } \hh
 t^* = \int_{a_+}^{\infty} \hat{f} \h d\g,
\end{equation*}
for some path $\g \subset \mathscr{R}_f$.

Moreover:
\begin{itemize}
\item If $t^*=\int_{a_+}^0 \fh d\g$
and $\Im t^*>0$ with 
$\arg \paren{t-t^*} \in \paren{-\frac{3\pi}{2},\frac{\pi}{2}}$, then
\begin{equation}
\label{eq:singularitatExpressio}
q(t) = -\frac{3\al}{2}  (t-t^*)^{\frac{2}{3}} + \OO(t-t^*)^{\frac{4}{3}}, \qquad
\text{for } \h 0<|t-t^*|\ll 1,
\end{equation}
where  $\al \in \complexs$ satisfies $\al^3=\frac{1}{2}$. 
If $\Im t^*<0$ and
$\arg\paren{t-t^*} \in \paren{-\frac{\pi}{2},\frac{3\pi}{2}}$, the same holds true.
\item If $t^*=\int_{a_+}^{\infty} \fh d\g$, 
then
\begin{equation}
\label{eq:singularitatExpressioInfinit}
q(t) = - \frac{1}{\sqrt{3}(t-t^*)}\left( 1 + \OO(t-t^*)\right),
 \qquad
 \text{for } \h 0<|t-t^*|\ll 1.
\end{equation}
\end{itemize}
\end{proposition}
\begin{proof}
To prove this result first we need to analyze all the possible values of $q^*$ that may lead to singularities, (see~\eqref{eq:candidatesSingularities}).
We will use the expressions of $t$ and $t^*$ given in~\eqref{eq:lemmaIntegralEquationGeneral} and~\eqref{eq:lemmaIntegralEquationSingularities}, respectively.

\begin{enumerate}
\item If $|q^*| \to  \infty$, we have
\[
t - t^* = \int^{q(t)}_{\infty} \hat{f} \h d\g.
\]
Then, since
\[
\fh(q) = \frac{1}{\sqrt{3} q^2} + \OO\paren{\frac{1}{q^3}}, \qquad \text{ for }\quad |q|\gg 1,
\]
we obtain
\[
t-t^* = -\frac{1}{\sqrt{3}q} + \OO\paren{\frac{1}{q^2}},
\]
which implies~\eqref{eq:singularitatExpressioInfinit}.
%
\item If $q^*=a_+$, we have that 
\begin{equation*}\label{proof:integralames}
t - t^* = \int^{q}_{a_+} \fh d\g.
\end{equation*}
The function $\fh$ can be written as 
\[
\hat{f}(q)= \frac{h_{a_+}(q)}{\sqrt{q-a_+}},
\]
for some function $h_{a_+}$ which is analytic and non-zero in a neighborhood of $a_+$. Then,  $h_{a_+}$ can be written as 
$h_{a_+}(q) = \sum_{k=0}^{\infty} c_k (q-a_+)^k$, with $c_0 \neq 0$
and,  for 
$0< \vabs{q-a_+} \ll 1$, we obtain from \eqref{proof:integralames}
\[
t-t^*=\sqrt{q-a_+} \sum_{k=0}^{\infty} \frac{c_k}{k+\frac{1}{2}} (q-a_+)^k,
\]
which implies $(t-t^*)^2=g_{a_+}(q)$ with 
\[
g_{a_+}(q)=(q-a_+)\left( \sum_{k=0}^{\infty} \frac{c_k}{k+\frac{1}{2}} (q-a_+)^{k} \right)^2.
\]
The function $g_{a_+}(q)$ is analytic on a neighborhood of $a_+$ and satisfies $g_{a_+}(a_+)=0$, $g'_{a_+}(a_+)=4 c_0^2 \neq 0$. 
Thus, applying the Inverse Function Theorem, 
\begin{equation}\label{proof:qexpressionZeroes}
q(t)=g_{a_+}^{-1}\left((t-t^*)^{2}\right),\qquad \text{for}\quad 0<\vabs{t-t^*}\ll 1.
\end{equation}
Therefore, $q(t)$ is analytic for $\vabs{t-t^*}\ll 1$.
%
%
One can analogously prove that the same happens at $q^*=a_-$ and $q^*=-1$. 

\item The value $q^*=1$ corresponds to the saddle point $(\la,\La)=(0,0)$ of $H_{\pend}$ (see~\eqref{def:HpendHosc}). Indeed, 
\begin{equation}\label{proof:classificationSaddle}
 \int^{q}_{1} \fh d\g\quad \text{is divergent.}
\end{equation}
This implies that $q(t)\neq 1$ for any complex $t$.
%

\item If $q^*=0$, 
\begin{equation*}
t - t^* = \int^{q}_{0} \fh \h d\g.
\end{equation*}
We can introduce $h_0(q)= \frac{1}{\sqrt{q}} \fh(q)$ which is analytic and of the form $h_0(q)=\sum_{k=0}^{\infty}c_k q^k$ with $c_0 \neq 0$. 
Then, for $0<\vabs{q} \ll 1$, 
we obtain
\begin{equation*}
(t-t^*)^{\frac{2}{3}} =g_0(q) = q \left(\sum_{k=0}^{\infty} \frac{c_k}{k+\frac{3}{2}} q^k \right)^{\frac{2}{3}}
= q \paren{\frac{2 c_0}{3} + \OO(q) }^{\frac{2}{3}},
\end{equation*}
%
where $g_0(q)$ is analytic in a neighborhood of $q=0$ and satisfies $g_0(0)=0$, $g'_0(0) = \paren{\frac{2}{3} c_0}^{\frac{2}{3}} \neq 0$.
Thus, applying the Inverse Function Theorem, 
\begin{equation}
\label{proof:seriesqEnZero}
q(t) = g_0^{-1}\paren{(t-t^*)^{\frac{2}{3}}}
= \sum_{k=1}^{\infty} C_k (t-t^*)^{\frac{2k}{3}},\quad \text{for } 0<|t-t^*|<\ll 1,
\end{equation}
for some $C_k \in \complexs$ and choosing  the Riemann sheet $\arg \paren{t-t^*} \in \paren{-\frac{3\pi}{2},\frac{\pi}{2}}$ for $\Im t^*>0$
and  $\arg \paren{t-t^*} \in \paren{-\frac{\pi}{2},\frac{3\pi}{2}}$
for $\Im t^*<0$.
%
Replacing~\eqref{proof:seriesqEnZero} in equation~\eqref{eq:odeSeparatrixQ} we obtain
$
C_1 = -\frac{3\al}{2},
$
where $\al \in \complexs$ satisfies $\al^3=\frac{1}{2}$, which implies~\eqref{eq:singularitatExpressio}.
\end{enumerate}
\end{proof}

\subsection{Singularities closest to the real axis}
\label{subsection:computationSingularities}

Proposition~\ref{proposition:singularitiesSeparatrix} provides the type of singularities that $q(t)$ may posses in its first Riemann sheet.
Then, to prove Theorem~\ref{theorem:singularitiesChangeq}, we look for those singularities which are closest to the real axis. 
To do so, we analyze
\begin{equation}\label{eq:integralsIniciSeccio}
\int_{a_+}^{0} \fh  d\g, \qquad
\int_{a_+}^{\infty} \fh  d\g,
\end{equation}
along all paths $\g \subset \mathscr{R}_f$ with such endpoints
and prove that, for all possible paths, the only singularities in the complex strip $\overline{\Pi_A}$ (see~\eqref{def:PiA})  are $t=\pm iA$.
%

We introduce the following paths
\begin{equation*}
\begin{split}
&\PP_0 = \claus{ \g: (s_{\ini},s_{\fin}) \to \mathscr{R}_f \st
\lim_{s \to s_{\ini}} (\proj \g(s), \arg \argf( \g(s)))=(a_+,0), \,
\lim_{s \to s_{\fin}} \proj \g(s)= 0}, \\
&\PP_{\infty} = \claus{ \g: (s_{\ini},s_{\fin}) \to \mathscr{R}_f \st
\lim_{s \to s_{\ini}} (\proj \g(s), \arg \argf( \g(s)))=(a_+,0), \h
\lim_{s \to s_{\fin}}\vabs{\proj \g(s)}= \infty},
\end{split}
\end{equation*}
with the natural projection $\proj:\mathscr{R}_f\to \complexs$.

We have chosen $\tht_{\ini} := \lim_{s \to s_{\ini}} 
\arg \argf( \g(s)) = 0$ without loss of generality. Indeed, since $q(t) \in [a_+,1)$ for $t \in \reals$ (see~\eqref{eq:basicsAboutq}), it could be either $0$ or $2\pi$. The paths with asymptotic argument $2\pi$ can be analyzed analogously and lead to singularities with opposed signed with respect to those given by paths in $\PP_0$, $\PP_\infty$.
%

Furthermore, the asymptotic argument of the paths at its endpoints is not specified since it is given by the path itself.

For a given path $\g \in \PP_0 \cup \PP_{\infty}$,
we can define a path $T[\g]:[s_{\ini},s_{\fin}) \to \complexs$ in the $t$-plane (or, more precisely, on the Riemann surface of $q(t)$) as 
\begin{equation}\label{def:operatorT}
T[\g](s)=  
\int_{s_{\ini}}^{s} \fh(\g(\tau)) \g'(\tau) d \tau, \quad
\text{ for } s \in [s_{\ini},s_{\fin}).
\end{equation}
Note that then the value of the integrals in~\eqref{eq:integralsIniciSeccio} is just
\begin{equation}\label{def:operatorSingularities}
t^*(\g) =
\lim_{s \to s_{\fin}} T[\g](s) =
\int_{\g} \fh(q) dq.
\end{equation}
%
%
Since we are interested in the singularities of $q(t)$ on its first Riemann sheet, we only consider the paths $T[\g]$ which belong to the complex strip $\Pi_A$
(Except, of course, the endpoint of the path $t^*(\g)\in\partial \Pi_A$). 
See Figure~\ref{fig:admisiblePaths}.


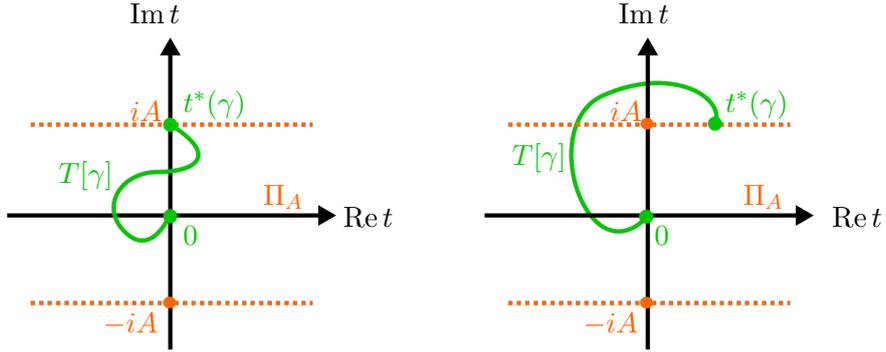
\begin{figure}[H] 
	\centering
	\begin{overpic}[scale=0.7]{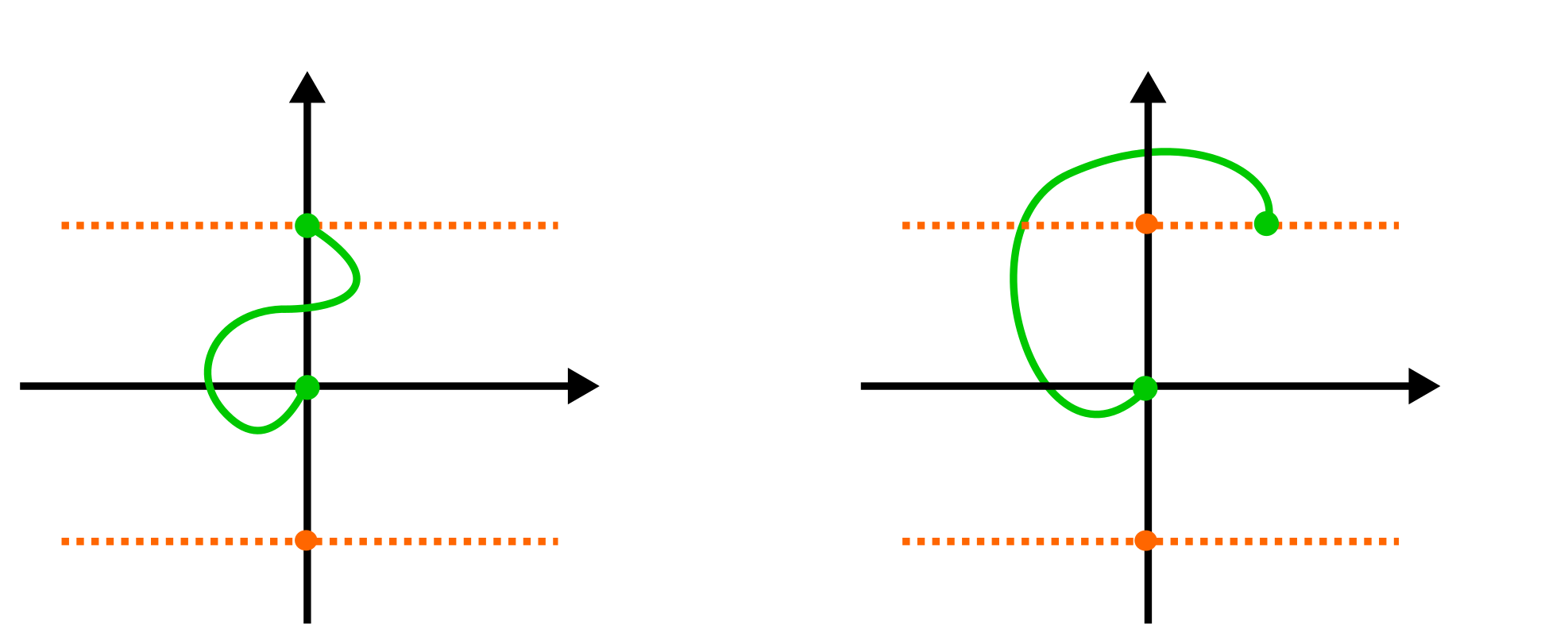}
		\put(21,13){{\color{myGreen} $0$ }}
		\put(74,13){{\color{myGreen} $0$ }}
		\put(15,27){{\color{myOrange} $iA$ }}
		\put(69,27){{\color{myOrange} $iA$ }}
		\put(21,28){{\color{myGreen} $t^*(\g)$ }}
		\put(82,28){{\color{myGreen} $t^*(\g)$ }}
		\put(12,3){{\color{myOrange} $-iA$ }}
		\put(66,3){{\color{myOrange} $-iA$ }}
		\put(30,17.5){{\color{myOrange} $\Pi_A$ }}
		\put(84,17.5){{\color{myOrange} $\Pi_A$ }}
		\put(39,15){$\Re t$}
		\put(94,15){$\Re t$}	
		\put(15,38){$\Im t$}		
		\put(70,38){$\Im t$}
		\put(7,20){{\color{myGreen} $T[\g]$ }}
		\put(58,22){{\color{myGreen} $T[\g]$ }}	
	\end{overpic}
	\caption{Example of paths $\g \in \PP_{0}\cup \PP_{\infty} $. 
	Left: $T[\g]\subset \Pi_A$ and $t^*(\g)=iA$.  
	Right: $T[\g] \not\subset {\Pi_A}$.}
	\label{fig:admisiblePaths}
\end{figure}

The following definition  characterizes the paths that we consider.
\begin{definition}\label{definition:visibleSingularity}
We say that a singularity $t^*$ of $q(t)$ is \emph{visible} if there exists a path $\g \in \PP_{0} \cup \PP_{\infty}$ such that
\begin{itemize}
	\item $t^*=t^*(\g)$,
	\item $T[\g](s) \in \Pi_A, \,$ for  $s\in[s_{\ini},s_{\fin})$.
\end{itemize}
\end{definition}

\begin{remark}
	In~\cite{SimoGelf08}, the authors use a different definition of visible singularity: $t^* \in \complexs$ is considered a visible singularity if $q(t)$ can be continued from the real axis and then along the vertical line with a path of the form
	\[
	\zeta(t) = \Re t^* + i t, \qquad \text{for} \quad  t \in [0,\Im{t^*}).
	\]
	This condition on the paths is more restrictive than merely imposing that $T[\g] \subset \Pi_{\upsilon}$ for $\upsilon=\Im t^*$. 
	However, to compute $t^*(\g)$, they are equivalent since both paths belong to $\Pi_{\upsilon}$.
	%
\end{remark}

Theorem~\ref{theorem:singularitiesChangeq} is a consequence of Proposition~\ref{proposition:singularitiesSeparatrix} and the following result.
\begin{proposition}\label{proposition:camins}
There exist two paths $\g_{\pm} \in \PP_{0}$ yielding the visible singularities  $t^*(\g_{\pm})=\pm iA$.
Moreover, these are the only two visible singularities of $q(t)$.
\end{proposition}

We devote the rest of this section to prove Proposition~\ref{proposition:camins}. 
Let us introduce some tools and considerations to simplify  the analysis of the integrals in~\eqref{eq:integralsIniciSeccio}.
\begin{itemize}
\item If $\g \subset \mathscr{R}_f$ is an integration path then, defining $\eta=\conj{\g}$\footnote{
We define the conjugation on a Riemann surface as the natural continuation of the conjugation in the complex plane. That is, for $z=(x,\tht) \in \mathscr{R}_f$, its conjugated is  $\conj{z}=(\conj{x},-\tht)\in \mathscr{R}_f$.}
, we have that
\begin{equation}\label{eq:reflection}
\int_{\eta} \fh(q) dq = \conj{\int_{\g} \fh(q) dq}.
\end{equation}
\item Notice that the paths considered cannot contain the singularities of $\fh$ (except in their endpoints) since $0 ,a_{\pm}, \pm 1 \notin \mathscr{R}_f$.
\item When saying that a path $\g \in \PP_{0} \cup \PP_{\infty}$ crosses $\reals$, we refer to the two lines whose complex projection onto $\mathscr{R}_f$ coincide with $\reals$. Analogously for any other interval.
\item Instead of detailing the paths $\g  \in \PP_{0} \cup \PP_{\infty}$, 
we only describe their projections $\proj \g \subset \complexs$.
This omission makes sense since paths $\g$ are continuous on $\mathscr{R}_f$ and, as a result, $\arg \argf(\g)$ must be continuous as well (see~\eqref{eq:fcompleteanalytic}).
Therefore, we can let the natural arguments of $\proj \g$ and the initial point $(a_+,0)$ of the path define $\arg \argf(\g)$. 
\end{itemize}

To prove Proposition~\ref{proposition:camins}, we classify the paths as follows.
\begin{enumerate}[label*=\Alph*.]
	\item Paths not crossing $\reals$:
	\begin{enumerate}[label*=\arabic*.]
		\item Paths in $\PP_{0}$ not crossing $\reals$: Lemma~\ref{lemma:caminsZeroSimple}.
		\item Paths in $\PP_{\infty}$ not crossing $\reals$: Lemma~\ref{lemma:caminsInfinitSimple}.
	\end{enumerate}
	\item Paths first crossing the real axis at $\reals \setminus [0,1]$:
	\begin{enumerate}[label*=\arabic*.]
		\item First crossing of $\reals$ at $(1,+\infty)$:  Lemma~\ref{lemma:caminsPol}.
		\item First crossing of $\reals$ at $(-\infty,a_-)$: Lemma~\ref{lemma:caminsNegatius2}.
		\item First crossing of $\reals$ at $(-1,0)$: Lemma~\ref{lemma:caminsNegatius1}.
		\item First crossing of $\reals$ at $(a_-,-1)$: Lemma~\ref{lemma:caminsBranca2}.
	\end{enumerate}
	\item Paths first crossing the real axis at $(0,1)$: 
	\begin{enumerate}[label*=\arabic*.]
		\item Paths in $\PP_{0}$ only crossing $\reals$ at $(0,1)$: Lemma~\ref{lemma:caminsBranca1a}.
		\item Paths in $\PP_{\infty}$  only crossing $\reals$ at $(0,1)$: Lemma~\ref{lemma:caminsBranca1b}.
		\item Paths also crossing $\reals \setminus [0,1]$: Lemma~\ref{lemma:caminsBranca1c}.
	\end{enumerate}
\end{enumerate}

\subsubsection{Paths not crossing the real axis}
In this section, we check the singularities resulting from the paths A.1 and A.2. 

\begin{lemma} \label{lemma:caminsZeroSimple}
There exist only two singularities, $t^*_{1,\pm}$, given by the paths $\g \in \PP_{0}$ not crossing the real axis. These singularities are  visible and
\[
 t^*_{1,\pm} = \mp iA,
\]
with $A$  defined in~\eqref{def:integralA} and satisfying $A \in  \left[\frac{3}{50}, \frac{3}{10}\right]$.
\end{lemma}

\begin{proof}
Let us consider paths $\g \in \mathcal{P}_0$ such that $\proj \g \subset \complexs^+ = \claus{\Im z>0}$. 
Notice that, since $\g$ does not cross the real axis, it does not encircle any singularity of $\fh$ and, by Cauchy's Integral Theorem, all the paths considered generate the same singularity $t_{1,+}^*$. 
The singularity $t^*_{1,-}$ is given by the conjugated paths (see~\eqref{eq:reflection}).

Let us consider the path $\g_*=\g_{*}^1 \vee \g_{*}^2 \vee \g_{*}^3$ with
\begin{equation}\label{proof:pathgamma1}
\begin{cases}
\g_{*}^1(q) = (q,0) & \text{ with } q \in (a_+,a_+ + \e], \\
\proj \g_{*}^2(\phi) = a_+ + \e e^{i\phi} & \text{ with } \phi \in [0,\pi], \\
\proj \g_{*}^3(q) = q  & \text{ with } q \in [a_+ -\e,0), \\
\end{cases}
\end{equation}
for $\e>0$ small enough,
(see Figure~\ref{fig:caminsZeroSimpleSuperior}).
\begin{figure}
	\centering
	\begin{overpic}[scale=0.8]{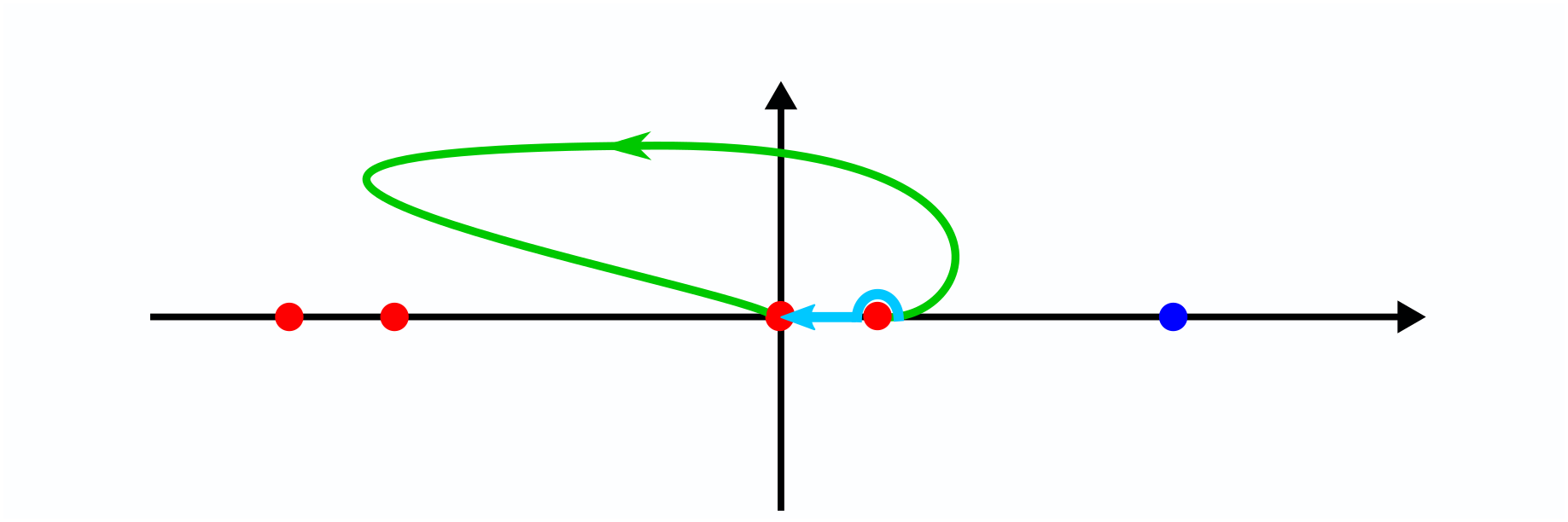}
		\put(48,9.5){{\color{red} $0$ }}	
		\put(56,9.5){{\color{red} $a_+$ }}
		\put(75,9.5){{\color{blue} $1$ }}
		\put(25,9.5){{\color{red} $-1$ }}
		\put(17,9.5){{\color{red} $a_-$ }}
		\put(93,12){$\Re q$}
		\put(47,30){$\Im q$}
		\put(39,26){{\color{myGreen} $\proj \g$ }}	
		\put(52,16){{\color{myBlue} $\proj \g_*$ }}
	\end{overpic}
	\caption{Example of a path $\g \in \PP_{0}$ such that $\proj \g \subset \complexs^+$. Path $\g_*$ as defined in~\eqref{proof:pathgamma1}.}
	\label{fig:caminsZeroSimpleSuperior}
\end{figure}
%
%
Then, the resulting singularity is
\[
t^*_{1,+} 
= t^*(\g_*) 
= \int_{\g_*} \fh(q) dq
= \sum_{j=1}^3 \int_{\g_{*}^j} \fh(q) dq.
\]
%
%
Since $\int_{\g_*^j} \fh(q) dq = \OO(\sqrt{\e})$ for $j=1,2$,  taking the limit $\e \to 0$, we have
\[
t_{1,+}^* = \lim_{\e \to 0} \int_{\g_*^3} \fh(q) dq.
\]
Then, by following the natural arguments of the path $\g_*$, we obtain
\[
\begin{array}{llll}
\arg(\g_{*}^3) = 0, & 
\arg(\g_{*}^3-a_+) = \pi, &
\arg(\g_{*}^3+1) = 0, & 
\arg(\g_{*}^3-a_-) = 0,
\end{array} 
 \]
and, as  a consequence, by the definition of $A$ in~\eqref{def:integralA}, we have
\begin{equation}\label{proof:singularityt1mes}
 t^*_{1,+}  = 
  \lim_{\e \to 0} \int_{\g_*^3} \fh(q) dq =
 \int_{a_+}^0 \frac{1}{x-1}\sqrt\frac{x}{3(x+1) \vabs{x-a_+} e^{i\pi}(x-a_-)} dq  = -iA.
\end{equation}
%
Moreover,
\begin{align*}
	A \leq & \frac{\sqrt{a_+}}{(1-a_+)\sqrt{3\vabs{a_-}}} 
	\int_0^{a_+} \frac{dx}{\sqrt{a_+-x}}
	= \frac{4\sqrt{3}}{3}\frac{a_+^{3/2}}{1-a_+}
	\leq \frac{3}{10}, \\
	A \geq & \frac{1}{\sqrt{3a_+(a_+ + 1)(a_+ -a_-)}}
	\int_0^{a_+} \sqrt{x} \, dx 
	= \frac{2 \sqrt[4]{74}}{9} {a_+^{3/2}}
	\geq \frac{3}{50}.
\end{align*}
Now, it just remains to see that $t_{1,+}^*$ is visible, namely, we check that $T[\g_*] \subset \Pi_A$. 
Indeed, for $s \in (0,a_+]$, we have that
\[
\vabs{\Im T[\g_*](s)} =  
\int_{s}^{a_+} \frac{1}{1-x}\sqrt\frac{x}{3(x+1)(a_+-x)(x-a_-)} dx
< A.
\]
\end{proof}

\begin{lemma} \label{lemma:caminsInfinitSimple}
	There exist only two singularities, $t_{2,\pm}^*$, resulting from the paths ${\g \in \PP_{\infty}}$ not crossing the real axis. These singularities are not visible and have imaginary part
	\[
	\Im (t_{2,\pm}^*) = \mp \pi \sqrt{\frac{2}{21}}.
	\]
\end{lemma}

\begin{proof}
Let us consider paths $\g \in \PP_{\infty}$ such that $\proj \g \subset \complexs^+$. Then, since $\fh(q)$ decays with a rate of $\vabs{q}^{-2}$ as $|q| \to \infty$,  all paths considered generate the same singularity, $t^*_{2,+}$.
The singularity $t^*_{2,-}$ is given by the conjugated paths.

Let us consider the path $\g_*=\g_{*}^1 \vee \g_{*}^2 \vee \g_{*}^3$ where
\begin{equation}\label{proof:pathgamma2}
\begin{cases}
\g_{*}^1(q) = (q,0) & \text{ with } q \in (a_+,1-\e], \\
\proj \g_{*}^2(\phi) = 1 + \e e^{i\phi} & \text{ with } \phi \in [\pi,0], \\
\proj \g_{*}^3(q) = q  & \text{ with } q \in [1+\e,+\infty), \\
\end{cases}
\end{equation}
for any small enough $\e>0$. (See Figure~\ref{fig:caminsInfinitSimpleSuperior}).
Then, the resulting singularity is
\[
t_{2,+}^* = t^*(\g_*) 
= \sum_{j=1}^3
\int_{\g_{*}^j} \fh(q) dq .
\]
\begin{figure}[H] 
	\centering
	\begin{overpic}[scale=0.8]{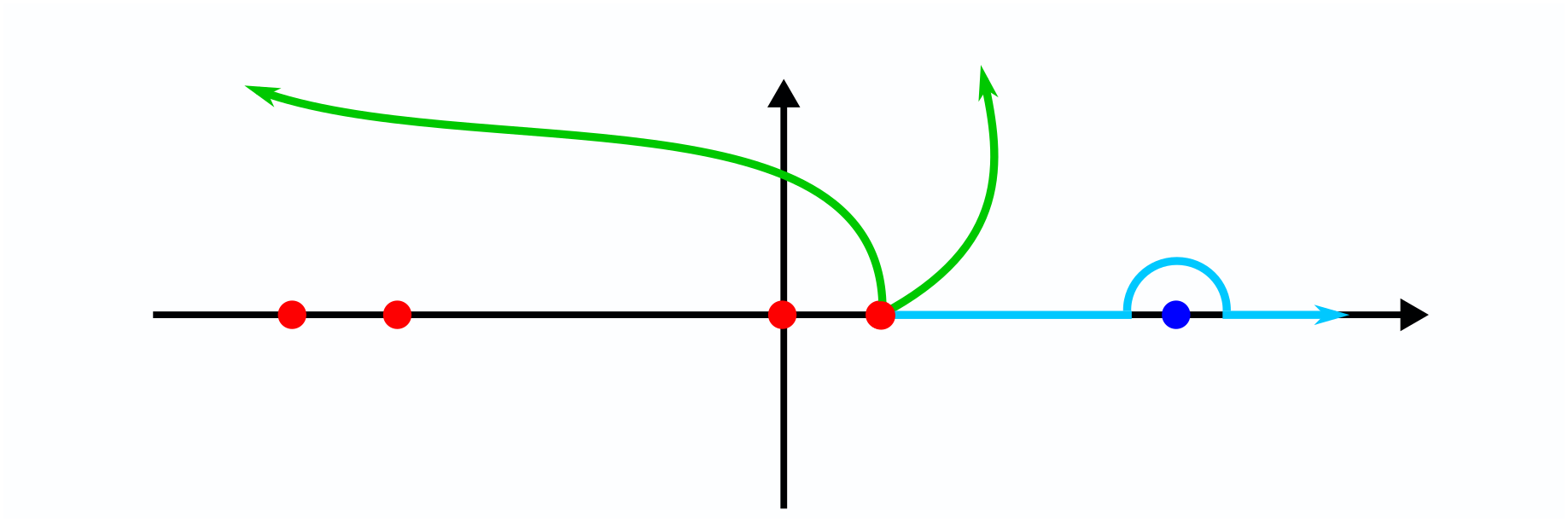}
		\put(48,9.5){{\color{red} $0$ }}	
		\put(56,9.5){{\color{red} $a_+$ }}
		\put(75,9.5){{\color{blue} $1$ }}
		\put(25,9.5){{\color{red} $-1$ }}
		\put(17,9.5){{\color{red} $a_-$ }}
		\put(93,12){$\Re q$}
		\put(47,30){$\Im q$}
		\put(56,20){{\color{myGreen} $\proj\g$ }}	
		\put(66,16){{\color{myBlue} $\proj \g_*$ }}
	\end{overpic}
	\caption{Example of a path $\g \in \PP_{\infty}$ such that $\proj \g \subset \complexs^+$. Path $\g_*$ as defined in~\eqref{proof:pathgamma2}.}
	\label{fig:caminsInfinitSimpleSuperior}
\end{figure}
Since $\fh(q)\in \reals$ when $\proj(q) \in (a_+,1)\cup(1,+\infty) \subset \reals$, the integrals on $\g_{*}^1$ and $\g_{*}^3$ take real values.
Therefore, $\g_{*}^2$ is the only path that contributes to the imaginary part to the singularity.
Notice that the path $\g_*^2$ partially encircles the pole $q=(1,0)$ of $\fh(q)$. Then, one has
\[
\Im (t_{2,+}^*)  = 
\Im \int_{\g_{*}^2} \fh(q) dq = 
- \pi \Res \paren{\fh,(1,0)} =
-\pi \sqrt{\frac{2}{21}}.
\]
Since, by Lemma \ref{lemma:caminsZeroSimple},  $|\Im(t^*_{2,+})|>A$, the singularity $t^*_{2,+}$ is not visible.
\end{proof}

\begin{remark} \label{remark:quatreSingularitiest1}
Using mathematical software, one can see that the singularities $t^*_{2,\pm}$ in Lemma~\ref{lemma:caminsInfinitSimple} satisfy
$t^*_{2,\pm} \approx -0.086697 \mp 0.969516i$.
\end{remark}

\subsubsection{Paths first crossing the real axis at \texorpdfstring{$\reals \setminus [0,1]$}{outside [0,1]}}
\label{subsubsection:pathsB}

In this section, we continue the proof of Proposition~\ref{proposition:camins} by checking that
the singularities generated by paths B.1 to B.4 (see the list in Section~\ref{subsection:computationSingularities}) are not visible.

First, we introduce some concepts. Let us consider a path $\g \in \PP_{0}\cup \PP_{\infty}$. 
We define the \emph{parameter of the first crossing} of the real line as
\[
 s_1(\g)= \inf \{s \in(s_{\ini},s_{\fin}) \st \Im \proj \g(s)=0 \},
\]
the \emph{location of the first crossing} as
\[
q_1(\g)= \proj \g(s_1(\g)) \in \reals \setminus \{ a_-, -1, 0, a_+, 1\},
\]
%
the \emph{piece of the path} before the first crossing of the real line as
\[
\g_1(\g) = \claus{\g(s) \st s \in (s_{\ini},s_1(\g))},
\]
and the \emph{time of the first crossing} as
\begin{equation*}
t_1(\g) = \int_{\g_1(\g)} \fh(q) dq.
\end{equation*}
In the following lemmas, we focus on the paths that stay in $\complexs^+$ until the first crossing of the real line, that is $\proj \g_1(\g) \subset \complexs^+$
(see~\eqref{eq:reflection} for the conjugate paths, i.e.  $\proj \g_1(\g) \subset \complexs^-$). 
%
%
%
\begin{remark}\label{remark:pathsCrossingReals}
	To prove that a singularity $t^*(\g)$ is not visible (see Definition~\ref{definition:visibleSingularity}) it is sufficient to check that $\vabs{\Im t_1(\g)} \geq A$.
	%
\end{remark}

\begin{lemma} \label{lemma:caminsPol}
The singularities $t^*(\g)$ given by paths $\g \in \PP_{0} \cup \PP_{\infty}$ such that 
$q_1(\g) \in (1,+\infty)$ are not visible.
\end{lemma}

\begin{proof}
Consider a path $\g \in \PP_{0} \cup 
\PP_{\infty}$ with $q_1=q_1(\g) \in (1,+\infty)$ and $\proj \g_1(\g)\subset \complexs^+$. 
%
%
%
%
Integrating the function $\fh$ along the path $\g_1(\g)$ is equivalent to integrate $\fh$ along the path
 $ {\eta}={\eta}^{1} \vee {\eta}^{2} \vee \eta^{3}$ where
\begin{equation}\label{proof:pathEtaPol}
\begin{cases}
 {\eta}^{1}(q) = (q,0) & \text{ with } q \in (a_+,1-\e], \\
\proj{\eta}^{2}(\phi) = 1 + \e e^{i\phi} & \text{ with } \phi \in [\pi,0], \\
\proj{\eta}^{3}(q) = q  & \text{ with } q \in [1+\e,q_1),
\end{cases}
\end{equation}
for  $\e>0$ small enough,
(see Figure~\ref{fig:caminsVoltesPol}).
\begin{figure}[b]
\centering
\begin{overpic}[scale=0.8]{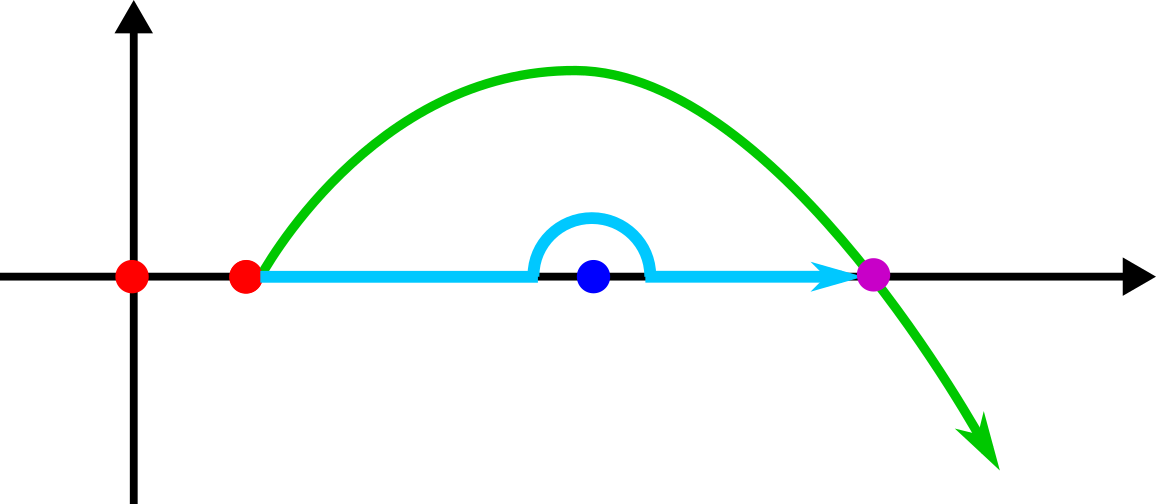}
	\put(8,14){{\color{red} $0$ }}
	\put(51,14){{\color{blue} $1$ }}	
	\put(18,14.5){{\color{red} $a_+$ }}
	\put(77,22){{\color{myPurple} $q_1(\g)$ }}
	\put(24,31){{\color{myGreen} $\proj \g$ }}
	\put(40,22){{\color{myBlue} $\proj \eta$ }}
	\put(101,18){$\Re q$}
	\put(7,44.5){$\Im q$}
\end{overpic}
	\caption{Example of a path $\g \in \HH_{q,\infty}$  such that ${q_1(\g) \in (1,+\infty)}$
	and $\proj \g_1(\g)\subset \complexs^+$. 
	The path $\eta$ has been defined in~\eqref{proof:pathEtaPol}.}
	\label{fig:caminsVoltesPol}
\end{figure}
Then,
\[
t_1(\g) = \int_{\eta} \fh(q) dq 
= \sum_{j=1}^3 \int_{\eta^j} \fh(q) dq,
\]
and the integrals on ${\eta}^{1}$ and ${\eta}^{3}$ take real values since $\fh(q)\in \reals$ when $\proj(q) \in (a_+,1)\cup(1,+\infty) \subset \reals$. Analogously to the proof of Lemma~\ref{lemma:caminsInfinitSimple}, we have that
\[
|\Im t_1 (\g)| = \left|\Im \int_{{\eta}^{2}} \fh(q) dq \right|=
 \pi \left|\Res\paren{\fh,(1,0)}\right| =
 \pi \sqrt{\frac{2}{21}} > A.
\]
Therefore, $t^*(\g)$ is not visible (see Remark~\ref{remark:pathsCrossingReals}).
\end{proof}

\begin{lemma} \label{lemma:caminsNegatius2}
The singularities $t^*(\g)$ given by paths $\g \in \PP_{0} \cup \PP_{\infty}$ such that 
$q_1(\g) \in (-\infty,a_-)$ are not visible.
\end{lemma}

\begin{proof}
Consider a path $\g \in \PP_{0} \cup \PP_{\infty}$ with $q_1=q_1(\g) \in (-\infty,a_-)$ and $\proj \g_1(\g) \subset \complexs^+$. 
%
%
To compute $t_1(\g)$ we introduce the auxiliary path
$
{\eta}_{\infty} = \g_1(\g) \vee \wt{\eta} 
$,
where
\begin{equation}\label{proof:pathEtaWT}
\proj \wt{\eta}(q)=q \qquad
\text{with }
q \in [q_1, -\infty),
\end{equation}
(see Figure~\ref{fig:caminsNegatius2}).
\begin{figure}
\centering
\begin{overpic}[scale=0.8]{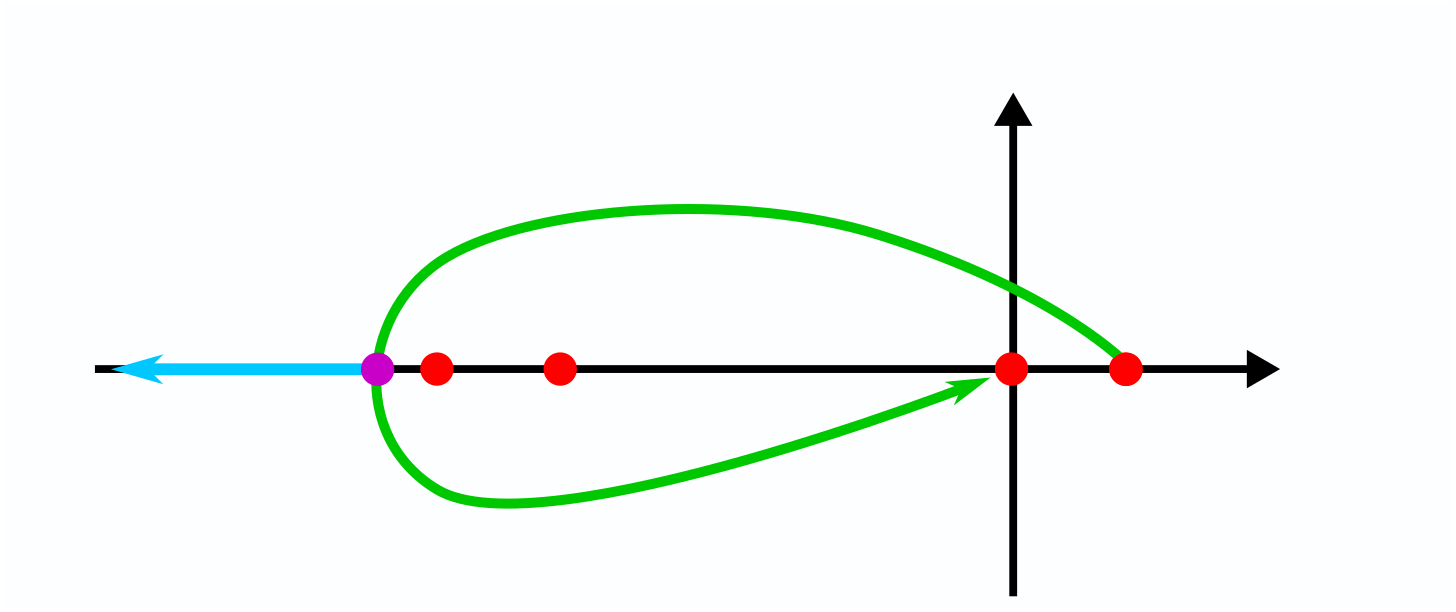}
	\put(71,12){{\color{red} $0$ }}	
	\put(76.5,12){{\color{red} $a_+$ }}
	\put(35,12){{\color{red} $-1$ }}
	\put(28,12){{\color{red} $a_-$ }}
	\put(89,15){$\Re q$}
	\put(66,37){$\Im q$}
	\put(41,29){{\color{myGreen} $\proj \g$ }}
	\put(16.5,12){{\color{myPurple} $q_1(\g)$ }}
	\put(16,19){{\color{myBlue} $\proj \wt{\eta}$ }}
\end{overpic}
\caption{Example of a path $\g \in \PP_{0}$  	
	such that ${q_1(\g) \in (-\infty,a_-)}$ 
	and $\proj \g_1(\g)\subset \complexs^+$. 
	The path $\wt{\eta}$ has been defined in~\eqref{proof:pathEtaWT}.}
\label{fig:caminsNegatius2}
\end{figure}
Then, taking into account that $\fh|_{\wt{\eta}} \subset \reals$,
\begin{equation*}
\vabs{\Im  t_1(\g)}=\left|
\Im \int_{\eta_{\infty}} \fh(q) dq \right|= 
\left|\Im \int_{\g_*} \fh(q) dq \right|=
 \pi \sqrt{\frac{2}{21}}>A,
\end{equation*}
where   $\g^*$ is the path  defined in~\eqref{proof:pathgamma2}.
%
%
Therefore $t^*(\g)$ is not visible.
\end{proof}

\begin{lemma} \label{lemma:caminsNegatius1}
The singularities $t^*(\g)$ given by paths $\g \in \PP_{0} \cup \PP_{\infty}$ such that 
$q_1(\g) \in (-1,0)$ are not visible.
\end{lemma}

\begin{proof}
Let $\g \in \PP_{0} \cup \PP_{\infty}$ be a path such that $q_1 = q_1(\g) \in (-1,0)$ and $\proj \g_1(\g) \subset \complexs^+$.
%
%
Integrating the function $\fh$ along the path $\g_1(\g)$ is equivalent to integrating  $\fh$ along
${\eta}={\eta}^1 \vee {\eta}^2 \vee \eta^3
\vee \eta^4 \vee \eta^5$
where
\begin{equation}\label{proof:pathEtaNegatius1}
\begin{cases}
\eta^1(q) = (q,0) & 
\text{ with } q \in (a_+,a_+ + \e], \\
\proj \eta^2(\phi) = a_+ + \e e^{i\phi} & 
\text{ with } \phi \in [0,\pi], \\	
\proj \eta^3(q) = q & 
\text{ with } q \in [a_+-\e,\e], \\
\proj \eta^4(\phi) = \e e^{i\phi} & 
\text{ with } \phi\in [0,\pi], \\
\proj \eta^5(q) = q & \text{ with } q \in [-\e,q_1), 
\end{cases}
\end{equation}
for $\e>0$ small enough,
(see Figure~\ref{fig:caminsNegatius1}).
\begin{figure}
\centering
\begin{overpic}[scale=0.8]{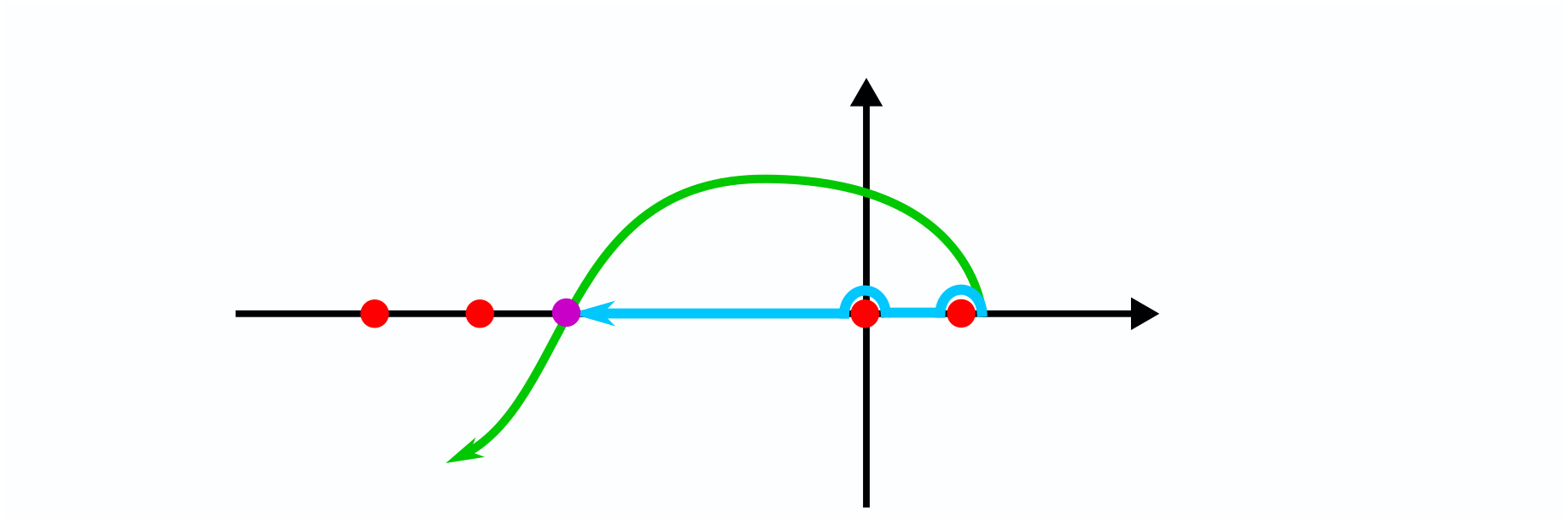}
	\put(53,9.5){{\color{red} $0$ }}	
	\put(61,9.5){{\color{red} $a_+$ }}
	\put(29,9.5){{\color{red} $-1$ }}
	\put(22,9.5){{\color{red} $a_-$ }}
	\put(75,12){$\Re q$}
	\put(52,30){$\Im q$}
	\put(48,23.5){{\color{myGreen} $\proj \g$ }}
	\put(45,15){{\color{myBlue} $\proj \eta$ }}
	\put(37,9.5){{\color{myPurple} $q_1(\g)$ }}
\end{overpic}
\caption{Example of a path $\g \in 	\PP_{\infty}$  such that ${q_1(\g) \in (-1,0)}$ and $\proj \g_1(\g)\subset \complexs^+$. 
The path $\eta$ has been defined in~\eqref{proof:pathEtaNegatius1}.}
\label{fig:caminsNegatius1}
\end{figure}
%
Using that $\int_{\eta^j} \fh(q) dq = \OO(\sqrt{\e})$ for $j=1,2,4$, that  $\fh|_{\eta^5} \subset \reals$, and~\eqref{proof:singularityt1mes}, one has that  $t^*(\g)$ is not visible since
\[
\vabs{\Im t_1(\g)}= 
\lim_{\e \to 0} 
\vabs{\int_{{\eta^3}} \fh(q) dq} = A.
\]
\end{proof}

\begin{lemma} \label{lemma:caminsBranca2}
The singularities $t^*(\g)$ given by paths $\g \in \PP_{0} \cup \PP_{\infty}$ such that 
$q_1(\g) \in (a_-,-1)$ are not visible.
\end{lemma}
\begin{proof}
Take a path $\g \in \PP_{0} \cup \PP_{\infty}$ with $q_1=q_1(\g)\in (a_-,-1)$ and $\proj \g_1(\g) \subset \complexs^+$.	
The integral of  the function $\fh$ along the path $\g_1(\g)$ coincides with the integral along the path
\begin{equation}\label{proof:pathEtaBranca2}
\eta = \bigvee_{j=1}^7 \eta^j,
\end{equation}
where the paths $\eta^j$, $j=1,2,3,4$, are defined in~\eqref{proof:pathEtaNegatius1} and
\[
\begin{cases}
\proj \eta^5(q) = q & 
\text{ with } q \in [-\e,-1+\e], \\
\proj \eta^6(\phi) = -1 + \e e^{i\phi} & 
\text{ with } \phi \in [0,\pi], \\
\proj \eta^7(q) = q & 
\text{ with } q \in [-1-\e,q_1), \\
\end{cases}
\]
for small enough $\e>0$,
(see Figure~\ref{fig:caminsBranca2}).
\begin{figure}
	\centering
	\begin{overpic}[scale=0.8]{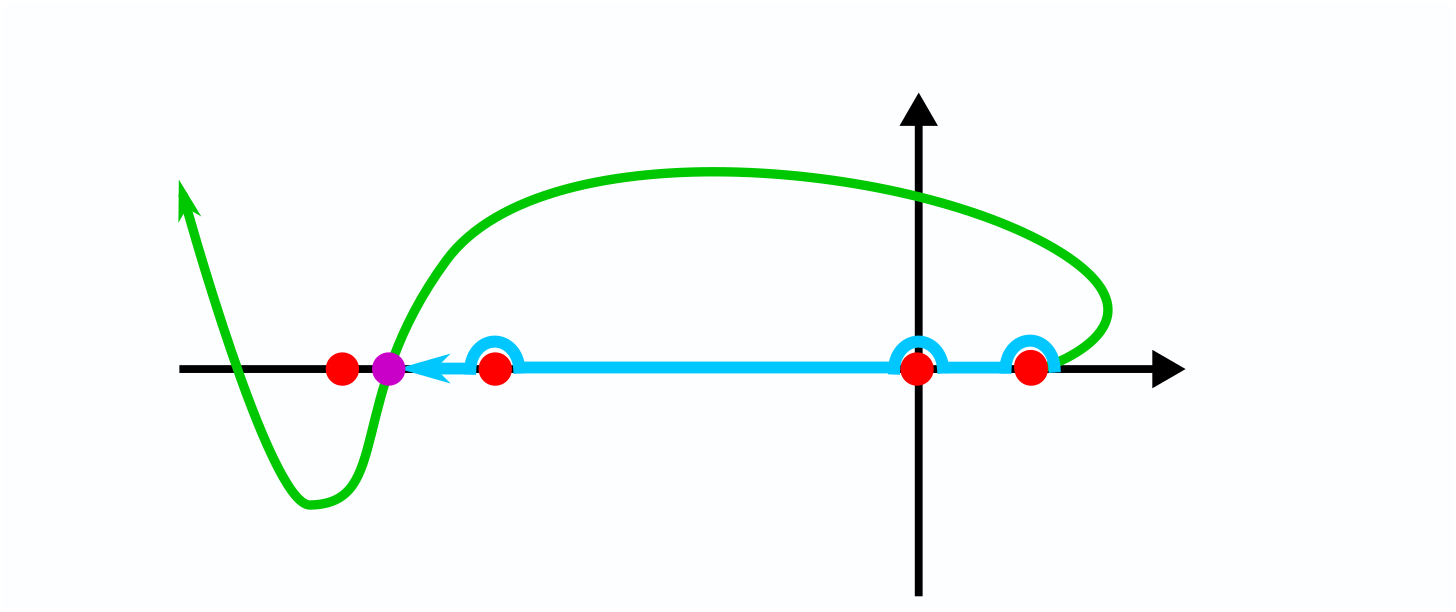}
		\put(64,12){{\color{red} $0$ }}	
		\put(71,12){{\color{red} $a_+$ }}
		\put(30,12){{\color{red} $-1$ }}
		\put(20,12){{\color{red} $a_-$ }}
		\put(83,15){$\Re q$}
		\put(60,37){$\Im q$}
		\put(45,32){{\color{myGreen} $\proj \g$ }}
		\put(50,19){{\color{myBlue} $\proj \eta$ }}
		\put(18,19){{\color{myPurple} $q_1(\g)$ }}
	\end{overpic}
	\caption{Example of a path $\g \in 	\PP_{\infty}$  such that ${q_1(\g) \in (a_-,-1)}$ and $\proj \g_1(\g)\subset \complexs^+$. 
	Path $\eta$ as defined in~\eqref{proof:pathEtaBranca2}.}
	\label{fig:caminsBranca2}
\end{figure}
Then, proceeding analogously to the proof of Lemma \ref{lemma:caminsNegatius1},
\begin{equation}\label{def:eta5to7}
\Im t_1(\g) = \lim_{\e\to 0 }\Im \int_{\eta} \fh(q) dq
= -A+\sum_{j=5}^7\lim_{\e\to 0 } \Im  \int_{\eta^j} \fh(q) dq.
\end{equation}
Since $\fh|_{\eta^5} \subset \reals$ and $\int_{\eta^6} \fh(q) dq = \OO(\sqrt{\e})$,
%
it only remains to compute the integral on $\eta^7$.
Following the natural arguments of the path $\eta$, we obtain
\[
\begin{array}{llll}
\arg(\eta^7) = \pi, & 
\arg(\eta^7-a_+) = \pi, &
\arg(\eta^7+1) = \pi, & 
\arg(\eta^7-a_-) = 0
\end{array} 
\]
and, as a consequence,
\begin{align*}
\lim_{\e \to 0} \int_{\eta^7} \fh(q) dq 
&= \int_{-1}^{q_1}  \frac{1}{q-1}\sqrt{\frac{|q|e^{i \pi}}{3|q+1|e^{i \pi}|q-a_+| e^{i\pi}(q-a_-)}} dq 
=- i B(q_1),
\end{align*}
where  $B(q_1)$ a real-valued, positive and strictly decreasing  function for $q_1 \in (a_-,-1)$.
%
Then, $t^*(\g)$ is not visible since,  by \eqref{def:eta5to7},
$\vabs{\Im t_1(\g)} = 
%
A + B(q_1)> A$.
%
\end{proof}

\subsubsection{Paths first crossing the real axis at \texorpdfstring{$[0,1]$}{[0,1]}}
\label{subsubsection:pathsC}
In this section, we check that
the singularities generated by paths C.1 to C.3 (see the list in Section~\ref{subsection:computationSingularities}) are either not visible or $\pm iA$. 

\begin{lemma} \label{lemma:caminsBranca1a}
The singularities $t^*(\g)$ given by paths $\g \in \PP_{0}$ only crossing $\reals$ at $(0,1)$ are either $t^*(\g)=iA$ or $t^*(\g)=-iA$.
\end{lemma}

\begin{proof}
The paths considered in this lemma can turn around the branching point $q=a_+$, but not around the other branching points nor the pole. 
Therefore, we classify these paths depending on how many turns they perform around $q=a_+$.
In order to do so, we define
\begin{equation}\label{proof:finalArgumentDef}
 \tht_{\fin}(\g) = \lim_{s \to s_{\fin}} \arg(\g(s)-a_+).
\end{equation}
%
%
The considered paths satisfy
$\tht_{\fin}(\g)=(2k+1)\pi$
for some $k \in \integers$ (see Figure~\ref{fig:caminsBranca1zero}).
 \begin{figure}
    \centering
        \begin{overpic}[scale=0.68]{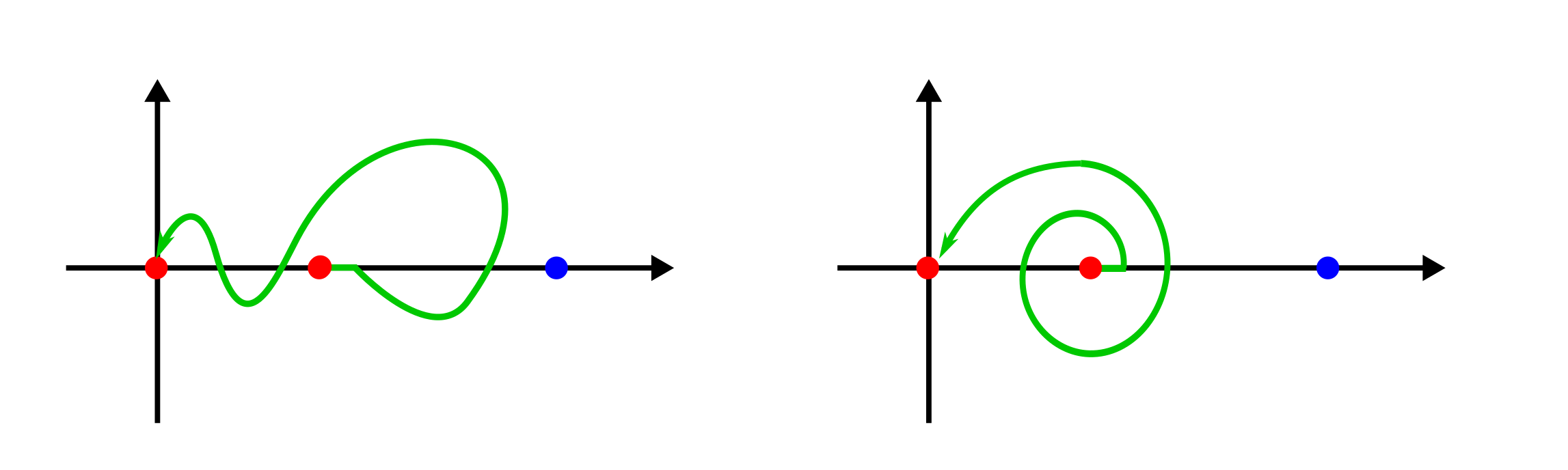}
		\put(8,10){{\color{red} $0$ }}
		\put(57,10){{\color{red} $0$ }}	
		\put(35,9.5){{\color{blue} $1$ }}
		\put(84,9.5){{\color{blue} $1$ }}
		\put(19.5,10){{\color{red} $a_+$ }}
		\put(69,10){{\color{red} $a_+$ }}
		\put(19,20){{\color{myGreen} $\proj \g$ }}
		\put(75,17){{\color{myGreen} $\proj \g$ }}
		\put(43.5,12){$\Re q$}
		\put(93,12){$\Re q$}
		\put(8,26){$\Im q$}
		\put(57,26){$\Im q$}
    \end{overpic}
    \caption{Example of paths $\g \in \PP_{0}$ only crossing $\reals$ at $(0,1)$. Left: $\tht_{\fin}(\g)=\pi$. Right: $\tht_{\fin}(\g)=3\pi$.}
    \label{fig:caminsBranca1zero}
 \end{figure}
%
%
Integrating the function $\fh$ along the path $\g$ is equivalent to integrating along ${\eta}={\eta}^1 \vee {\eta}^2 \vee \eta^3$ with 
\begin{equation}\label{proof:pathsEtaB}
\begin{cases}
\eta^1(q) = (q,0) & 
\text{ with } q \in 
(a_+,a_+ - \e], \\
\proj \eta^2(\phi) = a_+ + \e e^{i\phi} & 
\text{ with } \phi \in [\pi,(2k+1)\pi], \\
\proj \eta^3(q) = q & 
\text{ with } q \in [a_+-\e,0),
\end{cases}
\end{equation}
for  small enough $\e>0$. Since
$\int_{\eta^j} \fh(q) dq = \OO(\sqrt{\e})$ for $j=1,2$,
\begin{equation*}
t^*(\g) =
\lim_{\e \to 0} \int_{\eta^3} \fh(q) dq = 
\int_{a_+}^0 \frac{1}{q-1}\sqrt\frac{q}{3(q+1) \vabs{q-a_+} e^{i(2k+1)\pi}(q-a_-)} dq  = (-1)^{k+1} iA.
\end{equation*}
\end{proof}

\begin{lemma} \label{lemma:caminsBranca1b}
The singularities $t^*(\g)$ given by paths $\g \in \PP_{\infty}$ only crossing $\reals$ at $(0,1)$ are not visible.
\end{lemma}
\begin{proof}
We analyze the paths $\g \in \PP_{\infty}$ such that $\proj \g$ goes to infinity on $\complexs^+$
(by~\eqref{eq:reflection}, the paths on $\complexs^-$ give conjugated results).
Following the proof of Lemma~\ref{lemma:caminsBranca1a}, we classify the paths $\g$ depending on $\tht_{\fin}(\g)$, the final argument with respect to $a_+$ (see~\eqref{proof:finalArgumentDef} and Figure~\ref{fig:caminsBranca1infty}). 
\begin{figure}
\centering
\begin{overpic}[scale=0.7]{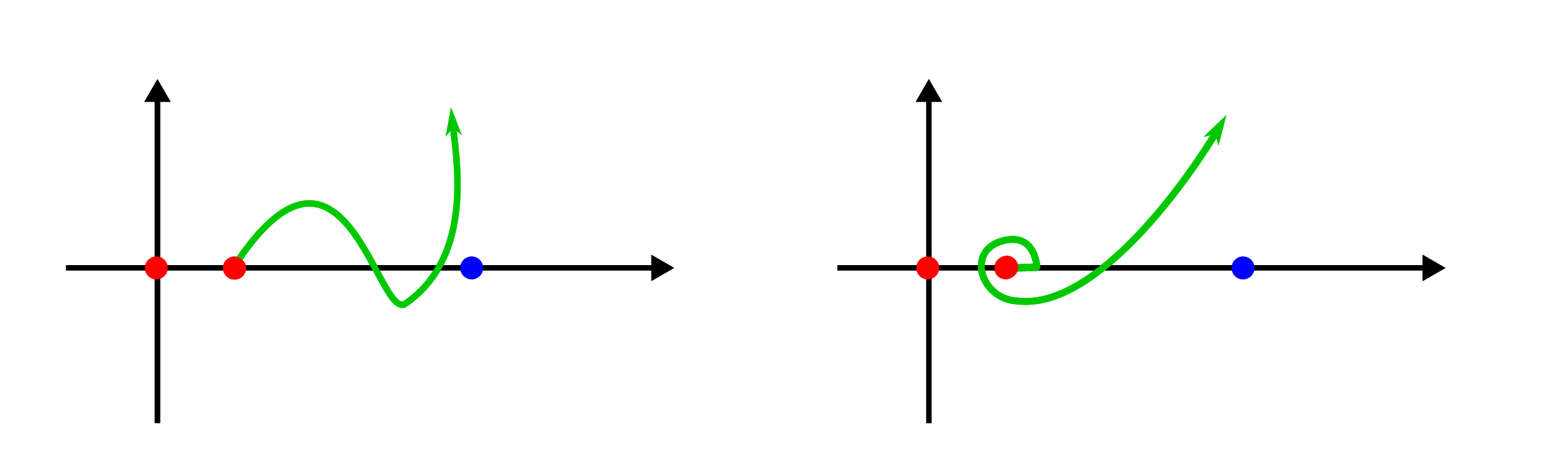}
	\put(8,10){{\color{red} $0$ }}
	\put(57,10){{\color{red} $0$ }}	
	\put(30,9.5){{\color{blue} $1$ }}
	\put(79,9.5){{\color{blue} $1$ }}
	\put(15,10){{\color{red} $a_+$ }}
	\put(63,9){{\color{red} $a_+$ }}
	\put(25,18){{\color{myGreen} $\proj \g$ }}
	\put(71,18){{\color{myGreen} $\proj \g$ }}
	\put(44,12){$\Re q$}
	\put(93,12){$\Re q$}
	\put(8,26){$\Im q$}
	\put(58,26){$\Im q$}
\end{overpic}
\caption{
	Example of paths $\g \in \PP_{\infty}$ only crossing $\reals$ at $(0,1)$ such that $\proj \g$ ends on the positive complex plane.
	Left: $\tht_{\fin}(\g)\in(0,\pi)$.
	Right: $\tht_{\fin}(\g)\in(2\pi,3\pi)$.
	}
\label{fig:caminsBranca1infty}
\end{figure}
The paths considered satisfy 
\[
\tht_{\fin}(\g)\in \big(2\pi k,(2k+1)\pi \big),
\quad \text{ for some } k \in \integers.
\]
We compute $t^*(\g)$ using  the path  ${\eta}={\eta}^1 \vee {\eta}^2 \vee \eta^3 \vee \eta^4
\vee \eta^5$ where the paths $\eta^1$, $\eta^2$ are defined in~\eqref{proof:pathsEtaB} and
\[
\begin{cases}
%
\proj \eta^3(q) = q & 
\text{ with } q \in [a_+ + \e,1-\e],\\
\proj \eta^4(\phi) = 1 + \e  e^{i\phi} & 
\text{ with } \phi \in [\pi,0], \\
\proj \eta^5(q) = q &
\text{ with } q \in [1+\e,+\infty),
\end{cases}
\]
for small enough $\e>0$. 
Since the integrals on $\eta^3$ and $\eta^5$ take real values and applying the results in Lemma~\ref{lemma:caminsBranca1a} for $\eta^1$ and $\eta^2$, we obtain
\[
\Im {t^*(\g)} = \Im \int_{\eta^4} \fh(q) dq.
\]
Proceeding as in the proof of Lemma~\ref{lemma:caminsInfinitSimple} and following the natural arguments of the path $\eta$, one deduces that
\[
\Im {t^*(\g)} = -\pi \Res \paren{\fh,(1,2\pi k)} =(-1)^{k+1} \pi \sqrt{\frac{2}{21}}.
\]
Therefore, since $\vabs{\Im t^*(\g)}>A$, the singularity is not visible.
\end{proof}

\begin{lemma} \label{lemma:caminsBranca1c}
	The singularities $t^*(\g)$ given by paths $\g \in \PP_{0} \cup \PP_{\infty}$ 
	both crossing $(0,1)$ and $\reals \setminus [0,1]$
	are not visible.
\end{lemma}

\begin{proof}
Let us define the parameter of the first crossing at $\reals \setminus [0,1]$ as 
\[
 s_2(\g)= \inf \{s \in(s_0,s_{\fin}) \st \Im \proj \g(s)=0, \, \Re \proj \g(s) \notin [0,1] \}
\]
and the corresponding point
\[
q_2(\g)=  \proj \g(s_2(\g)) \in 
\reals \setminus \left\{ [0,1], a_-, -1 \right\}.
\]
We consider paths $\g$ with $q_2= q_2(\g) \in (1,+\infty)$ and such that $\proj \g$ approaches $q_2(\g)$ from $\complexs^+$ (see~\ref{eq:reflection}). The cases $q_2 \in (-\infty,a_-)$,  $q_2 \in(a_-,-1)$ and $q_2 \in (-1,0)$ are proved analogously.

%
The strategy is to classify the paths $\g$ depending on how many turns they perform around $q=a_+$ before crossing $\reals \setminus [0,1]$.
To this end, we define $\tht_2(\g) = \arg(q_2(\g)-a_+)$ (see Figure~\ref{fig:caminsBranca1creuament}).
 \begin{figure}[t] 
    \centering
    \begin{overpic}[scale=0.7]{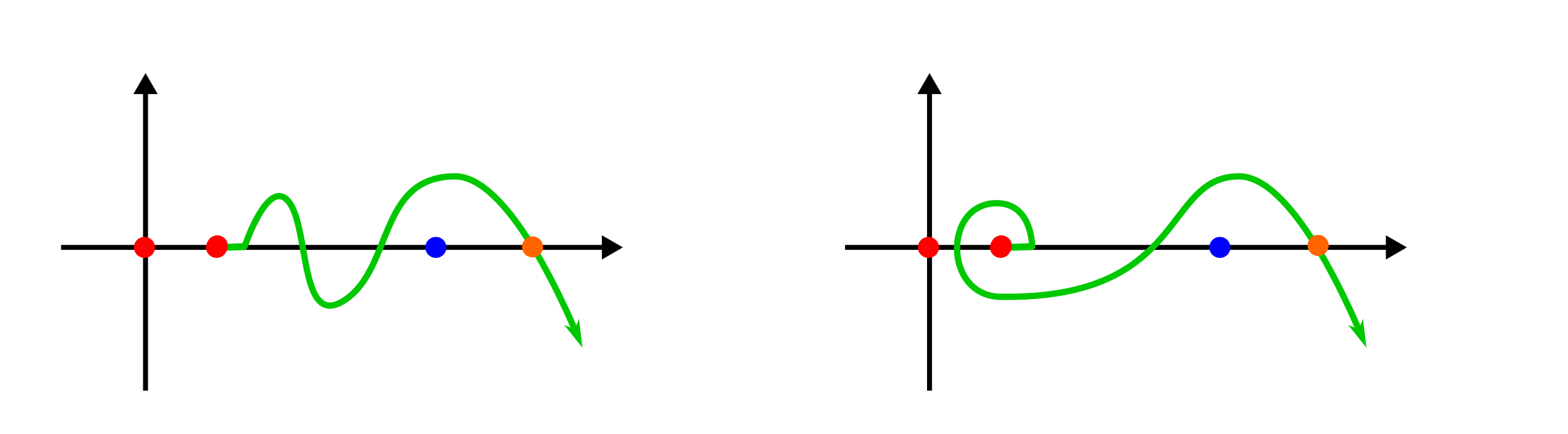}
    	\put(7.5,9.5){{\color{red} $0$ }}
    	\put(57.5,9.5){{\color{red} $0$ }}	
    	\put(27.5,9){{\color{blue} $1$ }}
    	\put(77.5,9){{\color{blue} $1$ }}
    	\put(13,9.5){{\color{red} $a_+$ }}
    	\put(63,9.5){{\color{red} $a_+$ }}
    	\put(40,11){$\Re q$}
    	\put(90,11){$\Re q$}
    	\put(8,24){$\Im q$}
    	\put(58,24){$\Im q$}
    	\put(28,18){{\color{myGreen} $\proj \g$ }}
    	\put(78,18){{\color{myGreen} $\proj \g$ }}
    	\put(34,14){{\color{myOrange} $q_2(\g)$ }}
    	\put(84,14){{\color{myOrange} $q_2(\g)$ }}
    \end{overpic}
    \caption{
    Example of paths $\g \in \PP_{\infty}$ crossing both $(0,1)$ and $\reals \setminus[0,1]$ with $q_2(\g) \in (1,+\infty)$ and such that $\proj \g$ approaches $q_2(\g)$ from $\complexs^+$.
    Left: $\tht_2(\g)=0$.
    Right: $\tht_2(\g)=2\pi$.
    }
    \label{fig:caminsBranca1creuament}
 \end{figure}
The paths we are considering satisfy  $\tht_2(\g)=2\pi k$ for some $k \in \integers$.
We also define the piece of  path before the crossing as
$
\g_2(\g) = \claus{\g(s) \st s \in (s_{\ini},s_2(\g))}
$
and the corresponding time
\begin{equation*}
	t_2(\g) = \int_{\g_2(\g)} \fh(q) dq.
\end{equation*}
To prove that a singularity $t^*(\g)$ is not visible, it is sufficient to check that $\vabs{\Im t_2(\g)} \geq A$.
\begin{enumerate}
\item Consider $\theta_2(\g)=2\pi k$ with $k$ an even number. 
Let us consider the path $\eta$ as defined in~\eqref{proof:pathEtaPol} replacing $q_1$ by $q_2$ in its definition.
This path $\eta$ lies entirely on the first Riemann sheet, that is $\arg \argf(\eta) \in (-\pi,\pi]$. 
%
%
Integrating the function $\fh$ along ${\g_2(\g)}$ is equivalent to integrating it along
\[
\xi= \widetilde{\eta}^1 \vee 
\widetilde{\eta}^2 \vee
\widetilde{\eta}^3 \vee
{\eta},
\]
where
\begin{equation}\label{proof:pathsEtaTilde}
\begin{cases}
\widetilde{\eta}^1(q) = (q,0) & 
\text{ with } q \in (a_+,a_+ +\e], \\
\proj \widetilde{\eta}^2(\phi)=a_+ + \e e^{i\phi} &
\text{ with } \phi \in [0,2\pi k], \\
\proj \widetilde{\eta}^3(q) = q & 
\text{ with } q \in [a_+ + \e,a_+),
\end{cases}
\end{equation}
for $\e>0$ small enough. 
Notice that this construction makes sense since the path $\widetilde{\eta}^3$ has argument 
$\arg \argf(\widetilde{\eta}^3)=2\pi k$ (which belongs to the first Riemann sheet).

Then, since
$\int_{\wt{\eta}^j} \fh(q) dq = \OO(\sqrt{\e})$ for $j=1,2,3$ and
applying Lemma~\ref{lemma:caminsPol}, we have
\[
\vabs{\Im t_2(\g)} = 
\lim_{\e \to 0}
\vabs{\Im \int_{\xi} \fh(q)dq} =
\vabs{\Im \int_{\eta} \fh(q)dq} > A.
\]

\item Consider $\theta_2(k)=2\pi k$ with $k$ an odd integer.
We define the  path $\eta^*$, lying on  the second Riemann sheet, as
\[
\eta^* = (\proj \eta, \arg \argf(\eta)+2\pi)\in \mathbb{C}\times (\pi,3\pi],
\]
where $\eta$ is the path introduced in~\eqref{proof:pathEtaPol} (replacing $q_1$ by $q_2$ in its definition).
Note that the path $\eta$ lies on the first Riemann sheet ($\arg \argf(\eta) \in (-\pi,\pi]$). 

It can be easily checked that switching the Riemann sheet implies a change in sign. That is,
\begin{equation}\label{proof:changeSheet}
\int_{\eta^*} \fh(q) dq = - \int_{\eta} \fh(q)dq.
\end{equation}
Then, integrating the function $\fh$ along the path ${\g_2(\g)}$ is equivalent to integrating it over 
\[
{\xi}=\widetilde{\eta}^1 \vee 
\widetilde{\eta}^2 \vee
\widetilde{\eta}^3 \vee
{\eta}^{*},
\]
where paths $\wt{\eta}^j$ for $j=1,2,3,$ are defined on~\eqref{proof:pathsEtaTilde}. 
This construction makes sense since the path $\widetilde{\eta}^3$ has argument $\arg \argf(\wt{\eta}^3)=2\pi k$ (which belongs to the second Riemann sheet).

Then, since
$\int_{\wt{\eta}^j} \fh(q) dq = \OO(\sqrt{\e})$ for $j=1,2,3$ and
applying Lemma~\ref{lemma:caminsPol} and formula~\eqref{proof:changeSheet}, we have
\[
\vabs{\Im t_2(\g)} = 
\lim_{\e \to 0}
\vabs{\Im \int_{\xi} \fh(q)dq} =
\vabs{\Im \int_{\eta} \fh(q)dq} > A.
\]
\end{enumerate}
\end{proof}

\subsection{Proof of Proposition~\ref{proposition:domainSeparatrix}}
\label{subsection:propZeroes}

For $t \in \reals$, $\La_h(t)$ satisfies $\La_h(t)=0$ if and only if $t=0$ (see Figure~\ref{fig:separatrix}).
To prove Proposition~\ref{proposition:domainSeparatrix}, we  follow the same techniques used in the proof of Theorem~\ref{theorem:singularities}.

%
Let us consider $q(t)=\cos(\frac{\la_h(t)}{2})$ as introduced in Theorem~\ref{theorem:singularitiesChangeq}. Then, by~\eqref{eq:equationLahSimplified},
\begin{equation*}
\La_h^2(t) = \frac{4}{3 q(t)} 
\big(1-q(t)\big)
\big(q(t)-a_+\big)
\big(q(t)-a_-\big).
\end{equation*}
Let $\La_h(t^*) = 0$ for a given $t^*$. Then, defining $q^*=q(t^*)$, we have three options:
\begin{equation*}
q^*=1,a_+,a_-.
\end{equation*}
%
We have seen that 
$q^*=1$ corresponds to the saddle equilibrium point, namely $\vabs{t^*}\to \infty$, (see~\eqref{proof:classificationSaddle}).
Therefore, it cannot lead to zeroes of $\La_h(t)$.
On the contrary, $q^*=a_{\pm}$ leads to zeroes of $\La_h(t)$, since we have seen that $q(t)$ is well defined and analytic in a neighborhood of such $t^*$ (see~\eqref{proof:qexpressionZeroes}).

To prove Proposition~\ref{proposition:domainSeparatrix} it only remains to compute all possible values of $t^* \in \ol{\Pi_A}$ such that $q^*=q(t^*)$ with $q^*=a_+,a_-$.
To do so, we use the techniques and results presented in Section~\ref{subsection:computationSingularities}.

From now on, we consider integration paths $\g:(s_{\ini}, s_{\fin}) \to \mathscr{R}_f$ with initial point $\lim_{s \to s_{\ini}} \g(s) = (a_+,0)$ and endpoint 
$\lim_{s \to s_{\ini}} \proj \g(s) = q^*=a_\pm$.
Moreover, we say that a zero $t^*$ of $\La_h$ is \emph{visible} if there exist a path $\g$ such that  $t^* = t^*(\g) \in \ol{\Pi_A}$ and $T[\g](s) \in \Pi_A$ for $s\in[s_{\ini},s_{\fin})$,  (see~\eqref{def:operatorT} and~\eqref{def:operatorSingularities}).
%

%
First, we recall some of the results obtained in Sections~\ref{subsubsection:pathsB} and~\ref{subsubsection:pathsC}.
\begin{itemize}
\item Consider $q_1 \in (-\infty, a_-) \cup (a_-,-1) \cup (1,+\infty)$. In the proofs of Lemmas~\ref{lemma:caminsPol}, \ref{lemma:caminsNegatius2}, \ref{lemma:caminsBranca2} and
\ref{lemma:caminsBranca1c} we have seen that
\begin{align}\label{proof:integralNotInPiA}
	\vabs{\Im \int_{a_+}^{q_1} \fh d\g} > A.
\end{align}
\item Consider $q_1\in(-1,0)$. In the proofs of Lemmas~\ref{lemma:caminsNegatius1} and~\ref{lemma:caminsBranca1c}, we have seen that
\begin{align}\label{proof:integralNotInPiA2}
	\vabs{\Im \int_{a_+}^{q_1} \fh d\g} = A.
\end{align}
\end{itemize}
Now, we classify the paths depending on its endpoint $q^*$.
\begin{enumerate}
\item Consider $q^*=a_-$.
Analogously to the proof of \eqref{proof:integralNotInPiA}, it can be seen that
\[
\vabs{\Im t^*(\g)} = 
\vabs{\Im \int_{a_+}^{a_-} \fh d\g} >A.
\]
Therefore, $q^*=a_-$ does not lead to any visible zero.

\item Consider $q^*=a_+$.
Notice that, by~\eqref{proof:integralNotInPiA} and~\eqref{proof:integralNotInPiA2},
any path crossing $\reals\setminus[0,1]$ leads to non-visible zeroes. 
Therefore, we only consider paths $\g$ either crossing $(0,1)$ or not crossing $\reals$.

Since in $(0,1)$ the only singularity of $\fh(q)$ is the branching point $q=a_+$, there exists a homotopic path $\eta=\eta^1 \vee \eta^2 \vee \eta^3$ defined by
\begin{equation*}
\begin{cases}
	\eta^1(q) = (q,0) & 
	\text{ with } q \in (a_+,a_+ + \e], \\
	\proj \eta^2(\phi) = a_+ + \e e^{i\phi} & 
	\text{ with } \phi \in [0, 2\pi k], \\
	\proj \eta^3(q) = q & 
	\text{ with } q \in [a_+ + \e, a_+), 
\end{cases}
\end{equation*}
for some $k \in \integers$ and $\e>0$ small enough. Then,
\begin{equation*}
t^*(\g) = \int_{a_+}^{a_+} \fh d \g = 
\lim_{\e \to 0} \int_{\eta} \fh(q) dq = 0.
\end{equation*}
Therefore, these paths lead to the only visible zero $t^*=0$.
\end{enumerate}
\qed

\section{The inner system of coordinates}
\label{section:proofBC-poincareFormulasInnerDerivation}
This section is devoted to prove Proposition~\ref{proposition:innerDerivation}. That  is we perform the suitable changes of coordinates, described in Section \ref{subsection:innerDerivation}, to Hamiltonian $H$ obtained in Theorem~\ref{theorem:HamiltonianScaling} (see~\eqref{def:hamiltonianScaling}) to obtain the inner Hamiltonian $\mathcal{H}$.
%
%
%
However, recall  that the Hamiltonian $H$ is defined by means of $H_1^{\Poi}$ (see~\eqref{def:hamiltonianPoincareSplitting}) which does not have a closed form.
For this reason, a preliminary step to to prove Proposition~\ref{proposition:innerDerivation} is to provide suitable expansions for $H_1^{\Poi}$ in an appropriate domain. This is done in  Section~\ref{subsection:proofB-poincareFormulas}.
Then, in Section~\ref{subsection:proofC-innerDerivation}, we apply the changes of coordinates introduced in Section~\ref{subsection:innerDerivation} to conclude the proof of the proposition.

\subsection{The Hamiltonian in Poincar\'e variables}
\label{subsection:proofB-poincareFormulas}



First, we give some formulae to translate the Delaunay variables and other orbital elements into Poincar\'e coordinates (see \eqref{def:changePoincare}). 
\begin{itemize}
	\item[--] \textbf{Eccentricity $e$} (see~\eqref{def:eccentricityDelaunay}):
	It can be written as
	\begin{equation}\label{eq:ePoincare}
	e = 2 \, \tilde{e}(L,\eta,\xi) \sqrt{\eta \xi}, 
	\quad \text{where} \quad
	\tilde{e}(L,\eta,\xi) = \frac{\sqrt{2L-\eta\xi}}{2L}
	= \frac{1}{\sqrt{2L}} + \OO(\eta \xi).
	\end{equation}
	Notice that $\tl{e}$ is analytic for $(L,\eta,\xi)\sim(1,0,0)$\footnote{
	This expansion is valid as long as $L\neq 0$.	
	However, since our analysis focuses on $L \sim 1$, to simplify notation we use this more restrictive domain.}.
	%
	%
	\item[--] \textbf{Argument of the perihelion $g$:} From the expression of $\eta$ and $\xi$ in \eqref{def:changePoincare2},
	\begin{equation}\label{eq:gTrigPoincare}
	\cos g = \frac{\eta + \xi}{2 \sqrt{\eta \xi}}, \qquad
	\sin g = -i \frac{\eta - \xi}{2 \sqrt{\eta \xi}}. 
	\end{equation}
	\item[--] \textbf{Mean anomaly $\ell$:}
	Since $\la=\ell+g$, we have that 
	\begin{equation*}
	\cos \ell = \frac{1}{2 \sqrt{\eta \xi}} \left( e^{-i \la} \eta + e^{i \la} \xi  \right), \qquad
	\sin \ell = \frac{i}{2 \sqrt{\eta \xi}} \left( e^{-i \la} \eta - e^{i \la} \xi  \right).
	\end{equation*}
	These expressions are not analytic at $(\eta,\xi)=(0,0)$. However, by~\eqref{eq:ePoincare},
	\begin{equation}\label{eq:eTriglPoincare}
	\begin{split}
	e\cos \ell = \tl{e}(L,\eta,\xi)\left( e^{-i \la} \eta + e^{i \la} \xi  \right), \quad
	e\sin \ell = i \tl{e}(L,\eta,\xi) \left( e^{-i \la} \eta - e^{i \la} \xi  \right),
	\end{split}
	\end{equation}
	are analytic for $(L,\eta,\xi) \sim (1,0,0)$.
	\item[--] \textbf{Eccentric anomaly $u$:}
	It can be implicitly defined by $u=\ell+e\sin u$ (see~\eqref{eq:uImplicitDefinition}),
	which implies
	\begin{align*}\label{eq:uSeriesDefinition}
	u = \ell + e\sin \ell + e^2 \cos\ell\sin\ell + 
	\OO(e\sin\ell,e\cos\ell)^3.
	\end{align*}
	Then, by~\eqref{eq:ePoincare} and \eqref{eq:eTriglPoincare},
	\begin{equation}\label{eq:trigoU}
	\begin{split}
	&e\cos u = \frac{1}{\sqrt{2L}}
	\paren{e^{-i\la}\eta+e^{i\la}\xi} 
	+ \frac{1}{2L} \paren{e^{-i\la}\eta-e^{i\la}\xi}^2 
	+ \OO(e^{-i\la}\eta,e^{i\la}\xi)^3, \\
	&e\sin u = 
	\frac{i}{\sqrt{2L}} \paren{e^{-i\la}\eta-e^{i\la}\xi}
	+\frac{i}{2L} \paren{e^{-2i\la}\eta^2-e^{2i\la}\xi^2} 
	+\OO(e^{-i\la}\eta,e^{i\la}\xi)^3,
	\end{split}
	\end{equation}
	which are also analytic for $(L,\eta,\xi)\sim(1,0,0)$.
\end{itemize}
For any  $\zeta \in [-1,1]$,  we define the function
\begin{equation}\label{def:functionD}
D[\zeta] = 
\paren{r^2 -2\zeta r\cos\tht + \zeta^2}
\circ \phi_{\Poi}.
\end{equation}
By the definition of $\mu H_1^{\Poi}$ in~\eqref{def:hamiltonianPoincareSplitting}, we have that
\begin{equation}\label{eq:decompositionH1PoiA}
\mu H_1^{\Poi} = \frac{1}{\sqrt{D[0]}} -\frac{1-\mu}{\sqrt{D[\mu]}}
-\frac{\mu}{\sqrt{D[\mu-1]}}.
\end{equation}

\begin{lemma}\label{lemma:seriesH1Poi}
	For $\vabs{(L-1,\eta, \xi)} \ll 1$ and any  $\zeta \in [-1,1]$, one can split $D[\zeta]$ as 
	\begin{align*}
		D[\zeta] =  D_0[\zeta] + D_1[\zeta] + D_2[\zeta] + D_{\geq 3}[\zeta],
	\end{align*}
	where 
	\begin{align*}
	D_0[\zeta](\la,L) =& \, 
	L^4 - 2\zeta L^2 \cos \la  +\zeta^2,  \\
	\begin{split}
	D_1[\zeta](\la,L,\eta,\xi) =& \,
	\eta  \frac{\sqrt{2 L^3}}{2}
	\paren{3 \zeta - 2 L^2 e^{-i\la}  -
		\zeta e^{-2i\la}} \\
	&+ \xi  \frac{\sqrt{2 L^3}}{2}
	\paren{3 \zeta - 2 L^2 e^{i\la}   -  \zeta e^{2 i\la}}, 
	\end{split} \\
	\begin{split}
	D_2[\zeta](\la,L,\xi,\eta) =& \,
	- {\eta^2}  \frac{L e^{-i\la}}{4}  \paren{
	\zeta + 2 L^2 e^{-i\la}  + 3 \zeta e^{-2 i\la} } 
	\\
	&- {\xi^2} \frac{L e^{i\la}}{4} \paren{\zeta +
	2 L^2 e^{i\la}  + 3 \zeta e^{2 i\la} }
+\eta \xi L \paren{
	3L^2 + 2\zeta\cos \la}.
	\end{split}
	\end{align*}
	Fix $\varrho\geq 0$.
	Then, for $\vabs{\Im \la}\leq \varrho$, the function
	$D_{\geq 3}[\zeta]$ is analytic and satisfies
	\begin{equation}\label{eq:boundD3mes}
	\vabs{D_{\geq 3}[\zeta](\la,L,\eta,\xi)} \leq 
	C \vabs{(\eta,\xi)}^3,
	\end{equation}
	 with $C=C(\varrho)$ a positive constant independent of $\zeta \in [-1,1]$.
\end{lemma}

\begin{proof}[Proof of Lemma~\ref{lemma:seriesH1Poi}]
In view of the definition of $D[\zeta]$ in \eqref{def:functionD}, we look for expansions for $r^2$ and $r\cos \tht$ (expressed in Poincar\'e coordinates) in powers of $(\eta,\xi)$.

Let us consider first $r^2$. 
Taking into account that $r=L^2(1-e\cos u)$ (see \eqref{eq:rthetaDefinition}) and the expansions in~\eqref{eq:trigoU} we obtain
\begin{equation}\label{proof:r2Poincare}
\begin{split}
r^2 =& \, L^4 
-L^3 \sqrt{2L} e^{-i\la} \eta
- L^3 \sqrt{2L} e^{i\la}\xi
+ 3 L^3\eta \xi   \\
&- \frac{L^3}{2} e^{-2i\la}\eta^2
- \frac{L^3}{2} e^{2i\la} \xi^2
+\OO(e^{-i\la}\eta, e^{i\la}\xi)^3.
\end{split}
\end{equation}
%
%
Now, we compute an expansion for $r\cos \tht$.
Taking into account \eqref{eq:thetaDefinition} and \eqref{eq:fTrigDefinition},
\begin{align*}
r \cos \tht
&= L^2 \paren{ \cos(g+u) - e\cos g - \left(\sqrt{1-e^2}-1\right)\sin u\sin g }.
\end{align*}
Notice that, since $\la=\ell+g$ and $u=\ell+e\sin u$, we have that $\cos(g+u) = \cos(\la+e\sin u)$ is analytic at $(\eta,\xi)=(0,0)$.
Then, using \eqref{eq:ePoincare},
\eqref{eq:gTrigPoincare} and \eqref{eq:trigoU},
we deduce
\begin{equation}\label{proof:rcosPoincare}
\begin{split}
r\cos \tht =& \,  L^2 \cos \la  
- \eta \frac{\sqrt{2 L^3}}{2} \paren{1
+ i e^{-i\la} \sin \la  } 
- \xi  \frac{\sqrt{2 L^3}}{2} \paren{ 1 - i e^{i\la} \sin \la  } \\
&- \eta \xi {L\cos \la }
+ \frac{L}{4} \eta^2 \paren{e^{-i\la} 
+ e^{-2i\la} \cos \la 
-2 i e^{-2i\la} \sin \la } \\ 
&+ \xi^2 \frac{L}{4} \paren{e^{i\la}+
 e^{2i\la} \cos \la
+2i e^{2i\la} \sin \la }
+ \OO(e^{i\la}\eta, e^{-i\la}\xi)^3.
\end{split}
\end{equation}
Then, joininig the results in~\eqref{proof:r2Poincare} and~\eqref{proof:rcosPoincare}
with the definition of $D[\zeta]$ in~\eqref{def:functionD}, we obtain its expansion in $(\eta,\xi)$.
Moreover, since $D[\zeta]$ is analytic for $(L,\eta,\xi)\sim(1,0,0)$ and $\vabs{\Im \la} \leq \varrho$, the terms of order $3$  satisfy the estimate in \eqref{eq:boundD3mes}.
\end{proof}

\begin{remark}\label{remark:seriesH1Poi}
Observe that the Hamiltonian $H^{\Poi} = H_0^{\Poi}+\mu H_1^{\Poi}$ in~\eqref{def:hamiltonianPoincareSplitting} is analytic away from collision with the primaries.	
By the decomposition of $\mu H_1^{\Poi}$ in~\eqref{eq:decompositionH1PoiA}, collisions with the primary $S$ are given by the zeroes of the function $D[\mu]$ and collisions with $P$ are given by the zeroes of $D[\mu-1]$.

Since our analysis is performed for $\vabs{(L-1,\eta,\xi)}\leq \e \ll 1$ and $0<\mu\ll 1$, 
by Lemma~\ref{lemma:seriesH1Poi}, one has
\begin{align*}
	D[\mu] = 1 + \OO(\mu,\e), \qquad
	D[\mu-1] = 2 + 2\cos \la + \OO(\mu,\e).
\end{align*}
That is, collisions with $S$ are not possible whereas  collisions with $P$ may take place when $\la \sim \pi$.
\end{remark}

\subsection{Proof of Proposition~\ref{proposition:innerDerivation}}
\label{subsection:proofC-innerDerivation}

To prove Proposition~\ref{proposition:innerDerivation}, we  analyze the Hamiltonian $H^{\Inner}$ which is given (up to a constant) by
\begin{equation*}
\frac{\de^{\frac{4}{3}}}{2 \al_+^2}
\paren{H \circ \phi_{\equi} \circ \phi_{\out} \circ \phi_{\inn}},
\end{equation*}
where the changes $\phi_{\equi}$, $\phi_{\out}$ and $\phi_{\inn}$ are defined in~\eqref{def:changeEqui}, \eqref{def:changeOuter} and \eqref{def:changeInner}, and $H=H_0+H_1$ (see \eqref{def:hamiltonianScalingH0} and $H_1$ in~\eqref{def:hamiltonianScalingH1}).

In the rest of the section, when performing changes of coordinates, to simplify notation, we omit the constant terms in the Hamiltonians.

%
Using the formulas for $H_1^{\Poi}$ in~\eqref{eq:decompositionH1PoiA} and Lemma~\ref{lemma:seriesH1Poi},  
%
we split $H_1^{\Poi}$ into two terms: one for the perturbation coming from the massive primary ($S$) and the other coming from the small primary ($P$),
\begin{align*}
H_1^{\Poi} = H_1^{\Poi,S} + H_1^{\Poi, P},
\end{align*}
which,  recalling that
$\mu = \de^4$, are defined as
\begin{equation}\label{proof:splitH1Poi}
\begin{split}
H_{1}^{\Poi, S} = \frac{1}{\de^4}
\paren{\frac{1}{\sqrt{D[0]}} -\frac{1-\de^4}{ \sqrt{D[\de^4]}}} 
\quad \text{ and } \quad
H_{1}^{\Poi, P} = - \frac{1}{\sqrt{D[\de^4-1]}}.
\end{split}
\end{equation}
%
We also define the Hamiltonian $H^{\equi} = H \circ \phi_{\equi}$, which can be split as
\begin{equation*}
	H^{\equi} = H_{0} + R^{\equi} 
	+ H_1^{\equi,P} + H_1^{\equi,S},
\end{equation*}
with
\begin{equation}\label{proof:H1Requi}
\begin{split}
H_1^{\equi, *}(\la,\La,x,y;\de) =& \,
H_1^{\Poi,*}\circ \phi_{\sca} \circ \phi_{\equi},
\qquad \text{for } {*}=S,P,\\
R^{\equi}(\la,\La,x,y;\de) =& 
- V(\la)
+ \frac{1}{\de^4} 
F_{\pend}(\de^2\La + \de^4\LtresLa(\de)) \\
&- 3 \de^2 \La \LtresLa(\de) 
+ \de (x \Ltresy(\de) + y \Ltresx(\de)).
\end{split}
\end{equation}
We recall that  $\phi_{\sca}$ is the scaling given in~\eqref{def:changeScaling}, $V$ is the potential  in~\eqref{def:potentialV},
$F_{\pend}$  is the function ~\eqref{def:Fpend} and
$(\LtresLa,\Ltresx,\Ltresy)$ are introduced in~\eqref{def:pointL3sca}. 

Then, the Hamiltonian $H^{\inn}$ can be written as 
\begin{align}\label{proof:hamiltonianInnerB}
H^{\inn} 
= \frac{\de^{\frac{4}{3}}}{2\al_+^2} 
\paren{
H_0 \circ \Psi
+ R^{\equi} \circ \Psi
+ H_1^{\Poi,P}\circ\Phi 
+ H_1^{\Poi,S}\circ\Phi},
\end{align}
where
\begin{equation*}
\Psi = \phi_{\out} \circ \phi_{\inn} 
\quad \text{ and } \quad
\Phi = \phi_{\sca} \circ \phi_{\equi} \circ \phi_{\out} \circ \phi_{\Inner}.
\end{equation*}
In the following lemmas, we introduce expressions for the changes $\Psi$ and $\Phi$.
%

\begin{lemma}\label{lemma:PsiMultipleChanges}
The change of coordinates $\Psi=(\Psi_{\la},\Psi_{\La},\Psi_x,\Psi_y)$	
satisfies
\begin{align*}
\Psi_{\la}(U) &= 
\pi + 3\al_+ \de^{\frac{4}{3}} U^{\frac{2}{3}}
\paren{ 1 + g_{\la}(\de^2 U)}, \\
\Psi_{\La}(U,W) &= 
-\frac{2\al_+}{3 \de^{\frac{2}{3}} U^{\frac{1}{3}} } 
\paren{ 1 + g_{\La}(\de^2 U)}
+ \frac{\al_+ U^{\frac{1}{3}} W}{\de^{\frac{2}{3}}} 
\paren{1+\wt{g}_{\La}(\de^2 U)}, \\
\Psi_x(X) &= \de^{\frac{1}{3}} \sqrt{2} \al_+ X, 
\\
\Psi_y(Y) &= \de^{\frac{1}{3}} \sqrt{2} \al_+ Y,
\end{align*}
where $g_{\la}(z)$, $g_{\La}(z)$, $\wt{g}_{\La}(z) \sim \OO(z^{\frac{2}{3}})$. 
Moreover, taking into account the time-parametrization of the separatrix $(\la_h,\La_h)$ given in~\eqref{eq:separatrixParametrization}, we have that
\begin{equation}\label{proof:separatrixInner}
\begin{split}
\La_h \circ \phi_{\Inner} =
-\frac{2\al_+}{3 \de^{\frac{2}{3}} U^{\frac{1}{3}}}
\paren{ 1 + g_{\La}(\de^2 U)}.
\end{split}
\end{equation}
\end{lemma}

\begin{lemma}\label{lemma:PhiMultipleChanges}
The change of coordinates
$\Phi=  (\Phi_{\la},\Phi_{L},\Phi_{\eta},\Phi_{\xi})$ satisfies
\begin{align*}
\Phi_{\la}(U)&=\Psi_{\la}(U), \quad &
\Phi_{L}(U,W) &= 1+\de^2\Psi_{\La}(U,W) + \de^4 \LtresLa(\de), \\
\Phi_{\eta}(X) &= \de \Psi_{x}(X) + \de^4 \Ltresx(\de),
\quad &
\Phi_{\xi}(Y) &= \de \Psi_{y}(Y) + \de^4 \Ltresy(\de),
\end{align*}
where $\Psi=(\Psi_{\la},\Psi_{\La},\Psi_x,\Psi_y)$ is the change of coordinates given in Lemma~\ref{lemma:PsiMultipleChanges}.
\end{lemma}
We omit the proofs of these lemmas  since they  are a straightforward consequence of  Theorem~\ref{theorem:singularities} and the definitions of the changes of coordinates (see \eqref{def:changeScaling}, \eqref{def:changeEqui}, \eqref{def:changeOuter} and \eqref{def:changeInner}).


%
%

\begin{proof}[End of the proof of Proposition~\ref{proposition:innerDerivation}]
We analyze each component of \eqref{proof:hamiltonianInnerB}.

We denote by $C(\cttInnDerA,\cttInnDerB)>0$ any constant satisfying that there exist
$\cttInnDerC, \cttInnDerAA, \cttInnDerBB>0$
independent of $\cttInnDerA, \cttInnDerB, \de$ such that $C(\cttInnDerA, \cttInnDerB) \leq \cttInnDerC \cttInnDerA^{\cttInnDerAA} \cttInnDerB^{\cttInnDerBB}$.
\begin{enumerate}
	\item We compute the first term of the Hamiltonian $H^{\inn}$ in~\eqref{proof:hamiltonianInnerB}.
	Since $H_{\pend}(\la_h,\La_h)=H_{\pend}(0,0)= -\frac{1}{2}$
	 and taking into account~\eqref{proof:separatrixInner}, we have
\begin{equation*}
\begin{split}	
\frac{\de^{\frac{4}{3}}}{2\al_+^2} 
H_0 \circ \Psi 
=& \frac{\de^{\frac{4}{3}}}{2\al_+^2} \paren{w -\frac{w^2}{6 \La_h^2(u)} + \frac{xy}{\de^2} } \circ \phi_{\inn}\\
=& \, W + XY - 
\frac{3}{4} U^{\frac{2}{3}}  W^2 \left(\frac{1}{1+g_{\La}(\de^2 U)} \right)^2 \\
=& \, W + XY - \frac{3}{4} U^{\frac{2}{3}}  W^2
+ \OO \paren{\de^{\frac{4}{3}} 
U^{\frac{2}{3}} W^2}.
\end{split}	
\end{equation*}	
Since $\vabs{U} \leq \cttInnDerA$ and $\vabs{W} \leq \cttInnDerB$, the error term $\OO (\de^{\frac{4}{3}} U^{\frac{2}{3}} W^2)$
 can be bounded by $C(\cttInnDerA,\cttInnDerB) \de^{\frac{4}{3}}$.

For the other terms in \eqref{proof:hamiltonianInnerB}, to simplify the notation, we are not specifying the dependence of the error terms on the variables $(U,W,X,Y)$.
Moreover, when referring to error terms of order $\OO(\de^{a})$, we mean that they can be bounded by $C(\cttInnDerA,\cttInnDerB)\de^{a}$.

\item For the second term of the Hamiltonian $H^{\inn}$ in~\eqref{proof:hamiltonianInnerB} (see 
by~\eqref{proof:H1Requi}) we have 
\begin{equation}\label{proof:RequiChange}
\begin{split}
\frac{\de^{\frac{4}{3}}}{2\al_+^2}
R^{\equi} \circ \Psi
=& 
-\frac{\de^{\frac{4}{3}}}{2\al_+^2} V(\Psi_{\la})
+ \frac{1}{2\al_+^2\de^{\frac{8}{3}}} 
F_{\pend} \paren{
\de^2 \Psi_{\La} + \de^4 \LtresLa(\de)} \\
&- \frac{3 \de^{\frac{10}{3}} \LtresLa(\de)}{2 \al^2_+} \Psi_{\La}
+ \frac{\de^{\frac{7}{3}}\Ltresy(\de)}{\sqrt{2} \al_+} \Psi_x 
+ \frac{\de^{\frac{7}{3}}\Ltresx(\de)}{\sqrt{2} \al_+} \Psi_y,
\end{split}
\end{equation}
where $F_{\pend}(z)=\OO(z^3)$ (see~\eqref{def:Fpend}) and $V(\la)$ is the potential given in \eqref{def:potentialV}.

First we analyze the potential term. 
By Lemma~\ref{lemma:PsiMultipleChanges}, we have that
\begin{align*}
-\frac{\de^{\frac{4}{3}}}{2\al_+^2} V(\Psi_{\la}(U))	
&= \frac{\de^{\frac{4}{3}}}{2\al_+^2 }
\frac{1}{\sqrt{2+2\cos \Psi_{\la}(U)}}
+ \OO(\de^{\frac{4}{3}}) \\
&= \frac{\de^{\frac{4}{3}}}{2\al_+^2 }
\paren{	9 \al^2_+ \de^{\frac{8}{3}} U^{\frac{4}{3}}\left( 1 + g_{\la}(\de^2 U)\right)^2 +
\OO\paren{\de^{\frac{16}{3}}U^{\frac{8}{3}}}}^{-\frac{1}{2}} 
+ \OO(\de^{\frac{4}{3}})
\\
&= \frac{1}{3 U^{\frac{2}{3}}} +
\OO(\de^{\frac{4}{3}}).
\end{align*}
Then, since $\cttInnDerA^{-1} \leq \vabs{U} \leq \cttInnDerA$ and $\vabs{(W,X,Y)} \leq \cttInnDerB$, by~\eqref{proof:RequiChange} and Lemma~\ref{lemma:PsiMultipleChanges}, 
\begin{align*}
\vabs{\frac{\de^{\frac{4}{3}}}{2\al_+^2}
R^{\equi} \circ \Psi - \frac{1}{3 U^{\frac{2}{3}}}}
\leq
C(\cttInnDerA,\cttInnDerB){\de^{\frac{4}{3}}}.
\end{align*}
\item We deal with the third term  of the Hamiltonian $H^{\inn}$ (see~\eqref{proof:splitH1Poi}).
%
Since
\[ \vabs{\Phi(U,W,X,Y)-(\pi,1,0,0)} 
\leq C(\cttInnDerA, \cttInnDerB) \de^{\frac{4}{3}}
\quad \text{and} \quad 
\vabs{\Im \Phi_{\la}(U)} \leq C(\cttInnDerA) \de^{\frac{4}{3}},
\]
the hypotheses of Lemma~\ref{lemma:seriesH1Poi} hold and therefore:
\begin{equation*}
\begin{split}
\frac{ \de^{\frac{4}{3}}}{2\al_+^2 }
{H_1^{\Poi,P} \circ \Phi}
&= - \frac{\de^{\frac{4}{3}}}{2\al_+^2}
\frac{1}{\sqrt{D[\de^4-1] \circ \Phi }} \\
&= - \paren{ \frac{2 \al_+}{\de^{\frac{8}{3}}}	
\paren{D_0+D_1+D_2+D_{\geq 3}}[\de^4-1] \circ \Phi }^{-\frac{1}{2}}.
\end{split}
\end{equation*}
We compute every term $D_j[\de^4-1]$, $j=0,1,2, \geq3$.
%
\begin{enumerate}[leftmargin=*, label*=\alph*.]
\item[a)] The term $D_0[\de^4-1]$ satisfies
\begin{equation*}
\begin{split}
D_0[\de^4-1](\la,L;\de) 
=& \, L^4 +2(1-\de^4)L^2\cos \la + (1-\de^4)^2 \\ 
=& \, 2(1+\cos \la)
+ 4(L-1)(1+\cos \la) 
+ 2(L-1)^2 (3+\cos\la) \\
&+ 4(L-1)^3
+(L-1)^4
-2\de^4(1+\cos \la) \\
&-4\de^4(L-1)\cos\la
-2\de^4(L-1)^2\cos\la
+\de^8.
\end{split}
\end{equation*}
Performing the change $\Phi$, by Lemma~\ref{lemma:PhiMultipleChanges}, we have that
\begin{equation*}
\begin{split}
\frac{2\al_+}{\de^{\frac{8}{3}}} D_0[\de^4-1]\circ \Phi 
&= \frac{2\al_+}{\de^{\frac{8}{3}}} 
\paren{2(1+\cos \Phi_{\la} ) 
+ 4 (\Phi_{L}-1)^2 +
\OO(\de^{\frac{4}{3}})}\\
&=  9 U^{\frac{4}{3}} 
+ 4 U^{\frac{2}{3}} W^2 - \frac{16}{3} W + \frac{16}{9 U^{\frac{2}{3}}}
 + \OO(\de^{\frac{4}{3}}).
\end{split}
\end{equation*}
\item[b)] Analgously the term $D_1[\de^4-1]$ satisfies
\begin{equation*}
\begin{split}
D_1[\de^4-1] 
=& \, \eta \frac{\sqrt{2 L^3}}{2}  
\left[
(-3- 2 e^{-i\la} + e^{-2i\la})
- 4 (L-1) e^{-i\la}
\right. \\ &\left. \qquad \qquad
- 2 (L-1)^2 e^{-i\la}
+\de^4 (3- e^{-2i\la})
\right] \\
&+ \xi \frac{\sqrt{2 L^3}}{2}  \left[
(-3- 2 e^{i\la} + e^{2i\la})
-  4 (L-1) e^{i\la}
\right. \\ &\left. \qquad \qquad \quad
- 2 (L-1)^2 e^{i\la}
+\de^4 (3- e^{2i\la}) \right]
\end{split}
\end{equation*}
and therefore
\begin{equation*}
\begin{split}
\frac{2\al_+}{\de^{\frac{8}{3}}}D_1[\de^4-1]\circ \Phi
=& \frac{2 \al^2_+}{\de^{\frac{4}{3}}}  X
\paren{-4i\paren{\Phi_{\la}-\pi} 
- 4(\Phi_L-1)
 + \OO(\de^{\frac{4}{3}}) }
\\
&+ \frac{2 \al^2_+}{\de^{\frac{4}{3}}} Y
\paren{4i\paren{\Phi_{\la}- \pi} 
	- 4(\Phi_L-1)
	+ \OO(\de^{\frac{4}{3}}) }
\\
=& \, X\paren{-12 i U^{\frac{2}{3}} + 4U^{\frac{1}{3}} - 
	\frac{8}{3 U^{\frac{1}{3}}}} \\
&+ Y\paren{12 i U^{\frac{2}{3}} + 4U^{\frac{1}{3}} - 
\frac{8}{3 U^{\frac{1}{3}}} }
+ \OO(\de^{\frac{4}{3}}).
\end{split}
\end{equation*}
\item[c)] The term $D_2[\de^4-1]$ satisfies
\begin{align*}
D_2[\de^4-1] =& 
- \eta^2 \frac{L e^{-i \la}}{4} \paren{
	{-1 + 2 L^2 e^{-i\la} -3 e^{-2i\la} } +
	\de^4 \paren{ 1 + 3 e^{-2i\la}} } \\
&- \xi^2 \frac{L e^{-i \la}}{4} \paren{
{-1 + 2 L^2 e^{i\la} -3 e^{2i\la} } +
\de^4 \paren{ 1 + 3 e^{2i\la}} } \\
&+ \eta \xi L \paren{ 
3L^2 - 2\cos \la  +
\de^4 2\cos \la }
\end{align*}
and
\begin{align*}
&\frac{2\al_+}{\de^{\frac{8}{3}}} D_2[\de^4-1]\circ \Phi
= - 3 X^2  -3 Y^2 + 5 XY + \OO(\de^{\frac{4}{3}}).
\end{align*}
\item[d)] By the estimates of 
$D_{\geq 3}[\de^4-1]$ in~\eqref{eq:boundD3mes} and Lemma~\ref{lemma:PhiMultipleChanges},
\begin{align*}
\vabs{\frac{2\al_+}{\de^{\frac{8}{3}}} 
D_3[\de^4-1]\circ \Phi} 
\leq C(\cttInnDerA,\cttInnDerB)
{\de^{\frac{4}{3}}}.
\end{align*}  
\end{enumerate}
Collecting these results, we  conclude that 
\begin{equation*}
	\begin{split}
		\frac{ \de^{\frac{4}{3}}}{2\al_+^2 }
		{H_1^{\Poi,P} \circ \Phi}
		&= - \frac{1}{3 U^{\frac{2}{3}}}
		\frac{1}{\sqrt{1+\JJ(U,W,X,Y)}} 
		+ \OO(\de^{\frac{4}{3}}),
	\end{split}
\end{equation*}
where the function $\JJ$ is given in~\eqref{def:hFunction}. 
\item 
%
Proceeding analogously as for $H^{\Poi,P}_1$, it can be checked that
\begin{align*}
\vabs{\frac{\de^{\frac{4}{3}}}{2\al_+^2}
H^{\Poi,S}_1 \circ \Phi }
\leq
C(\cttInnDerA,\cttInnDerB)\de^{\frac{4}{3}}.
\end{align*}
\end{enumerate}
\end{proof}

\section{Analysis of the inner equation}
\label{section:proofD-AnalysisInner}

We split the proof of Theorem~\ref{theorem:innerComputations} into two parts. 
In Section~\ref{subsection:innerExistence} we prove the existence of the solutions $\ZuInn$ and $\ZsInn$ and the estimates in~\eqref{result:innerExistence}. 
In Section~\ref{subsection:innerDifference}, we 
provide the asymptotic formula for the difference $\Delta \ZInn = \ZuInn-\ZsInn$ given in \eqref{result:innerDifference}.
For both parts, we follow the approach given in~\cite{BalSea08}.

Throughout this section, we fix the $\beta_0 \in (0,\frac{\pi}{2})$ appearing in the definition of the domains $\DuInn$ and $\DsInn$ in \eqref{def:domainInnner} and $\EInn$ in~\eqref{def:domainInnnerDiff}.
We denote the components of all the functions and operators by a numerical sub-index $f=(f_1,f_2,f_3)^T$, unless stated otherwise.
In order to simplify the notation, we denote by $C$ any positive constant independent of $\kappa$.
%

\subsection{Existence of suitable solutions of the inner equation}
\label{subsection:innerExistence}

From now on we deal only with the analysis for $\ZuInn$. The analysis for $\ZsInn$  is analogous.

%

\subsubsection{Preliminaries and set up}

The invariance equation~\eqref{eq:invariantEquationInner}, that is $\partial_U \ZuInn = 
\AAA \ZuInn + \RRR[\ZuInn]$,  can be written as $\LL \ZuInn = \RRR[\ZuInn]$
where $\LL$ is  the linear operator
\begin{equation}\label{def:operatorL}
	\LL \varphi =(\partial_U-\AAA)\varphi.
\end{equation}
Notice that if we can construct  a left-inverse of $\LL$ in an appropriate Banach space, we can write~\eqref{eq:invariantEquationInner} as a fixed point equation to be able apply the Banah fixed point theorem.



Given $\nu \in \reals$ and $\rhoInn>0$, we define the norm
\begin{align*}
\normInn{\varphi}_{\nu}= \sup_{U \in \DuInn} \vabs{U^{\nu} \varphi(U) },	
%
%
\end{align*}
where the domain $\DuInn$ is given in~\eqref{def:domainInnner}, 
and we introduce the Banach space
\begin{equation*}
\begin{split}
	\XcalInn_{\nu}&= 
	\left\{ \varphi: \DuInn \to \complexs  \st  
	\varphi \text{ analytic, } 
	\normInn{\varphi}_{\nu} < +\infty \right\}.
\end{split}
\end{equation*}
Next lemma, proven in~\cite{Bal06}, gives some properties of these Banach spaces.
We use this lemma throughout the section  without mentioning it.
%
%

\begin{lemma}\label{lemma:sumNorms}
	Let $\rhoInn>0$ and $\nu, \eta \in \reals$. The following statements hold:
	\begin{enumerate}
	\item If $\nu > \eta$, then $\XcalInn_{\nu} \subset \XcalInn_{\eta}$ and
	$
	\normInn{\varphi}_{\eta} \leq \left( 
	{\rhoInn}{\cos \beta_0}
	\right)^{\eta-\nu} 
	\normInn{\varphi}_{\nu}.
	$
	\item If $\varphi \in \XcalInn_{\nu}$ and 
	$\zeta \in \XcalInn_{\eta}$, then the product
	$\varphi \zeta \in \XcalInn_{\nu+\eta}$ and
	$
	\normInn{\varphi \zeta}_{\nu+\eta} \leq \normInn{\varphi}_{\nu} \normInn{\zeta}_{\eta}.
	$
	\end{enumerate}
\end{lemma}

In the following  lemma, we introduce a left-inverse of the operator $\LL$ in~\eqref{def:operatorL}.

\begin{lemma}\label{lemma:Glipschitz}
Consider the operator
\begin{equation*}
	\GG[\varphi](U) = \left(\int_{-\infty}^U \varphi_1(S) dS, 
	\int_{-\infty}^U e^{-i (S-U)} \varphi_2(S) dS , 
	\int_{-\infty}^U e^{i(S-U)} \varphi_3(S) dS \right)^T.
\end{equation*}
Fix $\eta>1$, $\nu>0$ and  $\rhoInn\geq 1$.
Then, $\GG: \XcalInn_{\eta} \times \XcalInn_{\nu} \times \XcalInn_{\nu} \to \XcalInn_{\eta-1} \times \XcalInn_{\nu} \times \XcalInn_{\nu}$ is a continuous linear operator and is a left-inverse of $\LL$. 

Moreover, there exist a constant $C>0$ such that
	\begin{enumerate}
	\item If $\varphi \in \XcalInn_{\eta}$, then $\GG_1[\varphi] \in \XcalInn_{\eta-1}$ and
	$
	\normInn{\GG_1[\varphi]}_{\eta-1} \leq C  \normInn{\varphi}_{\eta}.
	$
	\item If $\varphi \in \XcalInn_{\nu}$ and $j=2,3$, then $\GG_j[\varphi] \in \XcalInn_{\nu}$ and
	$
	\normInn{\GG_j[\varphi]}_{\nu} \leq   C \normInn{\varphi}_{\nu}.
	$
	\end{enumerate}
\end{lemma}
\begin{proof}
It follows the same lines as the proof of Lemma 4.6 in~\cite{Bal06}.
\end{proof}
Let us then define the fixed point operator 
\begin{equation}\label{def:operatorFF}
\FF= \GG \circ \RRR.
\end{equation}
%
A solution of 
$
\ZuInn=\FF[\ZuInn]
$
belonging to $\XcalInn_{\eta} \times \XcalInn_{\nu} \times \XcalInn_{\nu}$
with $\eta,\nu>0$
satisfies equation~\eqref{eq:invariantEquationInner} and the asymptotic condition~\eqref{eq:asymptoticConditionsInner}.
Therefore, to prove the first part of Theorem~\ref{theorem:innerComputations} and the asymptotic estimates in \eqref{result:innerExistence}, we look for a fixed point of  the operator $\FF$  in the Banach space
\begin{equation*}
	\XcalInnTotal = \XcalInn_{\frac{8}{3}} \times \XcalInn_{\frac{4}{3}} \times \XcalInn_{\frac{4}{3}}, 
\end{equation*}
endowed with the norm
\[
\normInnTotal{\varphi}= 
\normInn{\varphi_1}_{\frac{8}{3}} + 
\normInn{\varphi_2}_{\frac{4}{3}} + 
\normInn{\varphi_3}_{\frac{4}{3}}.
\] 

\begin{proposition}\label{proposition:innerExistence}	There exists $\rhoInn_0>0$ such that for any $\rhoInn \geq \rhoInn_0$, the fixed point equation $\ZuInn = {\FF}[\ZuInn]$ has a solution $\ZuInn \in \XcalInnTotal$.
Moreover, there exists a constant $\cttInnExistA>0$, independent of $\rhoInn$, such that
\[
\normInnTotal{\ZuInn} \leq \cttInnExistA.
\]
\end{proposition}
\begin{remark}\label{remark:kappaBigEnough}
Notice that $\DuInn \subseteq \DuInnR{\rhoInn_0}$ when $\rhoInn\geq\rhoInn_0$ (see~\eqref{def:domainInnner}). 
Then, for some $\nu \in \reals$, if $\zeta \in \XcalInn_{\nu}$ (defined for $\rhoInn$) then $\zeta \in \XX_{\nu}$ (defined for $\rhoInn_0$).
This allows us to take $\rhoInn$ as big as we need.
\end{remark}

\subsubsection{Proof of Proposition~\ref{proposition:innerExistence}}
%
%
%
We first state a technical lemma whose proof is postponed until Section~\ref{subsubsection:innerTechnicalProofsA}.
For $\varrho>0$, we define the closed ball 
\begin{equation*}
B(\varrho) = \claus{\varphi \in \XcalInnTotal \st
\normInnTotal{\varphi}\leq \varrho}.
\end{equation*}

\begin{lemma}\label{lemma:boundsRRR}
Let ${\RRR}$ be the operator defined in~\eqref{def:operatorRRRInner}.
Then, for  $\varrho>0$ and for $\rhoInn>0$ big enough,
there exists a constant $C>0$ 
such that, for any
$\ZInn \in B(\varrho)$, 
\begin{align*}
&\normInn{\RRR_1[\ZInn]}_{\frac{11}{3}} \leq C, \qquad
\normInn{\RRR_j[\ZInn]}_{\frac{4}{3}} \leq C, \qquad
j=2,3,
\end{align*}
	and
	\begin{align*}
	\normInn{\partial_W\RRR_1[\ZInn]}_3 &\leq C, \quad &
	\normInn{\partial_X\RRR_1[\ZInn]}_{\frac{7}{3}} &\leq C, \quad &
	\normInn{\partial_Y\RRR_1[\ZInn]}_{\frac{7}{3}} &\leq C, \\
	\normInn{\partial_W\RRR_j[\ZInn]}_{\frac{2}{3}} &\leq C, \quad &
	\normInn{\partial_X\RRR_j[\ZInn]}_{2} &\leq C , \quad &
	\normInn{\partial_Y\RRR_j[\ZInn]}_{2} &\leq C, \quad j=2,3.
\end{align*} 
\end{lemma}

The next lemma gives properties of the operator $\FF$. 
%
%
%

\begin{lemma}\label{lemma:firstIterationInner}
	Let ${\FF}$ be the operator defined in~\eqref{def:operatorFF}. 
	Then, for $\rhoInn>0$ big enough, 
	there exists a constant $\cttInnExistB>0$ independent of $\kappa$ such that
	\[
	\normInnTotal{{\FF[0]}} \leq \cttInnExistB.
	\]
%

Moreover, for  $\varrho>0$ and $\rhoInn>0$ big enough, 
there exists a constant $\cttInnExistC>0$ independent of $\kappa$  such that, 
for any $\ZInn=(\WInn,\XInn,\YInn)^T$, $\ZInnB=(\WInnB,\XInnB,\YInnB)^T \in B(\varrho) \subset \XcalInnTotal$, 
	\begin{align*} 
	\normInnSmall{{\FF_1}[\ZInn]-{\FF_1}[\ZInnB]}_{\frac{8}{3}} &\leq \cttInnExistC
	\left( \frac{1}{\rhoInn^2} 
	\normInnSmall{\WInn - \WInnB}_{\frac{8}{3}} 
	+ \normInnSmall{\XInn - \XInnB}_{\frac{4}{3}}
	+ \normInnSmall{\YInn - \YInnB}_{\frac{4}{3}} \right), \\
	%
	\normInnSmall{{\FF_j}[\ZInn]-{\FF_j}[\ZInnB]}_{\frac{4}{3}} &\leq \frac{\cttInnExistC}{\rhoInn^2} \normInnTotalSmall{\ZInn - \ZInnB},
	\qquad
	j=2, 3.
	\end{align*}
\end{lemma}

\begin{proof}
%
The estimate for  $\FF[0]$ is a direct consequence of Lemmas~\ref{lemma:Glipschitz} and~\ref{lemma:boundsRRR}.

To estimate the Lipschitz constant, we first estimate each component $\RRR_j[\ZInn]-\RRR_j[\ZInnB]$ separately  for $j=1,2,3$. By the mean value theorem we have
\begin{align*}
\RRR_j[\ZInn]-\RRR_j[\ZInnB] 
&=\boxClaus{\int_0^1 D \RRR_j[s \ZInn + (1-s) \ZInnB]  d s }
(\ZInn - \ZInnB) .
\end{align*}
Then, for $j=2,3$, we have
\begin{align*}
\MoveEqLeft[4]{\normInnSmall{\RRR_1[\ZInn]-\RRR_1[\ZInnB]}_{\frac{11}{3}} \leq
\normInnSmall{\WInn - \WInnB}_{\frac{8}{3}}
\sup_{\varphi \in {B(\varrho)}}  \normInn{\partial_W\RRR_1[\varphi]}_{1} }\\
&+ \normInnSmall{\XInn - \XInnB}_{\frac{4}{3}}
\sup_{\varphi \in {B(\varrho)}} \normInn{\partial_X{\RRR_1}[\varphi]}_{\frac{7}{3}} 
+ \normInnSmall{\YInn - \YInnB}_{\frac{4}{3}} \sup_{\varphi \in {B(\varrho)}} 
\normInn{\partial_Y\RRR_1[\varphi]}_{\frac{7}{3}}, \\
\MoveEqLeft[4]{\normInnSmall{\RRR_j[\ZInn]-\RRR_j[\ZInnB]}_{\frac{4}{3}} \leq
\normInnSmall{\WInn - \WInnB}_{\frac{8}{3}}
\sup_{\varphi \in {B(\varrho)}} \normInn{\partial_W\RRR_j[\varphi]}_{-\frac{4}{3}} }\\
&+ \normInnSmall{\XInn - \XInnB}_{\frac{4}{3}} 
\sup_{\varphi \in {B(\varrho)}} \normInn{\partial_X{\RRR_j}[\varphi]}_{0} 
+ \normInnSmall{\YInn - \YInnB}_{\frac{4}{3}}
\sup_{\varphi \in {B(\varrho)} } 
\normInn{\partial_Y \RRR_j[\varphi]}_{0}.
\end{align*}
Applying Lemma~\ref{lemma:boundsRRR}, we obtain
\begin{align*}
\normInnSmall{{\RRR_1}[\ZInn]-{\RRR_1}[\ZInnB]}_{\frac{11}{3}} &\leq C
\left( \frac{1}{\rhoInn^2} \normInnSmall{\WInn - \WInnB}_{\frac{8}{3}} 
+ \normInnSmall{\XInn - \XInnB}_{\frac{4}{3}}
+ \normInnSmall{\YInn - \YInnB}_{\frac{4}{3}} \right),\\
\normInnSmall{{\RRR_j}[\ZInn]-{\RRR_j}[\ZInnB]}_{\frac{4}{3}} &\leq \frac{C}{\rhoInn^2} \normInnTotalSmall{\ZInn - \ZInnB},
\quad j=2,3.
\end{align*}
Finally, applying Lemma~\ref{lemma:Glipschitz}, we obtain the estimates in the lemma.
\end{proof}

Lemma~\ref{lemma:firstIterationInner} shows that,
by assuming $\rhoInn>0$ big enough, the operators $\FF_2$ and $\FF_3$ have Lipschitz constant less than $1$.
However, this is not the case for $\FF_1$.
To overcome this problem, we apply a Gauss-Seidel argument and define a new operator 
\begin{equation*} 
\wt{\FF}[\ZInn]=
\wt{\FF}[(\WInn,\XInn,\YInn)]=
\begin{pmatrix} 
\FF_1[\WInn,\FF_2[\ZInn],\FF_3[\ZInn]] \\ 
\FF_2[\ZInn] \\
\FF_3[\ZInn] \end{pmatrix},
\end{equation*}
%
which has the same fixed points as $\FF$ and turns out to be contractive in a suitable ball. 

\begin{proof}[End of the proof of Proposition \ref{proposition:innerExistence}]
We first obtain an estimate for $\normInnTotalSmall{\wt{\FF}[0]}$. Notice that
\begin{equation*} 
\wt{\FF}[0]= \FF[0]
+ 
\big(\wt \FF[0]-\FF[0]\big)=
\FF[0]
+ 
\big(
\FF_1\left[0,\FF_2[0],\FF_3[0]\right]-\FF_1[0], 0, 0 \big)^T.
\end{equation*} 
Then, by Lemma~\ref{lemma:firstIterationInner}, $(0,\FF_2[0],\FF_3[0])^T\in \XcalInnTotal$ and 
\begin{align*}
\normInnTotalSmall{\wt{\FF}[0]} &\leq
\normInnTotal{\FF[0]} +
\normInn{\FF_1[0,\FF_2[0],\FF_3[0]]-\FF_1[0]}_{\frac{8}{3}} \\
&\leq \normInnTotal{\FF[0]}
+ C \normInn{\FF_2[0]}_{\frac{4}{3}}
+ C \normInn{\FF_3[0]}_{\frac{4}{3}} 
\leq C \normInnTotal{\FF[0]}.
\end{align*}
Thus, we can fix $\varrho>0$ such that
\begin{align*}
\normInnTotalSmall{\wt{\FF}[0]} \leq \frac{\varrho}{2}. 
\end{align*}
Now, we prove that the operator $\wt{\FF}$ is contractive in $B(\varrho) \subset \XcalInnTotal$. 
By Lemma~\ref{lemma:firstIterationInner} and assuming $\rhoInn>0$ big enough, we have that for $\ZInn, \ZInnB \in B(\varrho)$,
\begin{align*}
\normInnSmall{\wt{\FF}_1[\ZInn]-\wt{\FF}_1[\ZInnB]}_{\frac{8}{3}} \leq& 
\frac{C}{\rhoInn^2} \normInnSmall{\WInn-\WInnB}_{\frac{8}{3}} +
\normInnSmall{\FF_2[\ZInn]-\FF_2[\ZInnB]}_{\frac{4}{3}} +
\normInnSmall{\FF_3[\ZInn]-\FF_3[\ZInnB]}_{\frac{4}{3}} \\
%
%
\leq& 
\frac{C}{\rhoInn^2} 
\normInnSmall{\WInn-\WInnB}_{\frac{8}{3}} +
\frac{C}{\rhoInn^2}\normInnTotalSmall{\ZInn-\ZInnB}
\leq	
\frac{C}{\rhoInn^2}\normInnTotalSmall{\ZInn-\ZInnB}, \\
\normInnSmall{\wt{\FF}_j[\ZInn]-\wt{\FF}_j[\ZInnB]}_{\frac{4}{3}}  
\leq&
\frac{C}{\rhoInn^2} \normInnTotalSmall{\ZInn-\ZInnB},
\qquad \text{ for } j=2, 3.
\end{align*}
Then, there exists $\kappa_0>0$ such that for $\kappa\geq\kappa_0$, the operator
$\wt{\FF}:B(\varrho) \to B(\varrho)$ is well defined and contractive. 
Therefore $\wt{\FF}$ has a fixed point 
$\ZuInn \in B(\varrho) \subset \XcalInnTotal$.
\end{proof}


\subsection{Asymptotic formula for the difference}
\label{subsection:innerDifference}

The strategy to prove the second part of Theorem~\ref{theorem:innerComputations} is divided in three steps. 
In Section~\ref{subsubsection:innerDiffPreliminaries} we characterize $\DZInn= \ZuInn-\ZsInn$ as a solution of a linear homogeneous equation.
In Section~\ref{subsubsection:innerDiffPotential}, we prove that $\DZInn$ is in fact  the unique solution of this linear equation in a suitable Banach space.
Finally, in Section~\ref{subsubsection:innerDiffExponential}, we introduce a  Banach subspace of the previous one (with exponential weights) to obtain
exponentially small estimates for $\DZInn$.


\subsubsection{A homogeneous linear equation for \texorpdfstring{$\DZInn$}{the difference}}
\label{subsubsection:innerDiffPreliminaries}
By Theorem~\ref{theorem:innerComputations}, the difference $\DZInn(U) = \ZuInn(U) - \ZsInn(U)$ is well defined for $U \in \EInn$ (see~\eqref{def:domainInnnerDiff}).
Since $\ZuInn$, $\ZsInn$ satisfy the same invariance equation~\eqref{eq:invariantEquationInner}, their difference $\DZInn$  satisfies
\begin{equation*}
\partial_U \DZInn = 
\AAA \DZInn + R(U)\DZInn,
\end{equation*}
where 
\begin{align}\label{def:operatorRDiff}
R(U)=
\int_0^1 D_Z\RRR[s\ZuInn + (1-s)\ZsInn](U) ds,
\end{align}
and $\AAA$ and $\RRR$ are given in~\eqref{def:matrixAAA} and~\eqref{def:operatorRRRInner}, respectively.
We denote by $R_{1}, R_{2}$ and $R_{3}$, the rows of the matrix $R$.

By the method of variation of parameters, there exists $c=(c_w,c_x,c_y)^T \in \complexs^3$ such that
\begin{align*}
\DZInn(U) = e^{\AAA U} \paren{ c
+ \int_{U_0}^{U} e^{-\AAA S} R(S) \DZInn(S) d S}.
\end{align*}
By Proposition~\ref{proposition:innerExistence}, 
$\DZInn = \ZuInn - \ZsInn$ satisfies  $\lim_{\Im U \to -\infty}\DZInn(U) =0$. Therefore $\DZInn$ satisfies
%
\begin{equation}\label{eq:differenceMainEq}
\DZInn = \DZo + \II[\DZInn],
\end{equation}
where $\II$ is  the linear operator
\begin{equation}\label{def:operatorEEE}
\II[\varphi](U)=
\begin{pmatrix}
\displaystyle
\int_{-i\infty}^{U}  
\langle R_{1}(S), \varphi(S)\rangle d S \\[1em]
\displaystyle
e^{iU}\int_{-i\infty}^{U}  
e^{-iS} 
\langle R_{2}(S), \varphi(S)\rangle d S \\[1em]
\displaystyle
e^{-iU} \int_{-i\rhoInn}^{U} e^{iS} 
\langle  R_{3}(S), \varphi(S)\rangle d S
\end{pmatrix},
\end{equation}
%
%
and $\DZo$ is the function
\begin{equation*}
\DZo(U) = (0,0,c_y e^{-iU})^T=(0,0,e^{\rhoInn} \DYInn(-i\rhoInn) e^{-iU})^T.
\end{equation*}

%




\subsubsection{Characterization of \texorpdfstring{$\DZInn$}{the difference} as a fixed point}
\label{subsubsection:innerDiffPotential}

%
Given $\nu \in \reals$ and $\rhoInn>0$, we define the norm
\begin{align*}
\normInnDiff{\varphi}_{\nu}= \sup_{U \in \EInn} \vabs{U^{\nu} \varphi(U) },	
	%
	%
\end{align*}
where the domain $\EInn$ is given in~\eqref{def:domainInnnerDiff}, 
and we introduce the Banach space
\begin{equation*}
\begin{split}
		\YcalInn_{\nu}&= 
		\left\{ \varphi: \EInn \to \complexs  \st  
		\varphi \text{ analytic }, 
		\normInnDiff{\varphi}_{\nu} < +\infty \right\}.
	\end{split}
\end{equation*}
Note that  $\YcalInn_{\nu}$ satisfy analogous properties as the ones in Lemma~\ref{lemma:sumNorms}. In this section, 
we use this lemma without mentioning it.
%

We state a technical lemma, whose proof is postponed to Section~\ref{subsubsection:innerTechnicalProofsB}.
\begin{lemma} \label{lemma:operatorEEpotential}
Let $\II$ be the operator defined in~\eqref{def:operatorEEE}. 
Then, for $\rhoInn>0$ big enough, 
there exists a constant $\cttInnDiffA>0$ independent of $\kappa$ such that,
for
$\Psi \in \YcalInn_{\frac{8}{3}} \times \YcalInn_{\frac{4}{3}} \times \YcalInn_{\frac{4}{3}}$,
\begin{align*}
\normInnDiff{\II_1[\Psi]}_{\frac{8}{3}} 
	&\leq \cttInnDiffA \paren{
	\frac{1}{\rhoInn^2} \normInnDiff{\Psi_1}_{\frac{8}{3}} +
	\normInnDiff{\Psi_2}_{\frac{4}{3}} +
	\normInnDiff{\Psi_3}_{\frac{4}{3}}}, \\
\normInnDiff{\II_j[\Psi]}_{\frac{4}{3}} 
	&\leq \frac{\cttInnDiffA}{\rhoInn^2} \paren{
	\normInnDiff{\Psi_1}_{\frac{8}{3}} +
	\normInnDiff{\Psi_2}_{\frac{4}{3}} +
	\normInnDiff{\Psi_3}_{\frac{4}{3}} }, \qquad j=2,3.
\end{align*}
\end{lemma}
These estimates characterize $\DZInn$ as the unique solution of~\eqref{eq:differenceMainEq} in  $\YcalInn_{\frac{8}{3}} \times \YcalInn_{\frac{4}{3}} \times \YcalInn_{\frac{4}{3}}$.

\begin{lemma}\label{lemma:innerExistenceDifference}
%
For $\rhoInn>0$ big enough, 
$\DZInn$ is the unique solution of equation~\eqref{eq:differenceMainEq} belonging to
$\YcalInn_{\frac{8}{3}} \times \YcalInn_{\frac{4}{3}} \times \YcalInn_{\frac{4}{3}}$.
In particular,
\begin{equation*}
\DZInn = \sum_{n \geq 0} \II^n[\DZo].
\end{equation*}
\end{lemma}

\begin{proof}
By Theorem~\ref{theorem:innerComputations}, for $\rhoInn>0$ big enough, $\DZInn$ is a solution of equation~\eqref{eq:differenceMainEq} which satisfies 
$
\DZInn = \ZuInn - \ZsInn \in \YcalInn_{\frac{8}{3}} \times \YcalInn_{\frac{4}{3}} \times \YcalInn_{\frac{4}{3}}.
$
Then, it only remains to prove that  equation~\eqref{eq:differenceMainEq} has a unique solution in $\YcalInn_{\frac{8}{3}} \times \YcalInn_{\frac{4}{3}} \times \YcalInn_{\frac{4}{3}}$.
To this end, it is enough to show that the operator $\II$ is contractive with a suitable norm in $\YcalInn_{\frac{8}{3}} \times \YcalInn_{\frac{4}{3}} \times \YcalInn_{\frac{4}{3}}$.
Taking
\[
\normInnDiff{\Psi}_{\times} = \normInnDiff{\Psi_1}_{\frac{8}{3}}
+ \kappa \normInnDiff{\Psi_2}_{\frac{4}{3}}
+ \kappa \normInnDiff{\Psi_3}_{\frac{4}{3}},
\]
Lemma~\ref{lemma:operatorEEpotential} implies
\[
\normInnDiff{\II[\Psi]}_{\times} \leq 
\frac{C}{\kappa} \normInnDiff{\Psi}_{\times}
\]
and, taking $\kappa$ big enough, the result is proven.
\end{proof}

\subsubsection{Exponentially small estimates for \texorpdfstring{$\DZInn$}{the difference}}
\label{subsubsection:innerDiffExponential}

Once we have proved that $\DZInn$ is the unique solution of~\eqref{eq:differenceMainEq} in $\YcalInn_{\frac{8}{3}} \times \YcalInn_{\frac{4}{3}} \times \YcalInn_{\frac{4}{3}}$, 
we use this equation to obtain exponentially small estimates for $\DZInn$. 

For any $\nu \in \reals$, we consider the norm
\begin{equation*}
\normInnDiffExp{\varphi}_{\nu} = \sup_{U \in \EInn} \vabs{U^{\nu} e^{i U} \varphi(U)},
\end{equation*}
and the associated Banach space
\begin{equation*}
\ZcalInn_{\nu} = \left\{
\varphi: \EInn \to \complexs \st 
\varphi \text{ analytic, }
\normInnDiffExp{\varphi}_{\nu} < +\infty
\right\}.
\end{equation*}
Moreover, for  $\nu_1, \nu_2, \nu_3 \in \reals$, we consider the product space 
\begin{align*}
\ZcalInn_{\nu_1,\nu_2,\nu_3}= 
\ZcalInn_{\nu_1} \times
\ZcalInn_{\nu_2} \times
\ZcalInn_{\nu_3},
\quad \text{with} \quad
\normInnDiffExp{\varphi}_{\nu_1,\nu_2,\nu_3}=
\textstyle\sum_{j=1}^3 \normInnDiffExp{\varphi_j}_{\nu_j}.
\end{align*}

Next lemma, gives some properties of these Banach spaces. It follows the same lines as Lemma~\ref{lemma:sumNorms}.
%


\begin{lemma}\label{lemma:sumNormsExp}
Let $\rhoInn>0$ and $\nu, \eta \in \reals$. The following statements hold:
\begin{enumerate}
\item If $\nu > \eta$, then $\ZcalInn_{\nu} \subset \ZcalInn_{\eta}$ and
$
\normInnDiffExp{\varphi}_{\eta} \leq \left( 
{\rhoInn}{\cos \beta_0}\right)^{\eta-\nu} 
\normInnDiffExp{\varphi}_{\nu}.
$
\item If $\varphi \in \ZcalInn_{\nu}$ and 
$\zeta \in \YcalInn_{\eta}$, then the product
$\varphi \zeta \in \ZcalInn_{\nu+\eta}$ and
$
\normInnDiffExp{\varphi \zeta}_{\nu+\eta} \leq \normInnDiffExp{\varphi}_{\nu} \normInnDiff{\zeta}_{\eta}.
$
\item If $\varphi \in \ZcalInn_{\nu}$ then 
$e^{iU}\varphi \in \YcalInn_{\nu}$ and
$
\normInnDiffSmall{e^{iU} \varphi}_{\nu} = \normInnDiffExp{\varphi}_{\nu}.
$
\end{enumerate}
\end{lemma}

The next lemma analyzes how the operator $\II$ acts on the space $\ZcalInn_{\frac{4}{3},0,0}$. Its proof  is postponed to Section~\ref{subsubsection:innerTechnicalProofsB}.

\begin{lemma} \label{lemma:operatorEEexponential}
%
Let $\II$ be the operator defined in~\eqref{def:operatorEEE}. For $\rhoInn>0$ big enough, 
there exists a constant $\cttInnDiffB>0$ independent of $\kappa$ such that,
for
$\Psi \in \ZcalInn_{\frac{4}{3},0,0}$,
\begin{align*}
\normInnDiffExp{\II_1[\Psi]}_{\frac{7}{3}} 
\leq \cttInnDiffB \normInnDiffExp{\Psi}_{\frac{4}{3},0,0}, 
\quad
\normInnDiffExp{\II_2[\Psi]}_{2} 
\leq  \cttInnDiffB \normInnDiffExp{\Psi}_{\frac{4}{3},0,0},
\quad
\normInnDiffExp{\II_3[\Psi]}_{0}
\leq \frac{\cttInnDiffB}{\kappa}  \normInnDiffExp{\Psi}_{\frac{4}{3},0,0}.
\end{align*}
Moreover, there exists $\wCInn(\rhoInn) \in \complexs$ (depending on $\Psi$), such that
\begin{equation*}
	\II_3[\Psi] - e^{-iU}\wCInn(\rhoInn) \in \ZcalInn_{1}.
\end{equation*}
\end{lemma}


\begin{proof}[End of the proof of the second part of Theorem~\ref{theorem:innerComputations}]

%
Lemma~\ref{lemma:operatorEEexponential} implies that operator $\II: \ZcalInn_{\frac{4}{3},0,0} \to \ZcalInn_{\frac{4}{3},0,0}$ is well defined and contractive. 
Indeed, taking $\rhoInn>0$ big enough and $\Psi \in \ZcalInn_{\frac{4}{3},0,0}$, 
\begin{equation}\label{proof:EEcontractive}
\begin{split}
\normInnDiffExp{\II[\Psi]}_{\frac{4}{3},0,0} 
&= \normInnDiffExp{\II_1[\Psi]}_{\frac{4}{3}}
+ \normInnDiffExp{\II_2[\Psi]}_0
+ \normInnDiffExp{\II_3[\Psi]}_0 \\
&\leq
\frac{C}{\rhoInn} \normInnDiffExp{\II_1[\Psi]}_{\frac{7}{3}}+
\frac{C}{\rhoInn^2} \normInnDiffExp{\II_2[\Psi]}_2
+ \normInnDiffExp{\II_3[\Psi]}_0
\leq 
\frac{C}{\kappa} 
\normInnDiffExp{\Psi}_{\frac{4}{3},0,0}.
\end{split}
\end{equation}
Therefore, since $\DZo=(0,0,c_y e^{-iU})^T \in \ZcalInn_{\frac{4}{3},0,0}$, 
Lemma~\ref{lemma:innerExistenceDifference} and \eqref{proof:EEcontractive} imply that
\begin{equation*}
\DZInn =(\DWInn,\DXInn,\DYInn)^T = \sum_{n \geq 0} \II^n[\DZo] \in \ZcalInn_{\frac{4}{3},0,0}.
\end{equation*}
Lemma~\ref{lemma:operatorEEexponential} implies 
$
\II:\ZcalInn_{\frac{4}{3},0,0} \to \ZcalInn_{\frac{7}{3},2,0} \subset 
\ZcalInn_{\frac{4}{3},0,0},
$
which allows to give better estimates for $\DZInn$. Indeed,  we have that ${\DZInn - \DZo} = \II[\DZInn] \in \ZcalInn_{\frac{7}{3},2,0}$,  which implies  
\begin{align*}
	\DWInn=\II_1[\DZInn] \in \ZcalInn_{\frac{7}{3}}, \qquad
	\DXInn=\II_2[\DZInn] \in \ZcalInn_{2}, \qquad
\end{align*}
Moreover, by the second statement in Lemma \ref{lemma:operatorEEexponential}, 
there exists $\wCInn(\kappa)$ such that
\[	\DYInn - c_y e^{-iU} - \wCInn(\rhoInn) e^{-iU} \in \ZcalInn_{1}.\]
Calling $\CInn = c_y + \wCInn(\kappa)$ we have that
$
\DYInn(U) - \CInn e^{-iU} \in \ZcalInn_{1},
$
and, therefore $\CInn = \lim_{\Im U \to -\infty}  \DYInn(U)e^{iU}$, which is independent of $\rhoInn$.
Then, $\DZInn$ is of the form
\begin{align*}
\DZInn(U) = e^{-iU} \Big((0,0,\CInn)^T + \chi(U) \Big),
\quad \text{with} \quad \chi \in 
\YcalInn_{\frac{7}{3}} \times \YcalInn_{2} \times \YcalInn_{1}.
\end{align*}
%


Now we prove that, if there exists $U_0 \in \EInn$ such that $\DZInn(U_0) \neq 0$, then $\CInn \neq 0$.
This implies  $\DZInn(U) \neq 0$ for all $U \in \EInn$,
since $\DZInn$ is a solution of an homogeneous linear differential equation.
Therefore  $c_y \neq 0$. Indeed, $c_y=0$ would imply  $\DZo=0$ and, by Lemma~\ref{lemma:innerExistenceDifference},  one could conclude $\DZInn \equiv 0$.

Thus, it only remains to prove  that $c_y \neq 0$ implies $\CInn\neq 0$. By Lemma~\ref{lemma:innerExistenceDifference},
\begin{align*}
\DZInn - \DZo = \sum_{n\geq 1} \II^n[\DZo].
\end{align*} 
In addition, by the estimate \eqref{proof:EEcontractive},  $\normInnDiffExp{\II_3}_{\frac{4}{3}} \leq \frac{1}{4}$  if  $\rhoInn>0$ is big enough. Since $\DYo=c_y e^{-iU}$, we deduce that
\begin{align*}
\normInnDiffExp{\DYInn - \DYo}_{0} 
\leq \sum_{n\geq 1} \frac{1}{4^n} \normInnDiffExp{\DYo}_{0}
= \frac{1}{3} \vabs{c_y}, 
\end{align*}
and, by the definition of the norm $\normInnDiffExp{\cdot}_0$, for any $U \in \EInn$,
\begin{align*}
\vabs{e^{iU}\DYo(U)} -
	\vabs{e^{iU}\DYInn(U)} \leq 
	 \frac{1}{3} \vabs{c_y}.
\end{align*}
%
%
Hence, using that $e^{iU}\DYInn=\CInn+\chi_3(U)$ with $\chi_3 \in \YcalInn_{1}$ and $e^{iU}\DYo(U)=c_y$, we have that for all $U \in \EInn$,
\begin{align*}
\vabs{e^{iU}\DYInn(U)} =
\vabs{\CInn + \chi_3(U)} 
\geq
\vabs{e^{iU}\DYo(U)} - \frac{1}{3} \vabs{c_y}
= \frac{2}{3} \vabs{c_y}.
\end{align*}
Finally, taking $\Im(U)\to-\infty$, we obtain that $\vabs{\CInn} \geq \frac{2}{3}\vabs{c_y} > 0$.
%
\end{proof}


\subsection{Proof of the technical lemmas}
\label{subsection:innerTechnicalProofs}

We devote this section to prove Lemma~\ref{lemma:boundsRRR} of Section~\ref{subsubsection:innerTechnicalProofsA}
and
Lemmas \ref{lemma:operatorEEpotential} and \ref{lemma:operatorEEexponential} 
of Section~\ref{subsubsection:innerTechnicalProofsB}.


\subsubsection{Proof of Lemma~\ref{lemma:boundsRRR}}
\label{subsubsection:innerTechnicalProofsA}

Fix $\varrho>0$ and take $\ZInn=(\WInn,\XInn,\YInn)^T \in B(\varrho) \subset \XcalInnTotal$. By the definition \eqref{def:operatorRRRInner} of $\RRR$,
\begin{align}\label{proof:operatorRRRInner}
\RRR[\ZInn](U) = \paren{\frac{f_1(U,\ZInn)}{1+g(U,\ZInn)},\,
 \frac{\wt{f}_2(U,\ZInn)}{1+g(U,\ZInn)}, \,
 \frac{\wt{f}_3(U,\ZInn)}{1+g(U,\ZInn)} },
\end{align}
where
	\begin{equation*}
	\begin{split}
	\wt{f}_2(U,\ZInn) &= f_2(U,\ZInn) - i{\XInn} g(U,\ZInn), \qquad
	\wt{f}_3(U,\ZInn) = f_3(U,\ZInn) + i{\YInn} g(U,\ZInn), 
	\end{split}
	\end{equation*}
with $g=\partial_W \KK$, $f=(-\partial_U \KK, i\partial_Y \KK,-i\partial_X \KK)^T$ and $\KK$ is the Hamiltonian given in~\eqref{def:hamiltonianK} in terms of the function $\JJ$ (see~\eqref{def:hFunction}).
%

We first estimate $\JJ$ and its derivatives.
For $\rhoInn>0$ big enough, we have 
	\begin{align*}
	& \vabs{\JJ(U,\ZInn)} \leq \frac{C}{U^{2}},\qquad
	\vabs{1+\JJ(U,\ZInn)} \geq 1 - \frac{C}{\rhoInn^2}\geq \frac{1}{2}.
	\end{align*}
Moreover, its  derivatives satisfy
\begin{align*}
	\vabs{\partial_U{\JJ}(U,\ZInn)} 
	&\leq \frac{C}{\vabs{U}^{3}},  &
	\vabs{\partial_W{\JJ}(U,\ZInn)} 
	&\leq \frac{C}{\vabs{U}^{\frac{4}{3}}},  &
	\vabs{\partial_X{\JJ}(U,\ZInn)},
	\vabs{\partial_Y{\JJ}(U,\ZInn)}  
	&\leq \frac{C}{\vabs{U}^{\frac{2}{3}}}, 
\end{align*}
and
\begin{align*}
	\vabs{\partial_{U W}{\JJ}(U,\ZInn)} 
	&\leq \frac{C}{\vabs{U}^{\frac{7}{3}}}, \quad &
	\vabs{\partial_{U X}{\JJ}(U,\ZInn)} 
	&\leq \frac{C}{\vabs{U}^{\frac{5}{3}}}, 
	\quad &
	\vabs{\partial_{U Y}{\JJ}(U,\ZInn)} 
	&\leq \frac{C}{\vabs{U}^{\frac{5}{3}}},
	\\
	\vabs{\partial^2_{W}{\JJ}(U,\ZInn)} 
	&\leq \frac{C}{\vabs{U}^{\frac{2}{3}}}, 
	\quad &
	\vabs{\partial_{W X}{\JJ}(U,\ZInn)} 
	&\leq \frac{C}{\vabs{U}}, 
	\quad &
	\vabs{\partial_{W Y}{\JJ}(U,\ZInn)}  
	&\leq \frac{C}{\vabs{U}}, 
	\\
	\vabs{\partial^2_{X}{\JJ}(U,\ZInn)} 
	&\leq \frac{C}{\vabs{U}^{\frac{4}{3}}},\quad &
	\vabs{\partial_{X Y}{\JJ}(U,\ZInn)} 
	&\leq \frac{C}{\vabs{U}^{\frac{4}{3}}}, 
	\quad &
	\vabs{\partial_{Y}^2{\JJ}(U,\ZInn)}
	&\leq \frac{C}{\vabs{U}^{\frac{4}{3}}}.
\end{align*}
Using these estimates, we obtain the following bounds for $g, f_1,\wt{f}_2$ and $\wt{f}_3$,
\begin{align*}
	\vabs{g(U,\ZInn)} &=
	\vabs{-\frac{3}{2}U^{\frac{2}{3}}\WInn +\frac{1}{6U^{\frac{2}{3}}}
	\frac{\partial_W \JJ}{(1+\JJ)^{\frac{3}{2}}}}
	\leq \frac{C}{\vabs{U}^2},  \\
	\vabs{f_1(U,\ZInn)} &=
	\vabs{\frac{\WInn^2}{2 U^{\frac{1}{3}}}
		-\frac{2}{9 U^{\frac{5}{3}}}  \frac{\JJ}{\sqrt{1+\JJ}(1+\sqrt{1+\JJ})}
		-\frac{1}{6 U^{\frac{2}{3}}} \frac{\partial_U \JJ}
		{(1+\JJ)^{\frac{3}{2}}}
	} \leq \frac{C}{\vabs{U}^{\frac{11}{3}}}, \\
	\vabs{\wt{f}_2(U,\ZInn)} &=
	\vabs{\frac{i}{6 U^{\frac{2}{3}}} \frac{\partial_Y \JJ}{(1+\JJ)^{\frac{3}{2}}}
		-i\XInn g(U,\ZInn)}
	\leq \frac{C}{\vabs{U}^{\frac{4}{3}}}, \\
	\vabs{\wt{f}_3(U,\ZInn)} &=
	\vabs{-\frac{i}{6 U^{\frac{2}{3}}} \frac{\partial_X \JJ}{(1+\JJ)^{\frac{3}{2}}}
		+i\YInn g(U,\ZInn)}
	\leq \frac{C}{\vabs{U}^{\frac{4}{3}}}.
\end{align*}
Analogously, we obtain estimates for the derivatives,
\begin{align*}
	\vabs{\partial_W {g}(U,\ZInn)} 
	&\leq  C \vabs{U}^{\frac{2}{3}}, \quad &
	\vabs{\partial_X{g}(U,\ZInn)} 
	&\leq \frac{C}{\vabs{U}^{\frac{5}{3}}},
	\quad & 
	\vabs{\partial_Y{g}(U,\ZInn)} 
	&\leq \frac{C}{\vabs{U}^{\frac{5}{3}}}, \\
	\vabs{\partial_W {f_1}(U,\ZInn)} 
	&\leq \frac{C}{\vabs{U}^{3}}, 
	\quad &
	\vabs{\partial_X {f_1}(U,\ZInn)} 
	&\leq \frac{C}{\vabs{U}^{\frac{7}{3}}}, 
	\quad &
	\vabs{\partial_Y {f_1}(U,\ZInn)} 
	&\leq \frac{C}{\vabs{U}^{\frac{7}{3}}}, \\
	\vabs{\partial_W {\wt{f}_2}(U,\ZInn)} 
	&\leq \frac{C}{\vabs{U}^{\frac{2}{3}}}, 
	\quad &
	\vabs{\partial_X {\wt{f}_2}(U,\ZInn)} 
	&\leq \frac{C}{\vabs{U}^{2}}, 
	\quad &
	\vabs{\partial_Y{\wt{f}_2}(U,\ZInn)}
	&\leq \frac{C}{\vabs{U}^{2}},\\
	\vabs{\partial_W {\wt{f}_3}(U,\ZInn)} 
	&\leq \frac{C}{\vabs{U}^{\frac{2}{3}}},
	\quad &
	\vabs{\partial_X{\wt{f}_3}(U,\ZInn)}
	&\leq \frac{C}{\vabs{U}^{2}},
	\quad &
	\vabs{\partial_Y {\wt{f}_3}(U,\ZInn)} 
	&\leq \frac{C}{\vabs{U}^{2}}.
\end{align*}
Using these results we  estimate the components of $\RRR$ in~\eqref{proof:operatorRRRInner},
\begin{align*}
	&\normInn{\RRR_1[\ZInn]}_{\frac{11}{3}} = 
	\normInn{ \frac{f_1(\cdot,\ZInn)}{1+g(\cdot,\ZInn)} }_{\frac{11}{3}}
	\leq C, 
	&\normInn{\RRR^{\Inn}_j[\ZInn]}_{\frac{4}{3}} = 
	\normInn{ \frac{\wt{f}_j(\cdot,\ZInn)}{1+g(\cdot,\ZInn)} }_{\frac{4}{3}} \leq C,
\end{align*}
for $j=2,3$. Moreover, 
\begin{align*}
	\normInn{\partial_W{\RRR_1}[\ZInn]}_{3} &= 
	\normInn{ 
		\frac{\partial_W f_1}{1+g} - 
		\frac{f_1 \partial_W g}{(1+g)^2} }_3 
	\leq C, &
	\normInn{\partial_X{\RRR_1}[\ZInn]}_{\frac{7}{3}} &=
	\normInn{ \frac{\partial_X f_1}{1+g} - \frac{f_1 \partial_X g}{(1+g)^2} }_{\frac{7}{3}} 
	\leq C, \\
	\normInn{\partial_W{\RRR_2}[\ZInn]}_{\frac{2}{3}} &=
	\normInn{\frac{\partial_W \wt{f}_2}{1+g} - \frac{\wt{f}_2 \partial_W g}{(1+g)^2} }_{\frac{2}{3}}  
	\leq C, &
	\normInn{\partial_X{\RRR_2}[\ZInn]}_{2} &=
	\normInn{ \frac{\partial_X \wt{f}_2}{1+g} - \frac{\tl{f}_2 \partial_X g}{(1+g)^2} }_2 
	\leq C.
	\end{align*}
Analogously, we obtain the rest of the estimates,
\begin{align*}
	\normInn{\partial_Y{\RRR_1}[\ZInn]}_{\frac{7}{3}}, \,
	\normInn{\partial_Y{\RRR_2}[\ZInn]}_{2}, \,
	\normInn{\partial_W{\RRR_3}[\ZInn]}_{\frac{2}{3}}, \,
	\normInn{\partial_X{\RRR_3}[\ZInn]}_{2}, \,
	\normInn{\partial_Y{\RRR^{\Inn}_3}[\ZInn]}_{2} \leq C.
	\end{align*}
\qed

\subsubsection{Proof of Lemmas~\ref{lemma:operatorEEpotential} and \ref{lemma:operatorEEexponential}}
\label{subsubsection:innerTechnicalProofsB}

Let us introduce, for $\rhoInn>0$ and $\al\geq 0$ , the following  linear operators,  
\begin{equation}\label{def:operatorBBB}
\begin{split}
\BB_{\al}[\Psi](U)= e^{i\al U} \int_{-i\infty}^U e^{-i\al S} \Psi(S) dS, 
	\quad
\wt{\BB}[\Psi](U)= e^{-i U} \int_{-i\rhoInn}^U e^{i S} \Psi(S) dS.
\end{split}
\end{equation}

The following lemma is proven in~\cite{Bal06}.

\begin{lemma}\label{lemma:operatorBB}
	Fix $\eta>1$, $\nu>0$, $\al>0$ and $\rhoInn>1$. Then,  the following operators are well defined
	\begin{align*}
		\BB_{0}: \YcalInn_{\eta} \to \YcalInn_{\eta-1}, \qquad
		\BB_{\al}: \YcalInn_{\nu} \to \YcalInn_{\nu}, \qquad
		\wt{\BB}: \YcalInn_{\nu} \to \YcalInn_{\nu}
	\end{align*}
	and there exists a constant $C>0$  such that 
	\begin{align*}
		\normInnDiff{\BB_{0}[\Psi]}_{\eta-1} \leq C \normInnDiff{\Psi}_{\eta}, \hh
		\normInnDiff{\BB_{\al}[\Psi]}_{\nu} \leq C \normInnDiff{\Psi}_{\nu}, \hh
		\normInnDiffSmall{\wt{\BB}[\Psi]}_{\nu} \leq C \normInnDiff{\Psi}_{\nu}.
	\end{align*}
\end{lemma}

It is clear that 
$\II_1[\Psi] = \BB_{0}[\langle R_{1}, \Psi\rangle]$,
$\II_2[\Psi] = \BB_{1}[\langle R_{2}, \Psi\rangle]$ and
$\II_3[\Psi] = \wt{\BB}[\langle R_{3}, \Psi\rangle]$ (see \eqref{def:operatorRDiff} and \eqref{def:operatorEEE}).
Thus, we  use this lemma to prove Lemmas~\ref{lemma:operatorEEpotential} and \ref{lemma:operatorEEexponential}.

\begin{proof}[Proof of Lemma~\ref{lemma:operatorEEpotential}] By the definition of the operator $R$ and Lemma~\ref{lemma:boundsRRR}, we have that
\begin{equation}\label{eq:boundsRDifference}
\begin{aligned}
\normInnDiff{R_{1,1}}_3 &\leq C, 
\quad &
\normInnDiff{R_{1,2}}_{\frac{7}{3}} &\leq C , 
\quad &
\normInnDiff{R_{1,3}}_{\frac{7}{3}} &\leq C , \\
\normInnDiff{R_{j,1}}_{\frac{2}{3}} &\leq C, 
\quad &
\normInnDiff{R_{j,2}}_{2} &\leq C , 
\quad &
\normInnDiff{R_{j,3}}_{2} &\leq C , 
\quad \text{ for } j=2,3.
\end{aligned}
\end{equation}
%
Then, by Lemma~\ref{lemma:operatorBB}, for  $\rhoInn$ big enough and  $\Psi \in \YcalInn_{\frac{8}{3}} \times \YcalInn_{\frac{4}{3}} \times \YcalInn_{\frac{4}{3}}$,  we have that
\begin{equation*}
\begin{split}
\normInnDiff{\II_1[\Psi]}_{\frac{8}{3}} &= 
\normInnDiff{\BB_{0}
[\langle R_{1}, \Psi\rangle]}_{\frac{8}{3}}
\leq C \normInnDiff{\langle R_{1}, \Psi\rangle}_{\frac{11}{3}}\\
 &\leq C \paren{
	\normInnDiff{R_{1,1}}_1 \normInnDiff{\Psi_1}_{\frac{8}{3}} +
	\normInnDiff{R_{1,2}}_{\frac{7}{3}} \normInnDiff{\Psi_2}_{\frac{4}{3}} +
	\normInnDiff{R_{1,3}}_{\frac{7}{3}} \normInnDiff{\Psi_3}_{\frac{4}{3}} }\\
	&\leq 
	C \paren{
	\frac{1}{\rhoInn^2} \normInnDiff{\Psi_1}_{\frac{8}{3}} +
	\normInnDiff{\Psi_2}_{\frac{4}{3}} +
	\normInnDiff{\Psi_3}_{\frac{4}{3}}},
	\end{split}
\end{equation*}
which gives the first estimate of the lemma.
Analogously, by Lemma~\ref{lemma:operatorBB},
\begin{align*}
\normInnDiff{\II_2[\Psi]}_{\frac{4}{3}} &= 
\normInnDiff{\BB_{1}[\langle R_{2}, \Psi\rangle]}_{\frac{4}{3}}
\leq C \normInnDiff{
\langle R_{2}, \Psi\rangle}_{\frac{4}{3}},\\
\normInnDiff{\II_3[\Psi]}_{\frac{4}{3}} &=
\normInnDiffSmall{\wt{\BB}[\langle R_{3}, \Psi\rangle]}_{\frac{4}{3}}
\leq C \normInnDiff{\langle R_{3}, \Psi\rangle}_{\frac{4}{3}},
\end{align*}
and applying~\eqref{eq:boundsRDifference},
for $j=2,3$, we have
\begin{align*}
\normInnDiff{\langle R_{j}, \Psi\rangle }_{\frac{4}{3}} &\leq
\normInnDiff{R_{j,1}}_{-\frac{4}{3}} \normInnDiff{\Psi_1}_{\frac{8}{3}} 
+
\normInnDiff{R_{j,2}}_{0} \normInnDiff{\Psi_2}_{\frac{4}{3}} 
+
\normInnDiff{R_{j,3}}_{0} \normInnDiff{\Psi_3}_{\frac{4}{3}} \\ 
&\leq \frac{C}{\rhoInn^2} \paren{
	\normInnDiff{\Psi_1}_{\frac{8}{3}} +
	\normInnDiff{\Psi_2}_{\frac{4}{3}} +
	\normInnDiff{\Psi_3}_{\frac{4}{3}} },
\end{align*}
which gives the second and third estimates of the lemma.
\end{proof}


\begin{proof}[Proof of Lemma~\ref{lemma:operatorEEexponential}]
Let us consider $\Psi \in \ZcalInn_{\frac{4}{3},0,0}$ and define
\[
\Phi(U)=e^{i U}R(U)\Psi(U),
\]
in such a way that, by the definition of the operator $\BB_{\al}$ in~\eqref{def:operatorBBB}, 
\begin{equation}\label{proof:integralsPsiToPhi}
\begin{split}
&e^{i U} \II_1[\Psi](U) 
= e^{i U} \int_{-i \infty}^{U} 
e^{-i S} \Phi_1(S) d S 
= \BB_{1}[\Phi_1], \\
&e^{i U} \II_2[\Psi](U)
= e^{i2U} \int_{-i\infty}^{U}  
e^{-i2S} 
\Phi_2(S) d S
= \BB_{2}[\Phi_2]  , \\
&e^{i U} \II_3[\Psi](U)	= 
\int_{-i\rhoInn}^{U} \Phi_3(S) d S.
\end{split}
\end{equation}
Since $e^{iU} \Psi \in \YcalInn_{\frac{4}{3}}\times \YcalInn_{0} \times \YcalInn_{0}$, by the estimates in~\eqref{eq:boundsRDifference}, we have that, for $j=2,3$,
\begin{equation}\label{proof:estimatesh1h2h3}
\begin{split}
\normInnDiff{\Phi_1}_{\frac{7}{3}}
&\leq \normInnDiff{R_{1,1}}_{1} 
\normInnDiff{e^{iU} \Psi_1}_{\frac{4}{3}} +
\sum_{k=2,3}
\normInnDiff{R_{1,k}}_{\frac{7}{3}} \normInnDiff{e^{iU}\Psi_k}_{0} 
\leq C \normInnDiffExp{\Psi}_{\frac{4}{3},0,0}, \\
\normInnDiff{\Phi_j}_{2}
&\leq \normInnDiff{R_{j,1}}_{\frac{2}{3}} \normInnDiff{e^{iU}\Psi_1}_{\frac{4}{3}} +
\sum_{k=2,3}
\normInnDiff{R_{j,k}}_{2} \normInnDiff{e^{iU}\Psi_k}_{0}
\leq C \normInnDiffExp{\Psi}_{\frac{4}{3},0,0}.
\end{split}
\end{equation}
Therefore, Lemma~\ref{lemma:operatorBB} and  \eqref{proof:integralsPsiToPhi} imply
\begin{equation*}
\begin{split}
\normInnDiffExp{\II_1[\Psi]}_{\frac{7}{3}}
&= \normInnDiff{\BB_{1}[\Phi_1]}_{\frac{7}{3}}
\leq C \normInnDiff{\Phi_1}_{\frac{7}{3}}
\leq C \normInnDiffExp{\Psi}_{\frac{4}{3},0,0}
, \\
\normInnDiffExp{\II_2[\Psi]}_{2}
&= \normInnDiff{\BB_{2}[\Phi_2]}_{2}
\leq C \normInnDiff{\Phi_2}_{2}
\leq C\normInnDiffExp{\Psi}_{\frac{4}{3},0,0}.
\end{split}
\end{equation*}
Now, we deal with operator $\II_3$. 
Notice that, 
by the definition of the operator $\BB_{\al}$ in~\eqref{def:operatorBBB} and \eqref{proof:integralsPsiToPhi}, we have that
\begin{equation*}
e^{i U} \II_3[\Psi](U) =
\int_{-i\rhoInn}^{-i\infty} \Phi_3(S) d S +
\int_{-i\infty}^{U} \Phi_3(S) d S 
= -\BB_0[\Phi_3](-i\kappa)
+ \BB_0[\Phi_3](U).
\end{equation*}
Then, by Lemma~\ref{lemma:operatorBB} and using the estimates~\eqref{proof:estimatesh1h2h3}, we obtain
\begin{align*}
\normInnDiffExp{\II_3[\Psi]}_{0} 
&\leq \vabs{\BB_0[\Phi_3](-i\kappa)} + \normInnDiff{\BB_0[\Phi_3]}_0
\leq 2 \normInnDiff{\BB_0[\Phi_3]}_0 \\
&\leq \frac{C}{\kappa} \normInnDiff{\BB_0[\Phi_3]}_1
\leq \frac{C}{\rhoInn} \normInnDiff{\Phi_3}_2
\leq \frac{C}{\rhoInn} \normInnDiffExp{\Psi}_{\frac{4}{3},0,0}.
\end{align*}
Finally, taking $\wCInn(\rhoInn)=-\BB_0[\Phi_3](-i\kappa)$, we conclude
\begin{equation*}
\normInnDiffExp{\II_3[\Psi](U)-e^{-iU}\wCInn(\rhoInn)}_1
= \normInnDiff{\BB_0[\Phi_3]}_1
\leq C \normInnDiff{\Phi_3}_2 
\leq C \normInnDiffExp{\Psi}_{\frac{4}{3},0,0}.
\end{equation*}
\end{proof}

\section*{Acknowledgments}
\addcontentsline{toc}{section}{Acknowledgments}
I. Baldom\'a has been partly supported by the Spanish MINECO--FEDER Grant PGC2018 -- 098676 -- B -- 100 (AEI/FEDER/UE) and the Catalan grant 2017SGR1049. 

M. Giralt and M. Guardia have received funding from the European Research Council (ERC) under the European Union's Horizon 2020 research and innovation programme (grant agreement No. 757802). 

M. Guardia is also supported by the Catalan Institution for Research and Advanced Studies via an ICREA Academia Prize 2019.


\addcontentsline{toc}{section}{Bibliography}

\begingroup
\footnotesize
\setlength\bibitemsep{2pt}
\setlength\biblabelsep{4pt}
\printbibliography
\endgroup


\end{document}